\documentclass[twoside,11pt]{article}

\usepackage[hyperref, preprint]{jmlr2e}
\usepackage{mathtools, extarrows}
\usepackage{enumitem}
\usepackage{subcaption}
\usepackage{algorithm}
\usepackage{algpseudocode}
\usepackage{nicefrac}
\usepackage{microtype}
\usepackage{xcolor}

\usepackage{stackengine}

\stackMath

\newcommand{\R}{\mathbf{R}}
\newcommand{\C}{\mathbf{C}}
\newcommand{\X}{\mathcal{X}}
\newcommand{\Z}{\mathcal{Z}}
\newcommand{\N}{\mathbb{N}}
\newcommand{\RR}{{\mathbf R}}
\newcommand{\NN}{{\mathbb N}}
\newcommand{\SSS}{{\mathbf S}}

\newcommand{\FF}{{\mathbb F}}

\newcommand{\realp}{\operatorname{Re}}
\newcommand{\cA}{\mathcal{A}}
\newcommand{\cB}{\mathcal{B}}
\newcommand{\cD}{\mathcal{D}}

\newcommand{\cF}{\mathcal{F}}
\newcommand{\cG}{\mathcal{G}}

\newcommand{\cI}{\mathcal{I}}
\newcommand{\cK}{\mathcal{K}}
\newcommand{\cL}{\mathcal{L}}

\newcommand{\cS}{\mathcal{S}}
\newcommand{\cX}{\mathcal{X}}

\newcommand{\cZ}{\mathcal{Z}}
\newcommand{\cU}{\mathcal{U}}
\newcommand{\cV}{\mathcal{V}}
\newcommand{\cW}{\mathcal{W}}
\newcommand{\xs}{x^\star}

\newcommand{\range}{\operatorname{range}}

\newcommand{\Expect}{\operatorname{\mathbb E}}
\newcommand{\EE}{\operatorname{\mathbb E}}

\newcommand{\PP}{\operatorname{\mathbb P}}
\newcommand{\D}{\mathcal{D}}

\newcommand{\argmin}{\operatornamewithlimits{arg\,min}}

\newcommand{\interior}{\operatorname{int}}
\newcommand{\diam}{\operatorname{diam}}
\newcommand{\proj}{{\rm proj}}

\def\shortdisplay{\setlength{\abovedisplayskip}{5pt}\setlength{\belowdisplayskip}{5pt}\setlength{\abovedisplayshortskip}{2pt}\setlength{\belowdisplayshortskip}{2pt}}

\let\oldselectfont\selectfont
\def\selectfont{\oldselectfont\shortdisplay}

\usepackage{lastpage}
\jmlrheading{25}{2024}{1-\pageref{LastPage}}{7/22; Revised
12/23}{2/24}{22-0832}{Joshua Cutler, Mateo D\'iaz, and Dmitriy Drusvyatskiy}

\ShortHeadings{Stochastic Approximation with Decision-Dependent Distributions}{Cutler, D\'iaz, and Drusvyatskiy}
\firstpageno{1}

\begin{document}

\tolerance 1414
\hbadness 1414
\emergencystretch 1.5em
\hfuzz 0.3pt
\widowpenalty=10000
\vfuzz \hfuzz
\raggedbottom

\title{Stochastic Approximation with Decision-Dependent Distributions: Asymptotic Normality and Optimality}

\author{\name Joshua\ Cutler \email jocutler@uw.edu\\
	\addr Department of Mathematics\\
	University of Washington\\
	Seattle, WA 98195, USA
	\AND
	\name Mateo\ D\'iaz \email mateodd@jhu.edu\\
	\addr Department of Applied Mathematics and Statistics\\
	Johns Hopkins University\\
	Baltimore, MD 21218, USA
	\AND
	\name Dmitriy\ Drusvyatskiy \email ddruv@uw.edu\\
	\addr Department of Mathematics\\
	University of Washington\\
	Seattle, WA 98195, USA
	\AND}

\editor{Francis Bach}

\maketitle

\begin{abstract}We analyze a stochastic approximation algorithm for decision-dependent problems, wherein the data distribution used by the algorithm evolves along the iterate sequence. The primary examples of such problems appear in performative prediction and its multiplayer extensions. We show that under mild assumptions, the deviation between the average iterate of the algorithm and the solution is asymptotically normal, with a covariance that clearly decouples the effects of the gradient noise and the distributional shift. Moreover, building on the work of H\'ajek and Le Cam, we show that the asymptotic performance of the algorithm with averaging is locally minimax optimal.
\end{abstract}

\section{Introduction}
\label{sec:intro}

The primary role of stochastic optimization in data science is to find a learning rule (e.g., a classifier) from a limited data sample which enables accurate prediction on unseen data. Classical theory crucially relies on the assumption that both the observed data and the unseen data are generated by the same distribution. Recent literature on strategic classification \citep{hardt2016strategic} and performative prediction \citep{perdomo2020performative}, however, has highlighted a variety of contemporary settings where this assumption is grossly violated. One common reason is that the data seen by a learning system may depend on or react to a deployed learning rule. For example, members of the population may alter their features in response to a deployed classifier in order to increase their likelihood of being positively labeled---a phenomenon called gaming. Even when
the population is agnostic to the learning rule, the decisions made by the learning system (e.g., loan
approval) may inadvertently alter the profile of the population (e.g., credit score). The goal of the
learning system therefore is to find a classifier that generalizes well under the response distribution. The situation may be further compounded by a population that reacts to multiple competing learners simultaneously \citep{narang2022multiplayer,wood2022stochastic,piliouras2022multi}.

In this work, we model decision-dependent problems using variational inequalities.
Namely, let $G(x,z)$ be a map that  depends on the decision $x$ and data $z$, and let the set $\mathcal{X}$ of feasible decisions be closed and convex. A variety of classical learning problems can be posed as solving the variational inequality
\begin{equation}\label{eqn:VI0}
	0\in \mathop{\EE}_{z\sim \mathcal{P}}G(x,z)+N_{\mathcal{X}}(x),\tag*{VI($\mathcal{P}$)}
\end{equation}
where $\mathcal{P}$ is some fixed distribution and $N_{\mathcal{X}}(x)=\{v\in\R^d \mid \langle v, y-x \rangle \leq 0 ~\text{for all}~ y\in\cX\}$ is the normal cone to $\mathcal{X}$ at $x\in\cX$. Two examples are worth keeping in mind: $(i)$ standard problems of supervised learning amount to $G(x,z)=\nabla_x \ell(x,z)$ being the gradient of some loss function to be minimized over $\cX$, and $(ii)$ stochastic games correspond to $G(x,z)$ being a stacked gradient of the players' individual losses. In both of these examples, \ref{eqn:VI0} encodes the standard first-order optimality conditions. The benefit of variational inequalities is that they yield a single framework for analyzing a wide range of learning problems, notably in optimization and game theory. We refer the interested reader to \citet{kinderlehrer2000introduction} and \citet{dontchev2009implicit} for a historical perspective and further details on the use of variational inequalities in applications.

Following the recent literature on performative prediction \citep{hardt2016strategic,perdomo2020performative,narang2022multiplayer}, we will be interested in settings where the distribution $\mathcal{P}$ is not fixed but rather varies with $x$. With this in mind, let $\mathcal{D}(x)$ be a family of distributions indexed by  $x\in \mathcal{X}$. The interpretation is that $\mathcal{D}(x)$ is the response of the population to a newly deployed learning rule $x$. We posit that the goal of a learning system is to find a point $x^{\star}$ so that
$x=x^{\star}$ solves the variational inequality VI($\mathcal{D}(x^{\star})$), or equivalently:
$$0\in \mathop{\EE}_{z\sim \mathcal{D}(x^{\star})} G(x^{\star},z)+N_{\mathcal{X}}(x^{\star}).$$
We will say that such points $x^{\star}$ are at {\em equilibrium}.
In words, a learning system that deploys an equilibrium point $x^{\star}$ has no incentive to deviate from $x^{\star}$ based only on the solution of the variational inequality VI($\mathcal{D}(x^{\star})$) induced by the response distribution $\mathcal{D}(x^{\star})$. The setting of performative prediction \citep{perdomo2020performative} corresponds to the  choice $G(x,z)=\nabla_x \ell(x,z)$ for some loss function $\ell$.\footnote{In the language of \citet{perdomo2020performative}, equilibria coincide with performatively stable points.} More generally, decision-dependent games, proposed by  \citet{narang2022multiplayer}, \citet{piliouras2022multi}, and \citet{wood2022stochastic}, correspond to the choice $G(x,z)=\big(\nabla_1 \ell_1(x,z),\ldots, \nabla_k \ell_k(x,z)\big)$ where $\nabla_i \ell_i(x,z)$ is the gradient of the $i$'th player's loss with respect to their decision $x_i$ and $\mathcal{D}(x)=\mathcal{D}_1(x)\times \cdots \times\mathcal{D}_k(x)$ is a product distribution.  The specifics of these two examples will not affect our results, and therefore we work with general maps $G(x,z)$.

Following the prevalent viewpoint in machine learning, we suppose that the only access to the data distributions $\mathcal{D}(x)$ is by drawing samples $z\sim \mathcal{D}(x)$. With this in mind, a natural algorithm for finding an equilibrium point $x^{\star}$ is the {\em stochastic forward-backward algorithm}:
\begin{equation}\label{eqn:VI_iteration}
	\begin{aligned}
		 & \textrm{Sample}~z_t\sim \mathcal{D}(x_t)                         \\
		 & \textrm{Set}~x_{t+1}=\proj_{\cX}\big(x_t-\eta_t G(x_t,z_t)\big),
	\end{aligned}\tag*{SFB}
\end{equation}
where $\proj_{\cX}$ is the nearest-point projection onto $\cX$. Specializing to performative prediction \citep{mendler2020stochastic} and its multiplayer extension \citep{narang2022multiplayer}, this algorithm reduces to a basic projected stochastic gradient iteration. The contribution of our paper can be informally summarized as follows.

\begin{center}
	\fbox{We show that averaged \ref{eqn:VI_iteration} is asymptotically optimal for finding equilibrium points.}
\end{center}
In particular, our results imply asymptotic optimality of the basic stochastic gradient methods for both single player and multiplayer performative prediction.

\subsection{Summary of Main Results}

Arguing optimality of an algorithm is a two-step process: $(i)$ estimate the performance of the specific algorithm and $(ii)$ derive a matching lower bound that is valid among all relevant estimation procedures. Beginning with the former, we build on the seminal work of \citet{polyak92}, wherein a central limit theorem is established for stochastic approximation algorithms for solving smooth equations. Letting $\bar x_{t} = \frac{1}{t}\sum_{i=1}^{t}x_{i}$ denote the running average of the SFB iterates, we show that the deviation $\sqrt{t}(\bar x_{t}-\xs)$ is asymptotically normal with an appealingly simple covariance. See Figure~\ref{fig:ex1} for an illustration.\footnote{Visit \url{https://github.com/mateodd25/Asymptotic-normality-in-performative-prediction} for code.}

\begin{theorem}[Asymptotic normality, informal; see Theorem~\ref{thm:anperf}]\label{thm:main_informal}
	Suppose that $G(\cdot,z)$ is $\alpha$-strongly monotone and Lipschitz continuous on $\cX$, $G(x,\cdot)$ is $\beta$-Lipschitz continuous on $\cZ$, and the distribution map $\mathcal{D}(\cdot)$ is $\gamma$-Lipschitz continuous on $\cX$ with respect to the Wasserstein-1 distance. Suppose moreover that $x^{\star}$ lies in the interior of $\mathcal{X}$ and $\eta_{t} \propto t^{-\nu}$ for some $\nu\in\big(\tfrac{1}{2}, 1\big)$. Then in the regime $\frac{\gamma\beta}{\alpha}<1$, the  \ref{eqn:VI_iteration} iterates $x_t$ converge to the equilibrium point $x^\star$ almost surely, and the averaged SFB iterates $\bar x_{t} = \frac{1}{t}\sum_{i=1}^{t}x_{i}$ satisfy
	\begin{equation*}
		\sqrt{t}(\bar{x}_t - x^{\star})\leadsto\mathsf{N}\big(0,W^{-1} \Sigma W^{-\top}\big),
	\end{equation*}
	where
	\begin{align*}
		\Sigma = \underset{z\sim\D(\xs)}{\Expect}\big[G(\xs, z)G(\xs, z)^\top\big]\quad \text{and}\quad W=\underbrace{\underset{z\sim\D(x^{\star})}{\Expect}[\nabla_x G(x^{\star},z)]}_\text{static}+ \underbrace{\frac{d}{dy} \underset{z\sim\D(y)}{\Expect}[G(\xs, z)]\Big|_{y=x^{\star}}}_\text{dynamic}\!.
	\end{align*}
\end{theorem}
A few comments are in order. First, the conditions on the data of the problem reduce to the standard assumptions in performative prediction \citep{perdomo2020performative} when $G(x, z) = \nabla_{x} \ell(x,z)$. In particular, $G(\cdot, z)$ being $\alpha$-strongly monotone and Lipschitz is then equivalent to the function $\ell(\cdot,z)$ being $\alpha$-strongly convex and smooth with Lipschitz continuous gradient. The strong monotonicity requirement can be loosened to hold only in expectation; see Theorem~\ref{thm:anperf} for the formal statement. Second, the regime $\frac{\gamma\beta}{\alpha}<1$ is, in essence, optimal because otherwise equilibrium points may even fail to exist. Third, the effect of the distributional shift on the asymptotic covariance is entirely captured by the second ``dynamic'' term in $W$. Indeed, when this term is absent, the product $W^{-1}\Sigma W^{-\top}$ is precisely the asymptotic covariance of the stochastic forward-backward algorithm applied to the static problem  VI($\mathcal{D}(x^{\star})$)  at equilibrium.\footnote{Of course, this analogy is entirely conceptual, since $\mathcal{D}(x^{\star})$ is unknown a priori.} The proof of Theorem~\ref{thm:main_informal} follows by interpreting \ref{eqn:VI_iteration} as a stochastic approximation algorithm for finding the zero of the nonlinear map $R(x)=\EE_{z\sim\D(x)}[G(x,z)]$ and then applying a variation of the classical asymptotic normality result of \citet[Theorem~2]{polyak92}.

\begin{figure}[t]
	\centering
	\begin{subfigure}[b]{0.3\textwidth}
		\centering
		\includegraphics[width=0.95\textwidth]{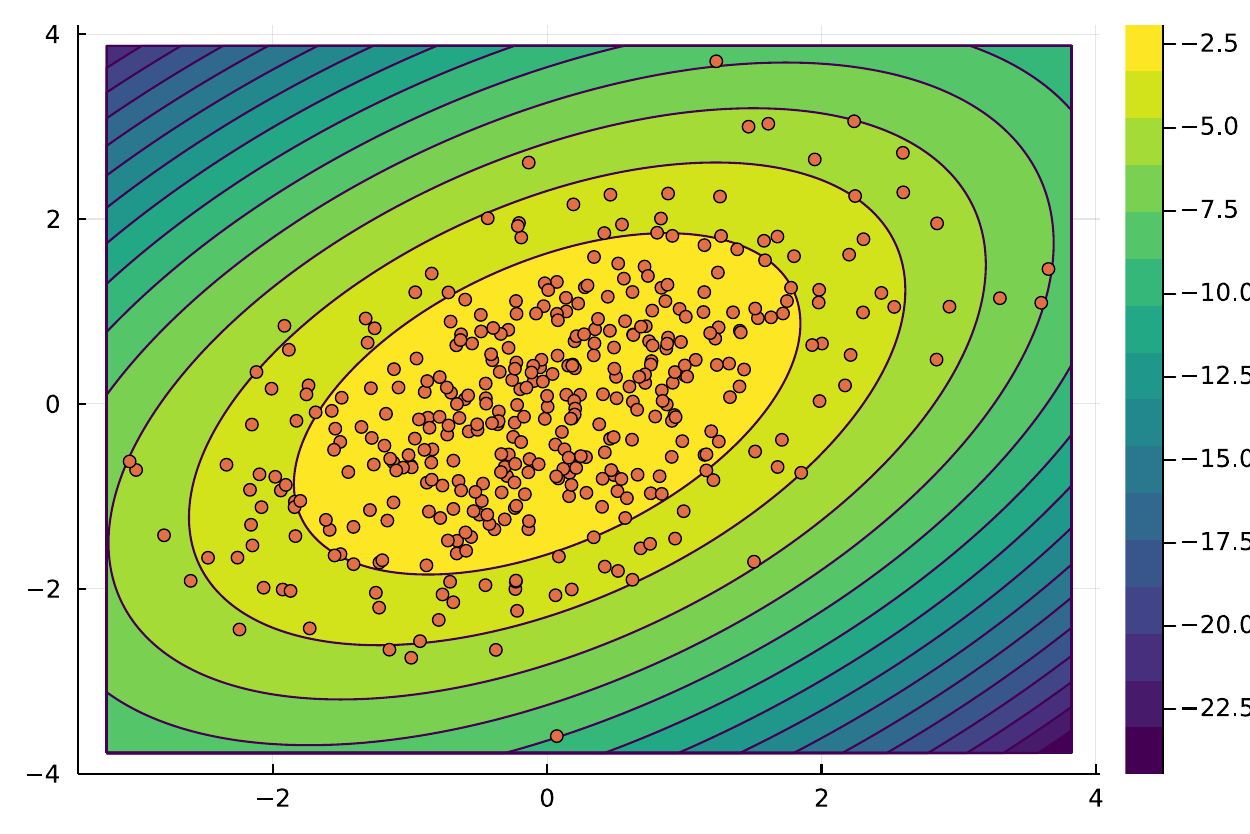}
		\includegraphics[width=0.95\textwidth]{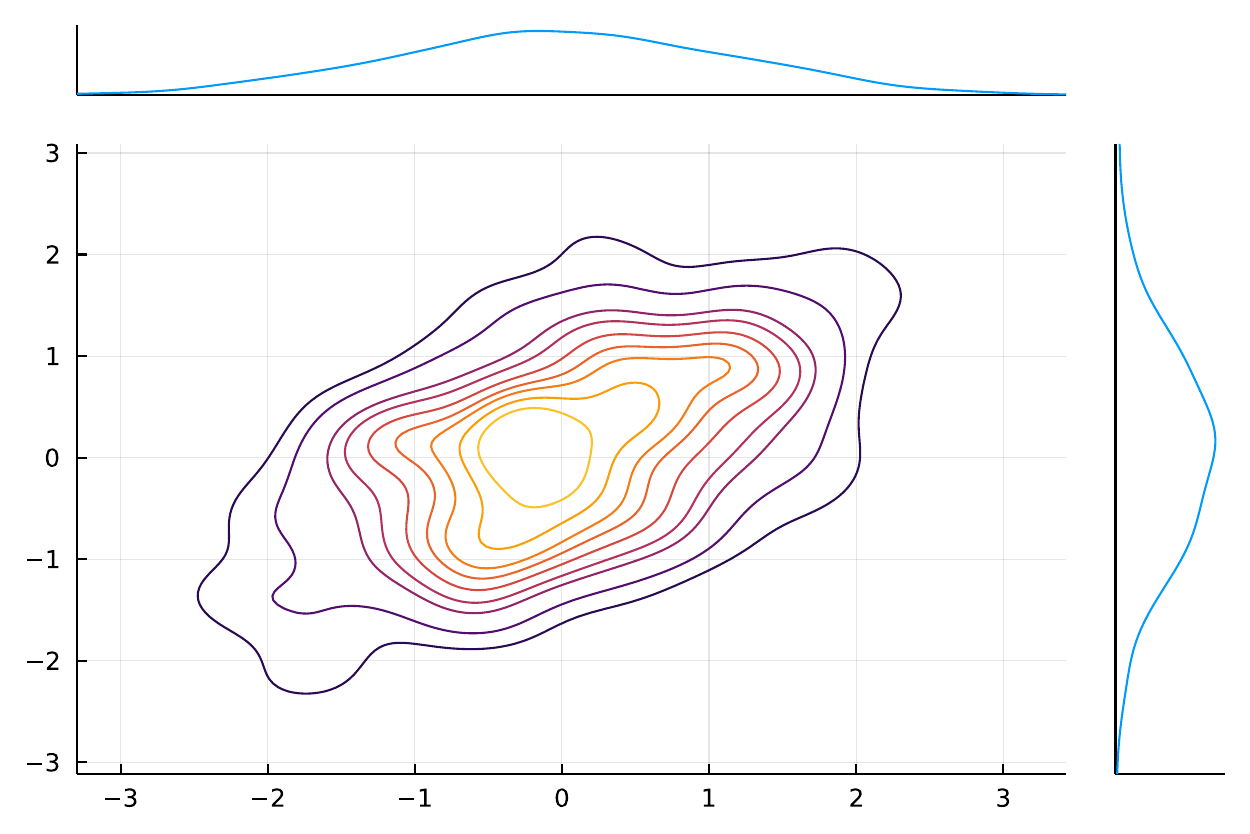}
		\includegraphics[width=0.95\textwidth]{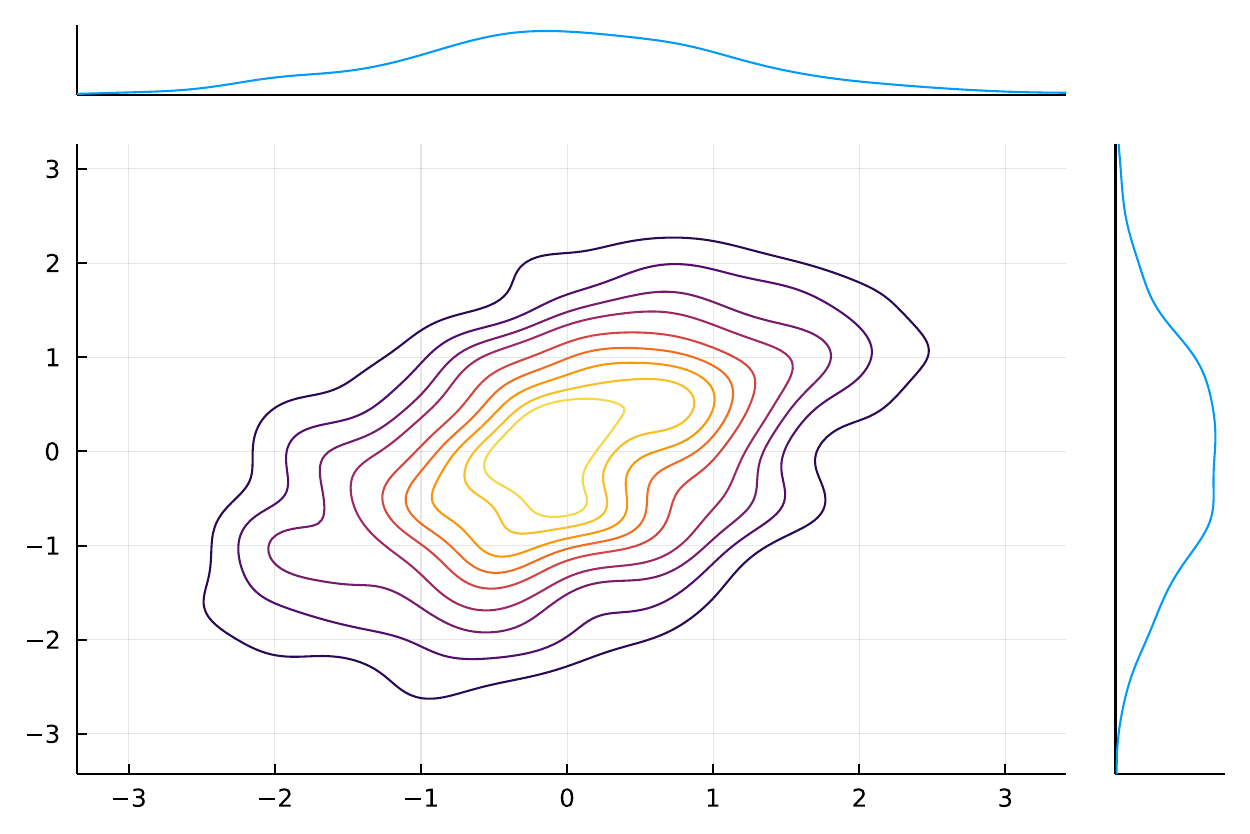}
		\caption{$\rho = 0.25$}\label{fig:quadratic}
	\end{subfigure}
	\hfill
	\begin{subfigure}[b]{0.3\textwidth}
		\centering
		\includegraphics[width=0.95\textwidth]{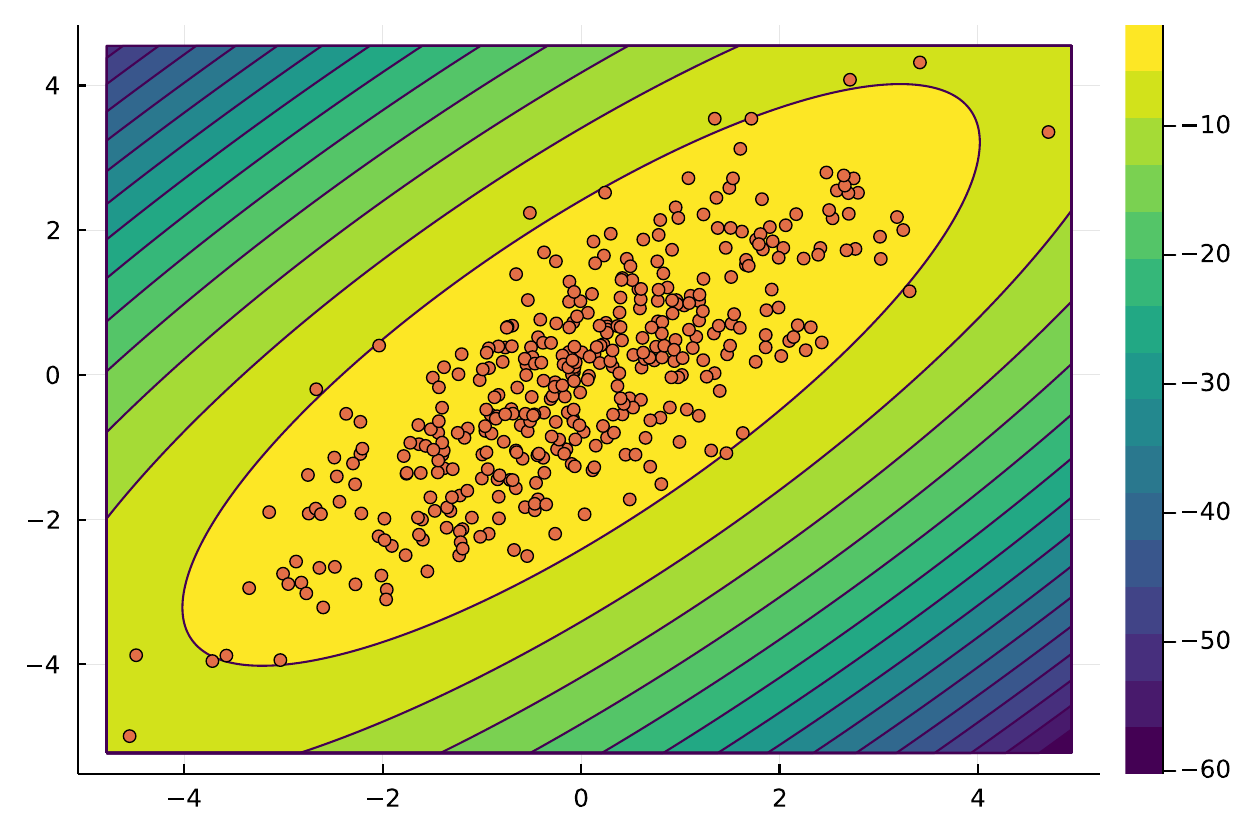}
		\includegraphics[width=0.95\textwidth]{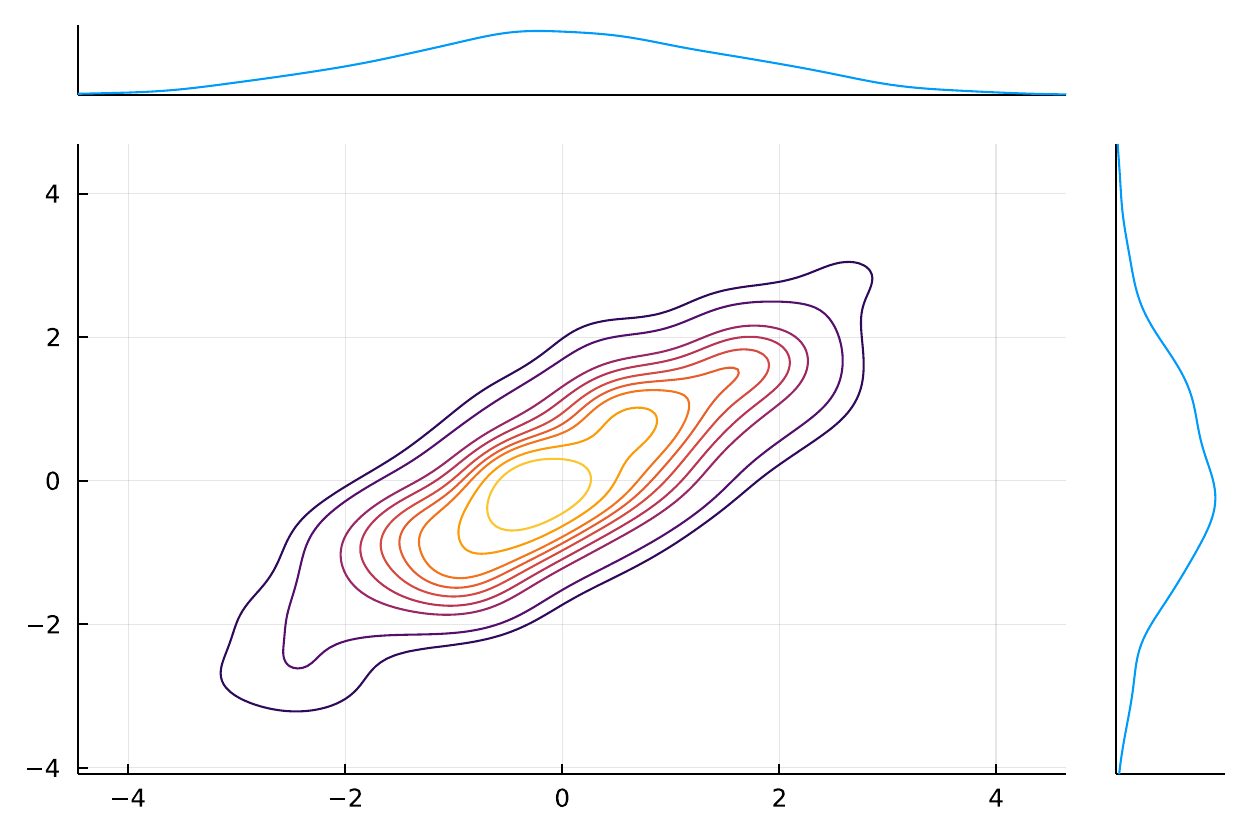}
		\includegraphics[width=0.95\textwidth]{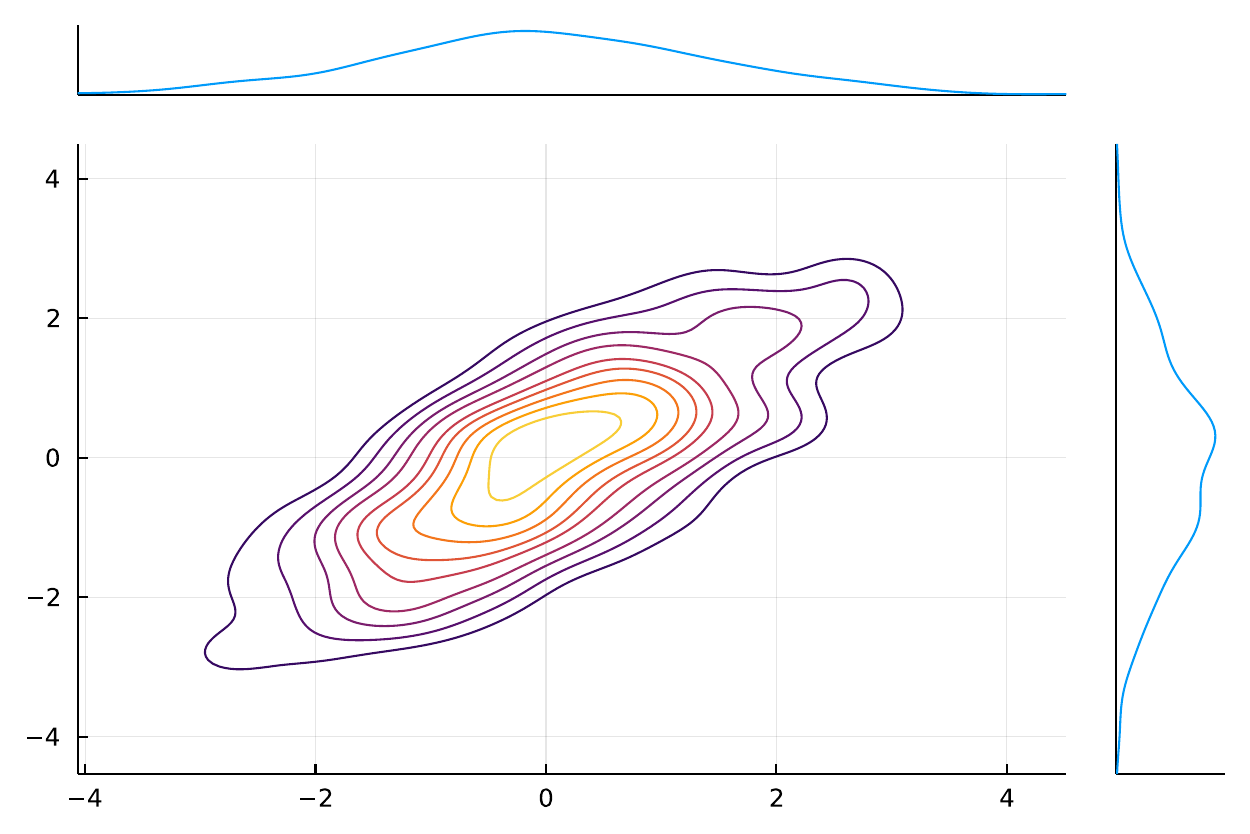}
		\caption{$\rho = 0.5$}\label{fig:quadratic}
	\end{subfigure}
	\hfill
	\begin{subfigure}[b]{0.3\textwidth}
		\centering
		\includegraphics[width=0.95\textwidth]{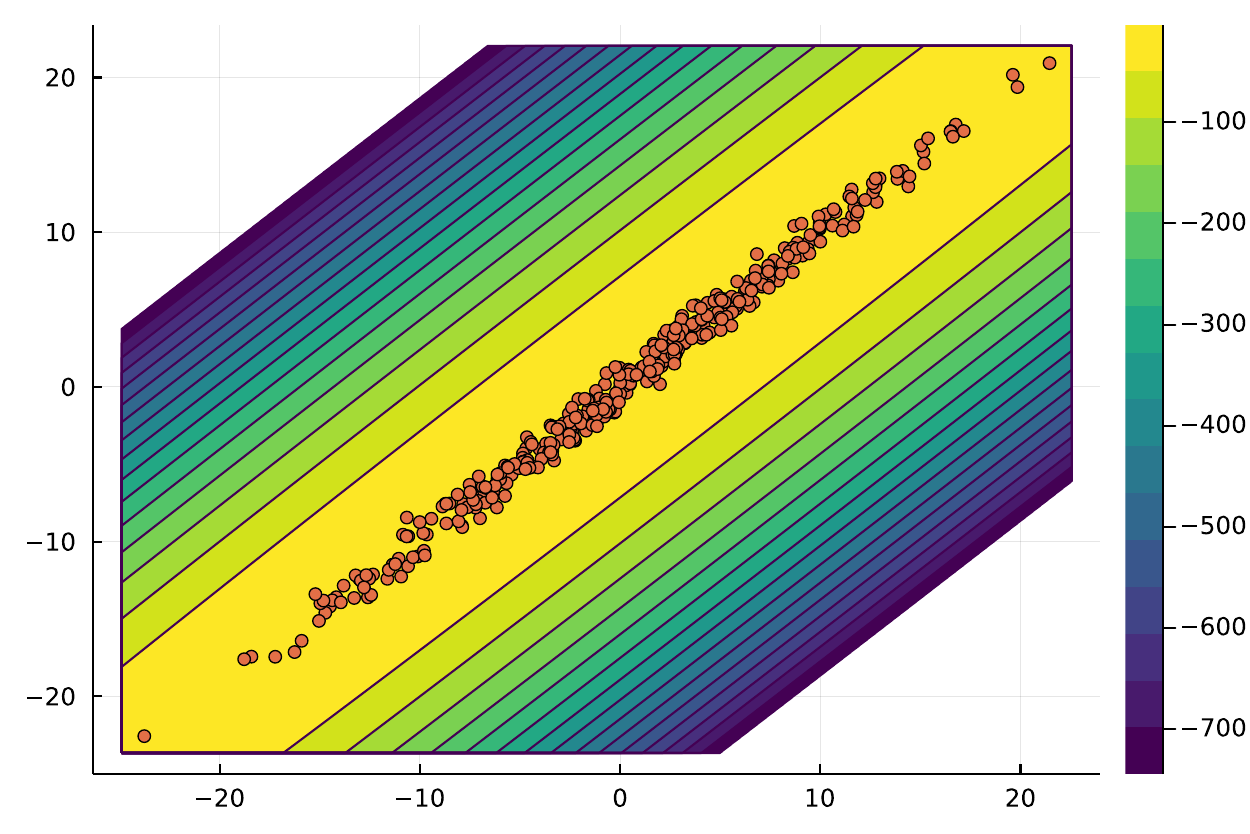}
		\includegraphics[width=0.95\textwidth]{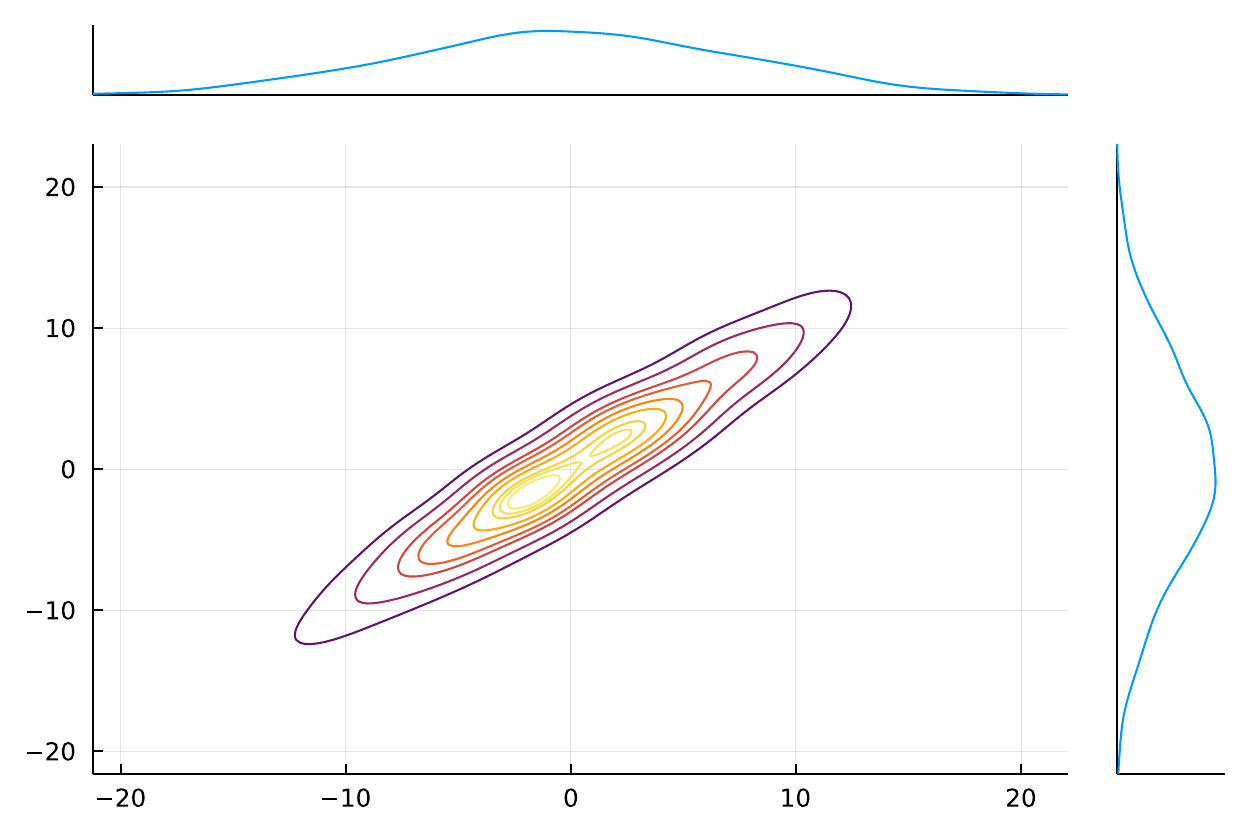}
		\includegraphics[width=0.95\textwidth]{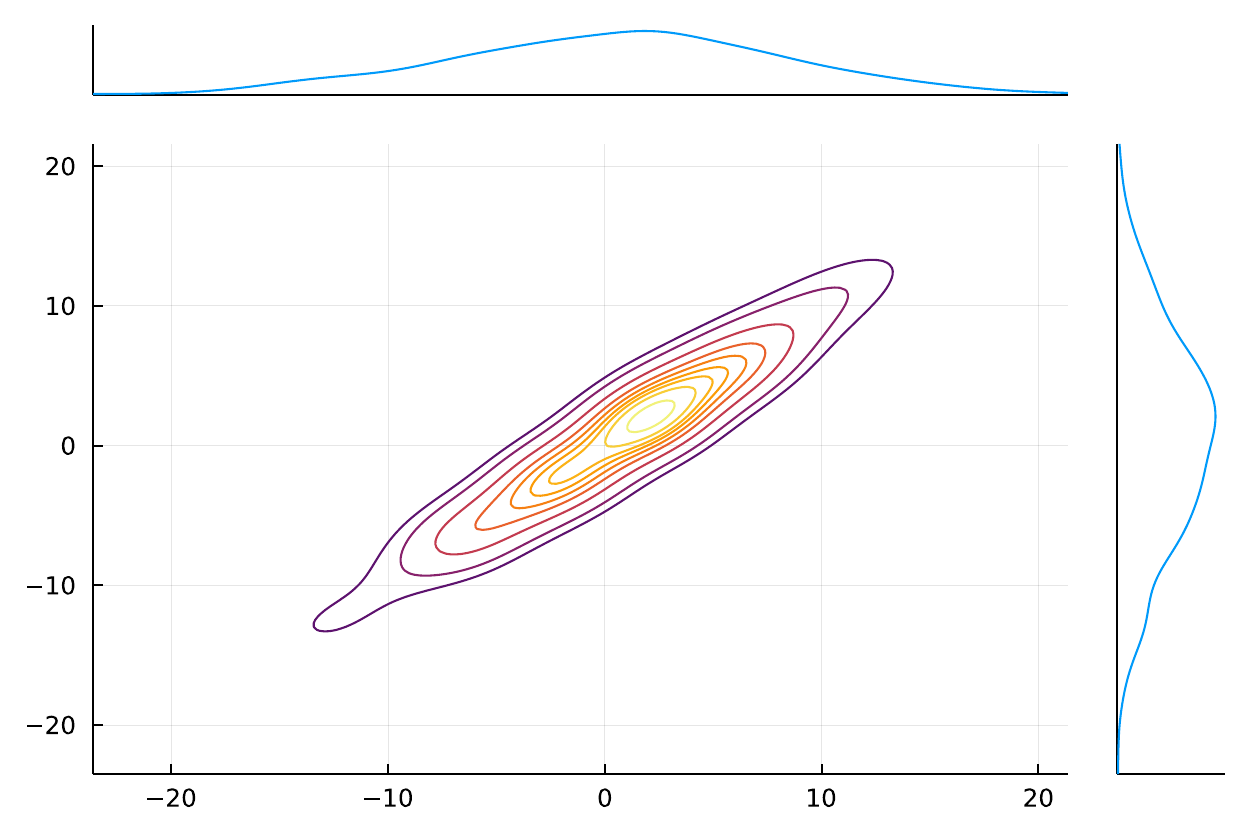}
		\caption{$\rho = 0.9$}\label{fig:landscape}
	\end{subfigure}
	\caption{Consider the problem corresponding to $G(x, z) = \nabla_x \ell (x, z)$ with $\ell(x, z) = \frac{1}{2} \|x-z\|^{2}$ and $\cD(x_{1}, x_{2}) = \mathsf{N}(\rho (x_{2}, x_{1}), I_2)$. A simple computation shows $\Sigma = I_{2}$ and
	$W = [1,-\rho; -\rho ,1].$
	As $\rho$ approaches one, $W^{-1}$ becomes ill conditioned. We run algorithm~\ref{eqn:VI_iteration} $400$ times using $\eta_{t} = t^{-3/4}$ for $10^6$ iterations.
	The first row depicts the resulting average iterates laid over the confidence regions (plotted in logarithmic scale) corresponding to the asymptotic normal distribution. The next two rows depict kernel density estimates from the asymptotic normal distribution (top) and the deviation $\sqrt{k} (\bar x_{k} - \xs)$ (bottom). }\label{fig:ex1}
\end{figure}

A reasonable question to ask is whether there exists an algorithm with better asymptotic guarantees than those of the stochastic forward-backward algorithm (with averaging). We will show that in a strong sense, the answer is no; averaged~\ref{eqn:VI_iteration} is asymptotically optimal. In particular, we will obtain an optimal bound on the performance of {\em any} estimation procedure for finding the equilibrium point along an adversarially-chosen sequence of small perturbations of the target problem.  The end result is summarized informally as follows.
\begin{theorem}[Asymptotic optimality, informal; see Theorem~\ref{thm:optimality}]\label{thm:lower_inform}
  Suppose the same setting as in Theorem~\ref{thm:main_informal} and let $\cL\colon \RR^{d} \rightarrow [0,\infty)$ be any symmetric, quasiconvex, lower semicontinuous loss functional. Fix any procedure for finding equilibrium points that outputs an estimator $\widehat x_k$ based on $k$ observed samples. As $k\to \infty$, there is a sequence of perturbed distribution maps $\mathcal{D}_k$ converging to  $\mathcal{D}$, along with corresponding equilibrium points $x_k^{\star}$ converging to $x^{\star}$, such that the following hold.
  \begin{enumerate}[label=(\roman*)] \item (\textbf{Lower bound}) \label{item:lowerbound_informal} The expected error $\mathbb{E}[\cL(\sqrt{k}(\widehat x_k-x_{k}^{\star}))]$ of the estimator $\widehat x_k$ on the perturbed problem is asymptotically lower-bounded by $\mathbb{E}[\cL(Z)]$, where $Z\sim \mathsf{N}(0, W^{-1} \Sigma W^{-\top})$.
	\item (\textbf{Tightness of SFB}) Moreover, if $\cL$ is bounded and continuous, then the lower bound in \ref{item:lowerbound_informal} is achieved by the estimator given by the averaged \ref{eqn:VI_iteration} iterate $\bar x_{k} = \frac{1}{k}\sum_{i=1}^{k}x_{i}$.
		  \end{enumerate}
\end{theorem}

The formal statement of the theorem and its proof follow closely the classical work of H\'{a}jek and Le Cam \citep{le2000asymptotics,van2000asymptotic} on statistical lower bounds and the more recent work of \citet{duchi2021asymptotic} on asymptotic optimality of the stochastic gradient method. In particular, the fundamental role of tilt-stability and the inverse function theorem highlighted by \cite{duchi2021asymptotic}  is replaced by the implicit function theorem paradigm.

Taken together, Theorems~\ref{thm:main_informal} and~\ref{thm:lower_inform} provide a solid theoretical footing for the practical application of \ref{eqn:VI_iteration}, which generalizes stochastic gradient descent. 
These results precisely quantify the asymptotic uncertainty of SFB, with confidence regions that are optimally narrow (in an appropriate sense) among all methods for finding equilibrium points. In particular, algorithms that use momentum or try to learn and adapt to how the distributions vary cannot achieve better asymptotic performance. Thus, stronger modeling assumptions are necessary to develop algorithms with provably superior asymptotic sample efficiency. 
We also note that all results in the paper extend directly to a minibatch variant of SFB, where in each iteration the update direction $G(x_t,z_t)$ is replaced by the  empirical average $\frac{1}{m}\sum_{i=1}^m G(x_t,z_{t,i})$ with $(z_{t,1}, \dots, z_{t,m})$ sampled i.i.d.\ from $\mathcal{D}(x_t)$. The only effect of the batching is that the asymptotic covariance $\Sigma$ is rescaled by $1/m$ in all results.

Before continuing, it is important to highlight a limitation of our results. In order to generate a sample $z_t\sim\mathcal{D}(x_t)$ in practice, one must first deploy the learning rule $x_t$ and then wait for the population to adapt. Consequently, the sampling and deployment have different associated ``costs.'' Our results can somewhat adapt to this imbalance by using minibatches, as explained above. Nonetheless, a more nuanced approach that balances sample complexity against the deployment cost is worth investigating in future work.

\subsection{Related Work}
Our work builds on existing literature in machine learning and stochastic optimization.

\paragraph{Learning with decision-dependent distributions.}
The basic setup for decision-dependent problems that we use is inspired by the performative prediction framework of \cite{perdomo2020performative} and its multiplayer extension developed independently by \citet{narang2022multiplayer}, \citet{piliouras2022multi}, and \citet{wood2022stochastic}. The stochastic gradient method for performative prediction was first introduced and analyzed by \cite{mendler2020stochastic}, while the stochastic forward-backward method for games was analyzed by \cite{narang2022multiplayer}. The related work of \cite{drusvyatskiy2020stochastic} showed that a variety of popular gradient-based algorithms for performative prediction can be understood
as the analogous algorithms applied to a certain static problem corrupted by a vanishing bias. In
general, performatively stable points (equilibria) are not  ``performatively optimal'' in the sense of \cite{perdomo2020performative}. Seeking to develop algorithms
for finding performatively optimal points, the work of \cite{miller2021outside} provides sufficient conditions for the prediction problem to be convex; extensions of such conditions to games appear in the papers of \cite{narang2022multiplayer} and \cite{wood2022stochastic}.
Algorithms for finding performatively optimal points under a variety of different assumptions and oracle models appear in the works of \cite{izzo2021learn}, \cite{jagadeesan2022regret}, \cite{miller2021outside}, \cite{narang2022multiplayer}, and \cite{wood2022stochastic}. The performative prediction framework is largely motivated by the problem of strategic classification \citep{hardt2016strategic}, which has been studied extensively from the perspective of causal inference \citep{bechavod2020causal, miller2020strategic} and convex optimization \citep{dong2018strategic}. Other lines of work \citep{brown2020performative, cutler2023distdrift, ray2022decision, wood2021online} in performative prediction have focused on the setting in which the environment evolves dynamically in time.

\paragraph{Stochastic approximation.}
There is extensive literature on stochastic approximation. The most relevant results for us are those of \cite{polyak92} that quantify the limiting distribution of the average iterate of stochastic approximation algorithms. Stochastic optimization problems with decision-dependent uncertainties have appeared in the classical stochastic programming literature; see, e.g., the works of \cite{ahmed2000strategic}, \cite{dupacova2006optimization}, \cite{jonsbraaten1998class}, \cite{rubinstein1993sensitivity}, and \cite{varaiya1988stochastic}.
We refer the reader to the recent paper of \cite{hellemo2018decision}, which discusses taxonomy and
various models of decision-dependent uncertainties. An important theme of these works is to utilize
structural assumptions on how the decision variables impact the distributions. In contrast, much of the work on performative prediction \citep{perdomo2020performative,narang2022multiplayer,wood2022stochastic,piliouras2022multi,drusvyatskiy2020stochastic,mendler2020stochastic} and our current paper are ``model-free.''

\paragraph{Local minimax lower bounds in estimation.}
There is a rich literature on minimax lower bounds in statistical estimation problems; we refer the reader to \citet[Chapter 15]{wainwright2019high} for a detailed treatment. Typical results of this type lower-bound the performance of any statistical procedure on a worst-case instance of that procedure. Minimax lower bounds can be quite loose as they do not consider the complexity of the particular problem that one is trying to solve but rather that of an entire problem class to which it belongs. More precise local minimax lower bounds, as developed by H\'ajek and Le Cam \citep{le2000asymptotics,van2000asymptotic}, provide much finer problem-specific guarantees. Building on this framework, \citet{duchi2021asymptotic} showed that the stochastic gradient method for standard single-stage stochastic optimization problems is, in an appropriate sense, locally asymptotically minimax optimal. Our paper builds heavily on this line of work.

\subsection{Outline} The outline of the paper is as follows. Section~\ref{sec:bas_notation} records some basic notation that we will use. Section~\ref{sec:performative} formally introduces/reviews the decision-dependent framework. In Section~\ref{sec:normal}, we show that the running average of the stochastic forward-backward algorithm is asymptotically normal (Theorem~\ref{thm:main_informal}), and identify its asymptotic covariance. Finally, Section~\ref{sec:optimality} presents the local minimax lower bound (Theorem~\ref{thm:lower_inform}). We defer many of the technical proofs to the appendices.

\section{Notation and Definitions}\label{sec:bas_notation}

Throughout, we let $\R^d$ denote the standard $d$-dimensional Euclidean space equipped with the dot product $\langle x, y \rangle = x^{\top}y$ and the induced norm $\|x\|=\sqrt{\langle x,x\rangle}$. For any set $\mathcal{X}\subset\R^d$, the symbol $\proj_{\mathcal{X}}(x)$ will denote the set $\argmin_{y\in\cX}\|y-x\|$ of nearest points of $\cX$ to $x\in \R^d$. We say that a function $\mathcal{L}\colon \RR^{d} \rightarrow \RR$ is {\em symmetric} if it satisfies $\cL(x) = \cL(-x)$ for all $x\in\R^d$, and we say that $\cL$ is {\em quasiconvex} if its sublevel set $\{x \mid \cL(x) \leq c\}$ is convex for any $c \in \RR.$ For any matrix $A\in \RR^{m\times n}$, the symbols  $\|A\|_{\rm op}$ and $A^{\dagger}$ stand for the operator norm and Moore-Penrose pseudoinverse of $A$, respectively.
For any two symmetric matrices $A,B\in\RR^{n\times n}$, we write $A \succeq B$ if the matrix $A - B$ is positive semidefinite.

\paragraph{Strong monotonicity and smoothness.} A map $F\colon\cX\to\R^d$ is called {\em $\alpha$-strongly monotone} on  $\mathcal{X}\subset\R^d$ if $\alpha>0$ and
$$\langle F(x)-F(x'),x-x'\rangle\geq \alpha\|x-x'\|^2\qquad \textrm{for all }x,x'\in \cX.$$
If $F = \nabla f$ for some $C^1$-smooth function $f$, then $\alpha$-strong monotonicity of $F$ is equivalent to $\alpha$-strong convexity of $f.$
We say that a map $F\colon\cX\to\R^m$ is {\em smooth} on a set $\mathcal{X}\subset\R^d$ if $F$ has a differentiable extension on an open neighborhood of each point of $\mathcal{X}$; further, we say that $F$ is {\em $\beta$-smooth} on $\cX$ if the Jacobian of $F$ satisfies the Lipschitz condition
$$\|\nabla F(x)-\nabla F(x')\|_{\rm op}\leq \beta \|x-x'\|\qquad \textrm{for all }x,x'\in \cX.$$

\paragraph{Probability measures.}  Given a nonempty Polish metric space $\Z$ (i.e., separable and complete), we equip $\Z$ with its Borel $\sigma$-algebra $\cB(\Z)$ and let $P_{1}(\Z)$ denote the set of probability measures on $\Z$ with finite first moment. We will measure the deviation between two measures $\mu,\nu\in P_{1}(\Z)$ using the Wasserstein-1 distance:
\begin{equation}\label{eq:W1}
	W_1(\mu,\nu) = \sup_{\phi\in\text{Lip}_1(\Z)}\!\left\{\underset{X\sim\mu}{\Expect}\big[\phi(X)\big] - \underset{Y\sim\nu}{\Expect}\big[\phi(Y)\big]\right\}\!.
\end{equation}
Here, $\text{Lip}_1(\Z)$ denotes the set of $1$-Lipschitz functions $\Z\rightarrow\R$. Equipped with the metric $W_1$, the set $P_1(\cZ)$ becomes a Polish metric space.

For any two probability measures $\mu$ and $\nu$ on $\cZ$ such that $\mu$ is absolutely continuous with respect to $\nu$ (denoted $\mu \ll \nu$) and any convex function $f\colon (0,\infty) \rightarrow \RR$ with $f(1) = 0$, the $f$-divergence of $\mu$ from $\nu$ is given by
\begin{equation}
	\label{eq:fdivergence}
	\Delta_{f}(\mu \!\parallel\! \nu) = \int_{\cZ} f\bigg(\frac{d\mu}{d\nu}\bigg)\,d\nu,
\end{equation}
where $\frac{d\mu}{d\nu}\colon\Z\rightarrow[0,\infty)$ denotes the Radon-Nikodym derivative of $\mu$ with respect to $\nu$ and we take $f(0)=\lim_{t\downarrow 0}f(t)$. Abusing notation slightly, if $\mu$ is not absolutely continuous with respect to $\nu$, then we set $\Delta_{f}(\mu \!\parallel\! \nu)=\infty$. We will refer to a Borel measurable map between metric spaces simply as \emph{measurable}. Likewise, we will refer to Borel measurable sets simply as {\em measurable}. 

\paragraph{Notions of convergence.} Given a sequence of random vectors $X_k\colon\Omega_k\rightarrow\R^m$ defined on probability spaces $(\Omega_k,\cS_k,P_k)$ and a random vector $X\sim\mu$ in $\R^m$, we write either $X_k \leadsto X$ or  $X_k \leadsto \mu$ to indicate that $X_k$ converges in distribution to $X$ (i.e., $\lim_{k\rightarrow\infty}{\EE}_{\smash{P_k}} [\varphi(X_k)] = {\EE}_{X\sim\mu}[\varphi(X)]$ for every bounded continuous function $\varphi\colon\R^m\rightarrow\R$). We write $X_k = o_{P_{k}}(1)$ if $X_k$ tends to zero in $P_k$-probability (i.e., $\lim_{k\rightarrow\infty}P_k\{\|X_k\|<\varepsilon\}=1$ for all $\varepsilon>0$). If $X$ and each $X_k$ are defined on a common probability space $(\Omega,\cS,P)$, then the notation $X_k \xlongrightarrow{p}X$ indicates that $X_k$ converges to $X$ in probability (i.e., $\lim_{k\rightarrow\infty}P\{\|X_k-X\|<\varepsilon\}=1$ for all $\varepsilon>0$), and the notation $X_k \xlongrightarrow{\text{a.s.}}X$ indicates that $X_k$ converges to $X$ almost surely (i.e., $P\{\omega\in\Omega \mid{\lim}_{k\rightarrow\infty} X_{k}(\omega) = X(\omega)\} = 1$).

For any pair of vector-valued sequences $(a_k)$ and $(b_k)$, we write $a_k = O(b_k)$ if there exists a constant $C>0$ such that $\|a_k\| \leq C\|b_k\|$ for all but finitely many $k$; we write $a_k = o(b_k)$ if for every $\varepsilon>0$, the inequality $\|a_k\| \leq \varepsilon \|b_k\|$ holds for all but finitely many $k$; we write $a_k = \Theta(b_k)$ if there exist constants $c,C>0$ such that $c \|b_k\| \leq \|a_k\| \leq C \|b_k\|$ for all but finitely many $k$; and we write $a_k\propto b_k$ if there exists a constant $c$ such that $a_k = cb_k$ for all but finitely many $k$.

\section{Background on Learning with Decision-Dependent Distributions}
\label{sec:performative}

In this section, we formally specify the class of problems that we consider along with relevant assumptions. In order to model decision-dependence, we fix a nonempty, closed, convex set $\mathcal{X}\subset\R^d$, a nonempty Polish metric space $(\Z,d_{\Z})$, and a map $\mathcal{D}\colon\mathcal{X}\to P_1(\Z)$. For ease of notation, we set $\D_x := \D(x)$ for each $x\in\cX$. Thus, 
 $\{\mathcal{D}_{x}\}_{x\in\cX}$ is a family of probability distributions on $\cZ$ indexed by points $x\in \cX$. The variational behavior of the map $\mathcal{D}\colon\mathcal{X}\to P_1(\Z)$ will play a central role in our work. In particular, following \cite{perdomo2020performative}, we will assume that $\mathcal{D}\colon\mathcal{X}\to P_1(\Z)$ is Lipschitz continuous. 

\begin{assumption}[Lipschitz distribution map]\label{assmp:lipdist}
  There is a constant $\gamma>0$ satisfying
  \begin{equation*}\label{distvar}
	W_1\big(\D(x),\D(x')\big)\leq \gamma\|x-x'\|  \qquad \textrm{for all }x,x'\in \cX.
\footnote{Assumption~\ref{assmp:lipdist} implies in particular that
$\{\mathcal{D}_{x}\}_{x\in\cX}$ is a {\em Markov kernel} from $\cX$
  to $\cZ$ (see Lemma~\ref{lem:markov}); this is crucial to have a well-defined probability space on which to analyze decision-dependent stochastic approximation problems.}
  \end{equation*}
\end{assumption}

Next, we fix a measurable map $G\colon\cX\times\cZ\to\R^d$ such that each section $G(x,\cdot)\colon\cZ\rightarrow\R^d$ is Lipschitz continuous, and we define the family of maps $G_{x} \colon \cX \rightarrow\R^d$ by setting
\begin{equation*}
G_x(y)=\underset{z\sim \mathcal{D}_{x}}{\Expect} G(y,z)	
\end{equation*}
for all $x,y\in\cX$; since $\cD_{x}$ has finite first moment, the Lipschitz continuity of $G(y,\cdot)$ guarantees that $G_{x}(y)$ is well defined. 
Additionally, we impose the following standard regularity conditions on $G\colon\X\times \mathcal{Z}\to\R^d$.

\begin{assumption}[Loss regularity]\label{assmp:lipgrad} There are constants $\beta, \bar L \geq 0$ and $\alpha>0$ and a measurable function $L\colon\cZ\rightarrow[0,\infty)$ satisfying the following three conditions.
  \begin{enumerate}[label=(\roman*)]
    \item\label{assmp:lipgrad1} {\bf (Lipschitz continuity)} For all $x,x'\in \mathcal{X}$ and $z,z'\in \mathcal{Z}$, the Lipschitz bounds
    \begin{align*}
    \|G(x,z)-G(x',z)\|&\leq L(z)\cdot\|x-x'\|,\\
    \|G(x,z)-G(x,z')\|&\leq \beta\cdot d_{\Z}(z,z')    
    \end{align*}
    hold. Further, the second moment bound ${\Expect}_{z\sim\D_{x}}[L(z)^2]\leq \bar L^2$ holds for all $x\in \cX$.
    \item\label{assmp:lipgrad2} {\bf (Monotonicity)} For all $x\in\X$, the map $G_x(\cdot)$ is $\alpha$-strongly monotone on $\cX$.
    \item\label{assmp:lipgrad3}  {\bf (Compatibility)} The inequality $\gamma\beta<\alpha$ holds.
  \end{enumerate}
\end{assumption}
A few comments are in order. Condition \ref{assmp:lipgrad1} asserts that the map $G(x,z)$ is separately Lipschitz continuous with respect to both $x$ and $z$; an immediate consequence is that  $G_{x}(\cdot)$ is $\bar L$-Lipschitz continuous. Condition \ref{assmp:lipgrad2} is a standard monotonicity requirement; when $G(x,z)=\nabla_{x}\ell(x,z)$, this corresponds to $\alpha$-strong convexity of the expected loss. Condition \ref{assmp:lipgrad3} ensures that the Lipschitz constant $\gamma$ of $\mathcal{D}(\cdot)$ is sufficiently small in comparison with the monotonicity constant $\alpha$, signifying that the dynamics are  ``mild.'' This condition is widely used in the existing literature; see, e.g., \cite{perdomo2020performative}, \cite{piliouras2022multi}, \cite{narang2022multiplayer}, and \cite{wood2022stochastic}.

Assumptions~\ref{assmp:lipdist} and \ref{assmp:lipgrad} imply the following useful Lipschitz estimate on the deviation $G_{x}(y)-G_{x'}(y)$ arising from the shift in distribution from $\mathcal{D}_x$ to $\mathcal{D}_{x'}$. We will use this estimate often in what follows. The proof is identical to that of Lemma 5 of \citet{narang2022multiplayer}; a short argument appears in Section~\ref{sec:proof_lem_dev}.  

\begin{lemma}[Deviation]\label{lem:dev_game_jac}
Suppose that Assumptions~\ref{assmp:lipdist} and \ref{assmp:lipgrad} hold. Then the estimate
\begin{align*}
\|G_{x}(y)-G_{x'}(y)\|&\leq \gamma\beta\cdot\|x-x'\|
\end{align*}
holds for all $x,x',y\in \X$.
\end{lemma}

Corresponding to each distribution $\cD_x$ is the variational inequality
\begin{equation}\label{eqn:VIx}
0\in \mathop{\EE}_{z\sim \mathcal{\cD}_x}G(y,z)+N_{\mathcal{X}}(y).\tag*{VI($\mathcal{D}_x$)}
\end{equation}
The following definition, originating in the work of \cite{perdomo2020performative} for performative prediction and \cite{narang2022multiplayer} for its multiplayer extension, is the key solution concept that we will use. 

\begin{definition}[Equilibrium point]
We say that $x^{\star}$ is an {\em equilibrium point} of the family of variational inequalities $\{\text{VI}(\cD_x)\}_{x\in\cX}$ if it satisfies:
$$0\in G_{x^{\star}}(x^{\star})+N_{\cX}(x^{\star}).$$
\end{definition}

Thus, $x^{\star}$ is an equilibrium point of $\{\text{VI}(\cD_x)\}_{x\in\cX}$ if $y = x^{\star}$ is itself a solution to the variational inequality $\text{VI}(\cD_{\xs})$ induced by the distribution $\mathcal{D}_{x^{\star}}$. Equivalently, these are exactly the fixed points of the  map
\begin{equation}\label{eqn:fixed_points}
\mathtt{Sol}(x):=\{y \mid 0\in G_{x}(y)+N_{\cX}(y)\},
\end{equation}
which is single-valued on $\cX$ by  the continuity and strong monotonicity of $G_x(\cdot)$ \citep[e.g., see][Example 12.7 and Proposition 12.54]{Rockafellar1998}. Equilibrium  points have a clear intuitive meaning: a learning system that deploys a learning rule $x^{\star}$ that is at equilibrium has no incentive to deviate from 
$x^{\star}$ based only on the data drawn from $\mathcal{D}(x^{\star})$. The key role of equilibrium points in (multiplayer) performative prediction is by now well documented; see, e.g., \cite{drusvyatskiy2020stochastic}, \cite{mendler2020stochastic}, \cite{narang2022multiplayer}, \cite{perdomo2020performative}, \cite{piliouras2022multi}, and \cite{wood2022stochastic}. Most importantly, equilibrium points exist and are unique under Assumptions~\ref{assmp:lipdist} and \ref{assmp:lipgrad}. The proof is identical to that of Theorem 7 of \citet{narang2022multiplayer}; we provide a short argument in Section~\ref{sec:proof_thm_exist} for completeness.  

\begin{theorem}[Existence]\label{thm:existence}
Suppose that Assumptions~\ref{assmp:lipdist} and \ref{assmp:lipgrad} hold. Then the map $\mathtt{Sol}(\cdot)$ is $\frac{\gamma\beta}{\alpha}$-contractive on $\mathcal{X}$ and therefore the problem admits a unique equilibrium point $x^{\star}$.
\end{theorem}

 We note in passing that when $\gamma\beta\geq \alpha$, equilibrium points may easily fail to exist; see, e.g., \citet[Proposition 3.6]{perdomo2020performative}. Therefore, the regime $\gamma\beta<\alpha$ is the natural setting to consider when searching for equilibrium points.

\section{Convergence and Asymptotic Normality}\label{sec:normal}
A central goal of performative prediction is the search for equilibrium points, which are simply the fixed points of the map $\mathtt{Sol}(\cdot)$ defined in \eqref{eqn:fixed_points}.
Though the map $\mathtt{Sol}(\cdot)$ is contractive, it cannot be evaluated directly since it involves evaluating the expectation $G_{x}(y)={\Expect}_{z\sim \mathcal{D}(x)}[G(y,z)]$. Employing the standard assumption that the only access to $\mathcal{D}(x)$ is through sampling, one may instead in iteration $t$ take a single stochastic forward-backward step on the problem corresponding to $\mathtt{Sol}(x_t)$. The resulting procedure is recorded in Algorithm~\ref{alg:1} below. In the setting of performative prediction \citep{mendler2020stochastic} and its multiplayer extension \citep{narang2022multiplayer}, the algorithm reduces to projected stochastic gradient methods.

\begin{algorithm}[h!]
  \caption{Stochastic Forward-Backward Method (SFB)\hfill} \label{alg:1}

  {\bf Input}: initial $x_0\in\mathcal{X}$ and step size sequence $(\eta_t)_{t\geq 0}\subset (0,\infty)$

  {\bf Step} $t\geq 0$:
  \begin{equation*}
    \begin{aligned}
      &\textrm{Sample~}  z_{t} \sim \cD(x_{t})  \\
      &\textrm{Set~} x_{t+1}=\proj_{\mathcal{X}}\big(x_t-\eta_tG(x_t,z_t)\big)
    \end{aligned}
  \end{equation*}
\end{algorithm}

For the remainder of Section~\ref{sec:normal}, we let $(x_t)_{t\geq0}$ denote the stochastic process generated by Algorithm~\ref{alg:1} on the probability space $(\cZ^{\NN}, \cB(\cZ^{\NN}), \mathbb{P})$, where $\PP =  \bigotimes_{i=0}^{\infty}\cD_{x_{i}}$ is the unique probability measure on the countable product space $\cZ^\NN$ satisfying
\begin{equation}\label{eq:prob}
	\PP(E_{0} \times \cdots \times E_{t} \times \cZ^{\NN}) = \int_{E_{0}}\cdots \int_{E_{t}}d\cD_{x_{t}}(z_{t})\cdots d\cD_{x_0}(z_0)
\end{equation}
for all $E_{0},\ldots,E_{t}\in\cB(\Z)$ and $t\geq0$ (see Theorem~\ref{thm:kolmogorov}). We will see that under very mild assumptions, the SFB iterates $x_t$ almost surely converge to the equilibrium point $x^{\star}$. To this end, we define for each $(x,z)\in\cX\times\cZ$ the noise vector
\begin{equation}\label{eq:noise_vector}
	\xi_{x}(z) := G(x, z) - G_{x}(x)
\end{equation}
and impose the following standard bound on the conditional second moment of the noise. In words, this assumption stipulates that the variance of the noise $\xi_{x_t}(z)$ with respect to the distribution $\cD_{x_t}$ induced by the iterate $x_t$ grows at most quadratically with the distance of $x_t$ to $x^{\star}$.

\begin{assumption}[Variance bound]\label{assmp:stoc}  
There is a constant $K\geq0$ such that for all $t\geq0$, the following bound holds almost surely: $$\underset{z_t\sim \mathcal{D}_{x_t}}{\Expect} \|\xi_{x_t}(z_t)\|^2 \leq K(1+\|x_t-\xs\|^2).$$ \end{assumption}

The subsequent proposition shows that the SFB iterates almost surely converge to the equilibrium point under Assumptions~\ref{assmp:lipdist}--\ref{assmp:stoc} and standard conditions restricting the rate of decrease of the step sizes $\eta_t$. The proof, which follows from a simple one-step improvement bound for the SFB method \cite[Theorem 24]{narang2022multiplayer} and an application of the Robbins-Siegmund almost supermartingale convergence theorem \citep{robbins71}, appears in Section~\ref{sec:proof_of_simple_prop}.
\begin{proposition}[Almost sure convergence]\label{prop:asconv}
  Suppose that Assumptions~\ref{assmp:lipdist}--\ref{assmp:stoc} hold
	and the step size sequence in Algorithm~\ref{alg:1} satisfies $\sum_{t=0}^{\infty}\eta_{t} = \infty$ and $\sum_{t=0}^{\infty}\eta_{t}^2 < \infty.$
  Then $x_{t}$ converges to $x^{\star}$ almost surely as $t\to\infty$, and $\sum_{t=0}^{\infty}\eta_{t}\|x_{t} - \xs\|^2 < \infty$ almost surely. Moreover, if $\eta_{t} = \Theta(t^{-\nu})$ for some $\nu\in\big(\tfrac{1}{2}, 1\big)$, then ${\EE} \|x_{t} - \xs\|^{2} = O(t^{-\nu})$ and hence $\sum_{t=1}^{\infty}t^{-1/2}\|x_{t} - \xs\|^2 < \infty$ almost surely.
\end{proposition}

The main result of this section is the asymptotic normality of the average iterates $$\bar x_{t} := \frac{1}{t}\sum_{i=1}^{t}x_{i},$$ for which we require the following additional assumption.

\begin{assumption}\label{assmp:const} The following four conditions hold.
\begin{enumerate}[label=(\roman*)]
\item\label{it:inter} {\bf (Interiority)} The equilibrium point $\xs$ lies in the interior of $\cX$.
\item\label{it:joint_smoothness} {\bf (Lipschitz Jacobian)} On a neighborhood of $x^{\star}$, the map $x \mapsto G_x(x)$ is differentiable with Lipschitz continuous Jacobian.
\item\label{it:uni_integral_smoothness} {\bf (Asymptotic uniform integrability)}  We have 
	$$\limsup_{t\to\infty}\underset{z_t\sim\cD_{x_t}}{\Expect}\!\big[\|G(\xs, z_t)\|^{2}\mathbf{1}_{\smash{\{\|G(\xs,z_t)\| \geq N\}}}\big] \xlongrightarrow{\text{a.s.}} 0 \qquad\text{as~}N\to\infty$$
	and
	$$\underset{z\sim\cD_{\xs}}{\Expect}\!\big[\|G(\xs, z)\|^{2}\mathbf{1}_{\smash{\{\|G(\xs,z)\| \geq N\}}}\big] \to 0 \qquad\text{as~}N\to\infty.$$
\item \label{it:Lindeberg} { \bf (Lindeberg's condition)} For all $\varepsilon>0$,
  		$$
  		\frac{1}{t}\sum_{i=0}^{t-1} \underset{z_i\sim\cD_{x_i}}{\Expect}\!\big[\|\xi_{x_i}(z_i)\|^{2}\mathbf{1}_{\smash{\{\|\xi_{x_i}(z_i)\| \geq \varepsilon\sqrt{t}\}}}\big] \xlongrightarrow{p} 0 \qquad\text{as~}t\to\infty.
  		$$
\end{enumerate}
\end{assumption}

A few comments are in order. First, the interiority condition \ref{it:inter} is a standard assumption for asymptotic normality results even in static settings \citep{polyak92}. The smoothness condition~\ref{it:joint_smoothness} is fairly mild. For example, it holds if the partial derivatives $\nabla_y G_x(y)$ and $\nabla_x G_x(y)$ exist and are Lipschitz continuous on a neighborhood of $(\xs,\xs)$; in turn, this holds if, on a neighborhood of $\xs$, each distribution $\mathcal{D}(x)$ admits a density $p(x,z) = \frac{d\cD(x)}{d\mu}(z)$ with respect to a common base measure $\mu \gg \cD(x)$ such that $G(\cdot, z)$ and $p(\cdot, z)$ are locally $C^{1,1}$-smooth\footnote{Recall that a map is said to be locally $C^{k,1}$-smooth if it is $C^k$-smooth and its $k^{\text{th}}$-order partial derivatives are locally Lipschitz continuous.} and sufficient integrability conditions hold to invoke dominated convergence.  

Asymptotic uniform integrability conditions such as \ref{it:uni_integral_smoothness} are key for obtaining convergence of moments \citep[see][Section~2.5]{van2000asymptotic}; it is used in our setting to establish $$\underset{z_t\sim\D_{\smash{x_t}}}{\Expect}\!\big[G(x_t, z_t)G(x_t, z_t)^\top\big] \xlongrightarrow{\text{a.s.}} \underset{z\sim\D_{\smash{\xs}}}{\Expect}\!\big[G(\xs, z)G(\xs, z)^\top\big] \qquad\text{as~}t\to\infty$$ (see Theorem~\ref{lem:covariance-convergence}). Condition \ref{it:uni_integral_smoothness} holds, for instance, if there exists a neighborhood $\cV$ of $\xs$ satisfying $\sup_{x\in\cV}{\Expect}_{z\sim\cD_{x}}\!\big[\|G(\xs, z)\|^{2}\mathbf{1}_{\smash{\{\|G(\xs,z)\| \geq N\}}}\big] \to 0$ as $N\to\infty$; in turn, this holds if $\sup_{x\in\cV}{\Expect}_{z\sim\cD_{x}}\!\big[\|G(\xs, z)\|^{q}\big] <\infty$ for some $q\in(2,\infty)$, e.g., if each random vector $G(\xs, z)$, with $z\sim \D_x$, is sub-Gaussian with the same variance proxy $\sigma^{2}$ for all $x\in\cV$.  Lindeberg's condition \ref{it:Lindeberg} imposes a standard constraint on the sequence of noise vectors $\xi_{x_t}(z_t)$ for application of the martingale central limit theorem (see Theorem~\ref{thm:clt}); it holds, for example, if both $\sup_{t\geq0} {\Expect}_{z_t \sim \D_{x_t}}\!\big[\|\xi_{x_t}(z_t)\|^{2}\big] < \infty$ almost surely and the asymptotic uniform integrability condition $\limsup_{t\to\infty}{\Expect}_{z_t \sim \D_{x_t}}\!\big[\|\xi_{x_t}(z_t)\|^{2}\mathbf{1}_{\smash{\{\|\xi_{x_t}(z_t)\| \geq N\}}}\big] \xlongrightarrow{p} 0$ as $N\to\infty$ is fulfilled.

We are now ready to present our main result. \begin{theorem}[Asymptotic normality]\label{thm:anperf}
  Suppose that Assumptions~\ref{assmp:lipdist}--\ref{assmp:const} hold and the step size sequence in Algorithm~\ref{alg:1} satisfies $\eta_{t} \propto t^{-\nu}$ for some $\nu\in\big(\tfrac{1}{2}, 1\big)$. Let $R\colon\cX\rightarrow\R^d$ and $\Sigma\succeq0$ be given by 
    \begin{equation*}
    R(x) = \underset{z\sim\D_x}{\Expect}[G(x,z)]
    \qquad\text{and}\qquad
	\Sigma = \underset{z\sim\D_{\xs}}{\Expect}\big[G(\xs, z)G(\xs, z)^\top\big],
  \end{equation*}
and let $\xi_t = \xi_{x_t}(z_t)$ denote the noise vector at step $t$ given by \eqref{eq:noise_vector}. 
Then, as $t\to\infty$, the iterates $x_t$ and their running averages $\bar{x}_t=\frac{1}{t}\sum_{i=1}^{t}x_{i}$ converge to $x^{\star}$ almost surely,
  \begin{equation*}
  	\sqrt{t}(\bar x_{t} - x^{\star}) = -\nabla R(\xs)^{-1}\bigg(\frac{1}{\sqrt{t}}\sum_{i=0}^{t-1}\xi_i\bigg)  + o_{\PP}(1),
  \end{equation*}
  and hence
  \begin{equation*}
    \sqrt{t}(\bar{x}_t - x^{\star})\leadsto\mathsf{N}\big(0,\nabla R(\xs)^{-1}\cdot\Sigma\cdot\nabla R(\xs)^{-\top}\big).
  \end{equation*}
\end{theorem}

Theorem~\ref{thm:anperf} asserts that  under mild assumptions, the deviations $\sqrt{t}(\bar{x}_t - x^{\star})$ converge in distribution to a Gaussian random vector with covariance matrix  
$\nabla R(\xs)^{-1}\cdot\Sigma\cdot\nabla R(\xs)^{-\top}$. Moreover, under mild regularity conditions  we may write
$$\nabla R(\xs)=\underbrace{\underset{z\sim\D(x^{\star})}{\Expect}[\nabla_x G(x^{\star},z)]}_\text{static}+ \underbrace{\frac{d}{dy} \underset{z\sim\D(y)}{\Expect}[G(\xs, z)]\Big|_{y=x^{\star}}}_\text{dynamic}\!.$$
It is part of the theorem's conclusion that the matrix $\nabla R(\xs)$ is invertible. It is worthwhile to note that the effect of the distributional shift on the asymptotic covariance is entirely captured by the second ``dynamic'' term in $\nabla R(\xs)$. When the distributions $\mathcal{D}(x)$ admit a density $p(x,z) = \frac{d\cD(x)}{d\mu}(z)$ as before, the Jacobian $\nabla R(x^{\star})$ admits the simple description:
$$\nabla R(\xs) = \underset{z\sim\D(x^{\star})}{\Expect}[\nabla_x G(x^{\star},z)] + \int G(\xs, z) \nabla_x p(\xs, z)^\top\,d\mu(z).$$

\begin{example}[Performative prediction with location-scale families]
As an explicit example of Theorem~\ref{thm:anperf}, let us look at the case when $G(x,z)=\nabla_x \ell(x,z)$ is the gradient of a loss function and $\mathcal{D}(x)$ is a ``linear perturbation'' of a fixed base distribution $\mathcal{D}_0$. Such distributions are quite reasonable when modeling performative effects, as explained by \cite{miller2021outside}. In this case, we have
$$z\sim \mathcal{D}(x)\quad \Longleftrightarrow \quad z-Ax\sim\mathcal{D}_0$$
for some fixed matrix $A\in\R^{n\times d}$, where $n$ is the dimension of the data $z$. Then a quick computation shows that we may write 
$$\nabla R(x)=\underset{z\sim\D(x)}{\Expect}\big[\nabla^2_{xx}\ell(x,z)+\nabla^2_{zx}\ell(x,z)A\big]$$
under mild integrability conditions. Thus, the dynamic part of $\nabla R(x^{\star})$ is governed by the product of the matrix of mixed partial derivatives $\nabla^2_{zx}\ell(x^{\star},z)\in\R^{d\times n}$  with $A$. The former measures the sensitivity of the gradient $\nabla_x\ell(x^{\star},z)$ at $x^\star$ to changes in the data $z$, while the latter measures the performative effects of the distributional shift.
\end{example}

\begin{example}[Multiplayer performative prediction with location-scale families]
More generally, let us look at the problem of multiplayer performative prediction \citep{narang2022multiplayer}. In this case, the map $G$ takes the form 
$$G(x,z)=\big(\nabla_1 \ell_1(x,z_1),\ldots,\nabla_k \ell_k(x,z_k)\big)$$
where $\ell_i$ is a loss for each player $i$ and $\nabla_i\ell_i$ denotes the gradient of $\ell_i$ with respect to the action $x_i$ of player $i$. The distribution $\mathcal{D}(x)$ takes the product form 
$$\mathcal{D}(x)=\mathcal{D}_1(x)\times\cdots\times \mathcal{D}_k(x).$$
As highlighted by \cite{narang2022multiplayer}, a natural parametric assumption is that 
there exist probability distributions $\mathcal{P}_{i}$ and matrices $A_i$, $A_{-i}$ such that the following holds:
$$z_i\sim \mathcal{D}_i(x) \quad \Longleftrightarrow\quad
z_i-A_ix_i-A_{-i}x_{-i}\sim \mathcal{P}_{i}.$$
Here $x_{-i}$ denotes the vector obtained from $x$ by deleting the coordinate $x_i$; thus, the distribution used by player $i$
 is a ``linear perturbation'' of a fixed base distribution $\mathcal{P}_{i}$. We can interpret the matrices $A_i$ and $A_{-i}$ as quantifying the performative effects of player $i$'s decisions and the rest of the players' decisions, respectively, on the  distribution $\mathcal{D}_i$ governing player $i$'s data.
It is straightforward to check the expression
$$\nabla R_i(x)=\mathop{\mathbb{E}}_{z_i\sim \mathcal{D}_i(x)} \!\left[\nabla^2_{x x_{i}}\ell_i(x,z_i)+ \nabla^2_{z_{i}x_{i}}\ell_i(x,z_i)[A_i, A_{-i}]\right]$$
under mild integrability conditions, where $[A_i, A_{-i}]x = A_ix_i + A_{-i}x_{-i}$. Thus, the dynamic part of $\nabla R_i(x^{\star})$ is governed by the product of the matrix of mixed partial derivatives $\nabla^2_{z_{i}x_{i}}\ell_i(x^{\star},z_i)$ with $[A_i, A_{-i}]$. \end{example}

\subsection{Proof of Theorem~\ref{thm:anperf}}
\label{sec:proof}

The proof of Theorem~\ref{thm:anperf} is based on the stochastic approximation result of Polyak and Juditsky \cite[Theorem 2]{polyak92}, which we review in Appendix~\ref{sec:rev_assymptot_norm}.
For the remainder of this section, we impose the assumptions of Theorem~\ref{thm:anperf}.

Consider the map $R\colon\cX\rightarrow\R^d$ given by
$
R(x) = G_x(x).
$
In light of the interiority condition $\xs\in\interior \X$ of Assumption~\ref{assmp:const}, the equilibrium point $\xs$ is the unique solution to the equation $R(x)=0$ on $\interior \X$. 
Observe that the noise vector $\xi_t = \xi_{x_t}(z_t)$ satisfies the relation
\begin{equation*}\label{eq:direction}
	G(x_t, z_t) = R(x_t) + \xi_t 
\end{equation*} 
and so we may write the iterates of Algorithm~\ref{alg:1} as
\begin{equation}\label{eq:process}
	x_{t+1} = x_{t} - \eta_{t}\big(R(x_t) + \xi_t + \zeta_{t}\big),
\end{equation}
where
\begin{equation}\label{eq:extra}
	\zeta_{t} := \frac{x_{t} - \eta_{t}\big(R(x_{t}) + \xi_{t}\big) - \proj_{\cX}\big(x_t-\eta_t\big(R(x_{t}) + \xi_{t}\big)\big)}{\eta_t}.
\end{equation}
Our goal is to apply Theorem~\ref{thm:polyak} to the process \eqref{eq:process} on the filtered probability space $(\cZ^{\NN}, \cB(\cZ^{\NN}), \FF, \PP)$, where $\mathbb{F} = (\cF_{t})_{t\geq 0}$ is the filtration given by
\begin{equation}\label{eq:filtration}
	\mathcal{F}_0 := \{\emptyset,\cZ^{\NN}\}\quad\text{and}\quad\mathcal{F}_t := \{A\times \cZ^{\NN} \mid A\in\cB(\Z^t)\} \qquad\text{for all~}t\geq 1
\end{equation} 
and $\PP =  \bigotimes_{i=0}^{\infty}\cD_{x_{i}}$ is given by \eqref{eq:prob}.
  In what follows, we establish the necessary assumptions for Theorem~\ref{thm:polyak}.
  
 To begin, we note that the map $R$ is Lipschitz continuous and strongly monotone on $\cX$; in particular, $R$ is measurable.
  
 \begin{lemma}[Lipschitz continuity and strong monotonicity]\label{lem:lip}
	The map $R$ is $(\bar{L}+\gamma\beta)$-Lipschitz continuous and $(\alpha - \gamma\beta)$-strongly monotone on $\cX$.
\end{lemma}
\begin{proof}
	Let $x,y\in\cX$. Then 
	\begin{equation*}
		\|R(x) - R(y)\|\leq \|G_{x}(x) - G_y(x)\| + \|G_{y}(x) - G_y(y)\| \leq (\gamma\beta + \bar{L})\|x-y\|
	\end{equation*}
	as a consequence of Lemma~\ref{lem:dev_game_jac} and the $\bar{L}$-Lipschitz continuity of $G_y(\cdot)$. Similarly, 
	\begin{align*}
		\langle R(x)-R(y),x-y\rangle &= \langle G_x(x)-G_y(x),x-y\rangle + \langle G_y(x)-G_y(y), x-y \rangle \\  &\geq -\|G_{x}(x) - G_y(x)\|\|x-y\| + \alpha\|x-y\|^2 \\ &\geq (-\gamma\beta + \alpha)\|x-y\|^2
	\end{align*}
	as a consequence of the $\alpha$-strong monotonicity of $G_y(\cdot)$ and Lemma~\ref{lem:dev_game_jac}.
\end{proof}
  
To establish Assumption~\ref{assmp:polyakstoc}, observe first that 
$\sup_{t\geq0} {\EE}\|\xi_{t}\|^{2} < \infty$ by Assumption~\ref{assmp:stoc} and Proposition~\ref{prop:asconv}. Clearly $x_t$ is $\cF_t$-measurable, $\xi_t$ and $\zeta_t$ are $\mathcal{F}_{t+1}$-measurable, and $\xi_t$ constitutes a martingale difference sequence satisfying
\begin{equation*}
	{\EE}[\xi_{t} \,|\, \cF_{t}] = \underset{z_t\sim\D_{x_t}}{\Expect}[G(x_t,z_t)] - G_{x_t}(x_t) = 0.
\end{equation*}
The following lemma shows that ${\Expect}[\xi_t\xi_t^\top \,|\, \mathcal{F}_t]$ converges to the positive semidefinite matrix
$$
\Sigma = \underset{z\sim\D_\xs}{\Expect}\big[G(\xs, z)G(\xs, z)^\top\big]
$$
almost surely as $t\to\infty$.

\begin{lemma}[Asymptotic covariance]\label{prop:covlimit}
	As $t\to\infty$, we have
	\begin{equation*}
		{\Expect}\big[G(x_t,z_t)G(x_t,z_t)^\top \,|\, \mathcal{F}_t\big]\xlongrightarrow{\textup{a.s.}}\Sigma\qquad\text{and}\qquad {\Expect}\big[\xi_t\xi_t^\top \,|\, \mathcal{F}_t\big]\xlongrightarrow{\textup{a.s.}}\Sigma.
	\end{equation*}
\end{lemma}
\begin{proof}
	Taking into account the almost sure convergence of $x_t$ to $\xs$ (Proposition~\ref{prop:asconv}), the uniform integrability condition \ref{it:uni_integral_smoothness} of Assumption~\ref{assmp:const}, and the Lipschitz condition \ref{assmp:lipgrad1} of Assumption~\ref{assmp:lipgrad}, we may apply Lemma~\ref{lem:covariance-convergence} with $g=G$ along any sample path witnessing $x_t\to x^{\star}$ to obtain ${\EE}[G(x_t,z_t)G(x_t,z_t)^\top\,|\,\cF_t\bigskip]\rightarrow\Sigma$ almost surely as $t\to\infty$. Therefore 
		\vspace{-3mm}
		\begin{equation*}
		{\Expect}\big[\xi_t\xi_t^\top \,|\, \mathcal{F}_t\big] = {\Expect}\big[G(x_t,z_t)G(x_t,z_t)^\top \,|\, \mathcal{F}_t\big] - R(x_t)R(x_t)^\top\xlongrightarrow{\text{a.s.}}\Sigma \qquad\text{as~}t\to\infty
		\end{equation*}
		by virtue of the continuity of $R$ and the relation $R(x^\star)=0$. 
\end{proof}

By Lemma~\ref{prop:covlimit}, we have $\frac{1}{t}\sum_{i=0}^{t-1}{\EE}\big[\xi_{i}\xi_{i}^{\top}\,|\, \cF_{i}\big] \xlongrightarrow{\text{a.s.}} \Sigma$ as $t\to\infty$.   Conditions~\ref{cond1polyakstoc} and \ref{cond2polyakstoc} of Assumption~\ref{assmp:polyakstoc} are now established, and Lindeberg's condition \ref{cond3polyakstoc}  of Assumption~\ref{assmp:polyakstoc} holds by item~\ref{it:Lindeberg} of Assumption~\ref{assmp:const}. Now consider the residual vector $\zeta_t$ given by \eqref{eq:extra}. Since $\xs\in\interior\X$ and $x_t \xlongrightarrow{\text{a.s.}} \xs$ as $t\to\infty$, we have $\PP\{\zeta_t = 0 \text{~for all but finitely many~} t\}=1$ and hence $\frac{1}{\sqrt{t}}\sum_{i=0}^{t-1}\|\zeta_i\| \xlongrightarrow{\text{a.s.}} 0$ as $t\to\infty$. Thus, condition~\ref{cond4polyakstoc} of Assumption~\ref{assmp:polyakstoc} holds, and the verification of Assumption~\ref{assmp:polyakstoc} is complete. 

We turn now to Assumption~\ref{assmp:polyak_lin_conv}. The first two conditions of Assumption~\ref{assmp:const} assert that the map $R$ is differentiable on a neighborhood of $\xs\in\interior\cX$. Since $R$ is $(\alpha - \gamma\beta)$-strongly monotone on $\cX$ (Lemma~\ref{lem:lip}), it follows that we have $\langle \nabla R(\xs) v, v \rangle \geq \alpha - \gamma \beta$ for every unit vector $v\in \SSS^{d-1}$ and hence every eigenvalue of $\nabla R(x^\star)$ has real part no smaller than $\alpha - \gamma \beta$. This is the content of the following lemma.

\begin{lemma}[Positivity of the Jacobian]\label{lem:posjac}
	For any point $x\in\interior\cX$ at which $R$ is differentiable, we have 
	\begin{equation}\label{ineq:posform}
		\langle \nabla R(x) v, v \rangle \geq \alpha - \gamma \beta \qquad \text{for all }v\in\mathbf{S}^{d-1}
	\end{equation}
	and hence every eigenvalue of $\nabla R(x)$ has real part no smaller than $\alpha - \gamma \beta$. In particular, $\nabla R(\xs)$ is positively stable. 
\end{lemma}
\begin{proof}
Suppose $R$ is differentiable at $x\in\interior\cX$.  By Lemma~\ref{lem:lip}, $R$ is $(\alpha - \gamma\beta)$-strongly monotone on $\cX$, so \eqref{ineq:posform} follows immediately from the definitions of differentiability and strong monotonicity: for any unit vector $v\in\mathbf{S}^{d-1}$,
$$
\langle \nabla R(x) v, v \rangle = t^{-2}\langle R(x+tv) - R(x), tv \rangle + o(1) \geq \alpha - \gamma\beta + o(1)\qquad\text{as~}t\to 0.
$$
Next, observe that \eqref{ineq:posform} implies $\lambda_{\min}(\nabla R(x) + \nabla R(x)^{\top}) \geq 2(\alpha - \gamma \beta)$. Now let $w \in \C^{d}$ be a normalized eigenvector of $\nabla R(x)$ with associated eigenvalue $\lambda \in\C$.   
	Letting $w^*$ denote conjugate transpose of $w$, we conclude
	$$
	2(\alpha - \gamma \beta) \leq w^*\big(\nabla R(x) + \nabla R(x)^{\top}\big)w  = w^*\nabla R(x)w+ (w^*\nabla R(x)w)^* =  \lambda + \bar{\lambda} = 2(\text{Re}\,\lambda),
	$$
	where the first inequality follows from the Rayleigh-Ritz theorem. Thus, every eigenvalue of $\nabla R(x)$ has real part no smaller than $\alpha - \gamma \beta$. In particular, every eigenvalue of $\nabla R(x)$ has positive real part, that is,  $\nabla R(x)$ is positively stable. The last claim of the lemma follows since $R$ is differentiable at $\xs\in\interior\cX$ by Assumption~\ref{assmp:const}.
\end{proof}

 Next, recall $\eta_t \propto t^{-\nu}$ for some $\nu\in\big(\tfrac{1}{2}, 1\big)$, i.e., there exist a constant $c>0$ and an index $T\geq1$ such that $\eta_t = c t^{-\nu}$ for all $t\geq T$. Clearly $\eta_t = o(1)$. Moreover, 
	$$
	0 \leq \frac{\eta_{t} - \eta_{t+1}}{\eta_{t}^{2}} = \frac{t^{\nu}}{(t+1)^{\nu}}\cdot\frac{(t+1)^{\nu} - t^{\nu}}{c}\leq \frac{(t+1)^{\nu} - t^{\nu}}{c} \qquad\text{for all~}t\geq T,
	$$
and since ${\lim_{t\rightarrow \infty}}\big((t+1)^{r} - t^{r}\big) = 0$ for any $r \in (0,1)$, we conclude $$\frac{\eta_{t}-\eta_{t+1}}{\eta_{t}} = o(\eta_t).$$ 
This establishes condition \ref{assmp:polyak_step_size} of Assumption~\ref{assmp:polyak_lin_conv}. 

Finally, by Proposition~\ref{prop:asconv}, we have $x_t \xlongrightarrow{\text{a.s.}} x^\star$ and hence $\bar{x}_t \xlongrightarrow{\text{a.s.}} x^\star$ as $t\to\infty$, and $\sum_{t=1}^{\infty}t^{-1/2}\|x_{t} - \xs\|^2 < \infty$ almost surely, which by Kronecker's lemma \cite[see][Lemma~2.5.9]{durrett2019_prob} implies $\frac{1}{\sqrt{t}}\sum_{i=0}^{t-1}\|x_i - x^\star\|^2 \xlongrightarrow{\text{a.s.}} 0$ as $t\to\infty$; on the other hand,
$$ R(x) - \nabla R(\xs)(x-\xs) = O(\|x-\xs\|^{2})\qquad\text{as } x\rightarrow \xs$$
since $\nabla R$ is Lipschitz continuous on a neighborhood of $\xs$ and $R(\xs)=0$. Therefore
\begin{equation}
    	\frac{1}{\sqrt{t}}\sum_{i=0}^{t-1} \|R(x_i) - \nabla R(\xs)(x_i - \xs)\| \xlongrightarrow{\text{a.s}} 0 \qquad\text{as~}t\to\infty.
    \end{equation}
Since $\nabla R(\xs)$ is positively stable (Lemma~\ref{lem:posjac}), this concludes the verification of Assumption~\ref{assmp:polyak_lin_conv}. An application of Theorem~\ref{thm:polyak} to the process \eqref{eq:process} completes the proof of Theorem~\ref{thm:anperf}.
 
\section{Asymptotic Optimality}
\label{sec:optimality}

In this section, we establish the local asymptotic optimality of Algorithm~\ref{alg:1}.
Our result builds on classical ideas from H\'ajek and Le Cam \citep{le2000asymptotics,van2000asymptotic} on lower bounds for statistical estimation and the more recent work of \citet{duchi2021asymptotic} on asymptotic optimality of the stochastic gradient method. Throughout, we fix a base distribution map $\cD \colon \X \rightarrow P_{1}(\cZ)$ and a map $G\colon\cX\times \mathcal{Z}\to\R^d$ satisfying Assumptions~\ref{assmp:lipdist} and \ref{assmp:lipgrad}. We will be concerned with evaluating the performance of estimation procedures for finding the equilibrium points induced by an adversarially-chosen sequence of small perturbations $\cD'$ of $\cD$, where each $\cD'$ is ``admissible'' in the following sense. 

\begin{definition}[Admissible distribution map]
A distribution map $\cD' \colon \X \rightarrow P_{1}(\cZ)$ is \emph{admissible} if Assumptions~\ref{assmp:lipdist} and \ref{assmp:lipgrad} hold with $\cD'$ in place of $\cD$ (allowing for different constants $\gamma'$, $\bar{L}'$, $\alpha'$ in place of $\gamma$, $\bar{L}$, $\alpha$). For each admissible distribution map $\cD' \colon \X \rightarrow P_{1}(\cZ)$, the corresponding equilibrium point is denoted by $\xs_{\cD'}$.
\end{definition}

Let us start with some intuition before delving into the details.
Roughly speaking, we aim to show that the asymptotic covariance of the normalized error $\sqrt{t}(\bar{x}_t - x^{\star})$ in Theorem~\ref{thm:main_informal} is ``optimal'' among all algorithms for finding equilibrium points. To capture the notion of optimal covariance, a standard approach is to probe the normalized error with nonnegative ``loss'' functions $\cL\colon \RR^{d} \rightarrow [0,\infty)$ that are symmetric, quasiconvex, and lower semicontinuous, interpreting the concentration of $X_{1}\sim\mathcal{P}_1$ to be ``better'' than that of $X_{2}\sim\mathcal{P}_2$ if the inequality ${\EE}[\cL(X_1)] \leq {\EE}[\cL(X_2)]$ holds for all such $\cL$; if $X_1$ and $X_2$ are square-integrable, this relation clearly entails the positive semidefinite ordering $\mathbb{E}[X_1X_1^{\top}]\preceq \mathbb{E}[X_2X_2^{\top}]$ of second-moment matrices.\footnote{Note $\mathbb{E}[X_1X_1^{\top}]\preceq \mathbb{E}[X_2X_2^{\top}]$ if and only if ${\EE}[\cL_{u}(X_{1})] \leq {\EE}[\cL_{u}(X_{2})]$ for all $u\in\R^{d}$, where $\cL_{u}\colon \RR^{d} \rightarrow [0,\infty)$ is given by $\cL_u(x) = (u^{\top}x)^{2}= u^{\top}(x x^{\top})u$.}

Using this idea, we consider a local asymptotic notion of minimax risk that evaluates the performance of an arbitrary sequence of estimators on problems close to the one we wish to solve.
Since our target problem models stochasticity using the base distribution map $\D$, we will parameterize close problems through perturbations of $\cD$. More concretely, we will carefully construct for each $u\in\R^d$ a perturbation $\cD^u$ of $\cD$ such that, as $u\to0$, the distribution map $\cD^u$ is admissible with equilibrium point $\xs_u:= \xs_{\cD^u}$ near $x^{\star}$.
The primary goal of this section is to show that if $\widehat x_{k}\colon\Z^k\rightarrow\R^d$ is an arbitrary sequence of estimators (i.e., $\widehat x_{k}$ is a measurable function of $k$ observed samples) and $\cL\colon \RR^{d} \rightarrow [0,\infty)$ is symmetric, quasiconvex, and lower semicontinuous, then the following lower bound holds:
\begin{equation}
  \label{eq:intuition}
\underbrace{\sup_{\cI \subset \R^d,\, |\cI| < \infty}\liminf_{k \rightarrow \infty} \max_{u\in \cI} \,{\EE}_{P_{k,{\smash{u/\sqrt{k}}}}}\big[\cL \big(\sqrt{k} \big(\widehat x_{k} - \xs_{\smash{u/\sqrt{k}}}\big)\big)\big]}_{\text{local asymptotic minimax risk}} \geq {\EE}[\cL(Z)],
\end{equation}
where $P_{k,v} = \bigotimes_{i=0}^{k-1}\cD_{\smash{\tilde{x}_{i}}}^{v}$ denotes the distribution on $\Z^k$ induced by $\cD^{v}$ along an arbitrary ``dynamic estimation procedure'' and $Z\sim \mathsf{N}(0, W^{-1}\Sigma W^{-\top})$ with $\Sigma$ and $W$ as in Theorem~\ref{thm:main_informal}. 

The lower bound \eqref{eq:intuition} provides a precise expression of the optimality of the covariance of the limit distribution $\mathsf{N}(0, W^{-1} \Sigma W^{-\top})$. Moreover, we will show that equality is achieved in \eqref{eq:intuition} upon specializing to the dynamic estimation procedure corresponding to Algorithm~\ref{alg:1} with step sizes $\eta_k \propto k^{-\nu}$ (as in Theorem~\ref{thm:main_informal}) and taking $\widehat{x}_k$ to be given by the average iterates $\bar{x}_k=\frac{1}{k}\sum_{i=1}^{k}x_{i}$, provided $\cL$ is bounded and continuous.

To formalize the preceding discussion, we begin by defining the dynamic estimation procedure used to define the sequence of distributions $P_{k,v} = \bigotimes_{i=0}^{k-1}\cD_{\smash{\tilde{x}_{i}}}^{v}$ appearing in \eqref{eq:intuition}.

 \begin{definition}[Dynamic estimation procedure]\label{ass:beautiful-algo}
 A \emph{dynamic estimation procedure} is a sequence of measurable maps $\mathcal{A}_k\colon \mathcal{Z}^{k}\times\cX^{k}\to \mathcal{X}$ such that for any initial point $ \tilde{x}_0\in \mathcal{X}$, the sequence of estimators $\tilde{x}_k\colon\Z^k\to\cX$ defined recursively by
 \begin{equation}
 	\label{eq:Palgo}
\tilde x_{k} = \cA_{k}(z_{0}, \dots, z_{k-1},  \tilde x_{0}, \dots,  \tilde x_{k-1})\end{equation}
satisfies
 \begin{align*}
		\tilde x_{k} \xlongrightarrow{\text{a.s.}} \xs \qquad\text{as~}k\to\infty
	\end{align*}
 with respect to the distribution $\bigotimes_{i=0}^{\infty}\cD_{\tilde{x}_{i}}$ on $\cZ^{\NN}$.
 \end{definition}

Thus, the dynamic estimation procedure $\cA_k$ plays the role of the decision-maker that selects the sequence of points at which to query a given distribution map; this generalizes the classical static setting wherein $z_0,z_1,\ldots$ are i.i.d.\ samples drawn from a fixed distribution. In the dynamic setting, we are concerned with algorithms for estimating the equilibrium point $x^\star$, so it is sensible to require that the iterates $\tilde{x}_k$ produced by the recursion \eqref{eq:Palgo} with $(z_{0}, \ldots, z_{k-1})\sim \bigotimes_{i=0}^{k-1}\cD_{\tilde{x}_{i}}$ converge almost surely to $x^{\star}$ as $k\to\infty$. Importantly, $\cA_k$ is assumed to be a deterministic function of its arguments. 
For example, the sequence of maps $\cA_k$ corresponding to
Algorithm~\ref{alg:1}, i.e.,
\begin{equation}\label{eq:sfb_dyn}
	\cA_{k+1}(z_{0}, \dots, z_{k},  x_{0}, \dots,  x_{k}) = \proj_{\cX}\big({x}_{k} - \eta_{k}G({x}_{k}, z_{k})\big)\qquad\text{for all~}k\geq 0,
\end{equation}
is a dynamic estimation procedure under the assumptions of  Proposition~\ref{prop:asconv};
although this particular map $\cA_{k+1}$ depends directly only on the last iterate $x_{k}$ and the last sample $z_{k}$, general dynamic estimation procedures may depend directly on any number of the previous samples and iterates.

We turn now to defining the perturbations $\cD^u$ of $\cD$ used to encode difficult instances near the target problem.

\subsection{Tilted Distributions}

Following \cite{duchi2021asymptotic} and \citet[Section~25.3]{van2000asymptotic}, for each distribution $\cD_{x} := \mathcal{D}(x)$ we will construct ``tilt perturbations'' $\cD_{x}^{u}$ parameterized by $u\in\R^d$. 
Henceforth, we fix an arbitrary nondecreasing $C^{3}$-smooth  function $h \colon \RR \rightarrow [-1, 1]$ such that the first three derivatives of $h$ are bounded and $h(t) = t$ for all $t$ in a neighborhood of zero. For each $x\in\cX$ and $u\in\R^d$, the tilted distribution $\cD_{x}^{u}\in P_1(\cZ)$ is defined by setting
\begin{equation}
  \label{eq:perturbation-of-the-force}
\cD^u_{x}(E) := \int_{E}\frac{1+h(u^{\top} g_x(z))}{C^{u}_{x}}\,d\cD_{x}(z)\qquad\text{for all~}E\in\cB(\cZ),
\end{equation}
where $g_x\colon\cZ\rightarrow\R^d$ is $\cD_x$-integrable with ${\Expect}_{z\sim\D_x}[g_x(z)] = 0$ and $C_{x}^{u}$ is the normalizing constant $C_{x}^{u} = 1 + {\EE}_{z\sim \D_x} [h(u^\top g_{x}(z))]$. The resulting parametric statistical model $\{\cD_x^u \mid u\in\R^d\}$ has score function $g_x$ at zero, i.e.,
\begin{equation*}
	\nabla_u\bigg(\!\log\frac{d\cD_{x}^{u}}{d\cD_{x}}(z)\!\bigg)\bigg|_{u=0} = g_{x}(z).
\end{equation*}
Thus, the collection of functions $\{u^{\top}g_x\colon\cZ\rightarrow\R \,|\, u\in\R^d\}$ forms a ``tangent space'' of the model $\{\cD_x^u \,|\, u\in\R^d\}$ at zero \citep[see][Example~25.15]{van2000asymptotic}. 
In the context of establishing the asymptotic optimality of Algorithm~\ref{alg:1}, we will see that the relevant score function is the noise  $\xi_{x}(z) = G(x, z) - G_{x}(x)$.

To guarantee that the tilted distribution map given by $x\mapsto\cD_x^u$ is admissible for small $u$, we require additional conditions on the base distribution map $\mathcal{D}$, the map $G$, and the function $g\colon\cX\times\cZ\rightarrow\R^d$ given by $g(x,z)=g_x(z)$. Despite being technical, these conditions (given in Assumption~\ref{ass:extra} and Definition~\ref{def:G_n} below) are mild and essentially amount to quantifying the smoothness of $\D$, $G$, and $g$. To quantify the smoothness of $\cD$, we will make use of a certain set of test functions to be integrated against each distribution $\mathcal{D}_{x}$.

\begin{definition}[Test functions]
Given a compact metric space $\mathcal{K}$, we let
$\mathcal{T}(\mathcal{K},\mathcal{Z})$ consist of all bounded measurable functions $\phi\colon\mathcal{K}\times\mathcal{Z}\rightarrow\R$ admitting a constant $L_{\phi}$ such that each section $\phi(\cdot, z)$ is $L_{\phi}$-Lipschitz on $\mathcal{K}$. For any $\phi\in\mathcal{T}(\mathcal{K},\mathcal{Z})$, we set $M_{\phi}:=\sup|\phi|$.
\end{definition}

\begin{assumption}\label{ass:extra}
  The following three conditions hold.
  \begin{enumerate}[label=(\roman*)]
    \item \textbf{(Compactness)} The set $\cX$ is compact, and the set $\cZ$ is bounded.
\item \textbf{(Smooth distribution map)} \label{ass:extra-smooth}
        There exists an increasing function $\vartheta\colon[0,\infty)\to[0,\infty)$ such that for every compact metric space $\cK$ and test function $\phi\in \mathcal{T}(\mathcal{K},\mathcal{Z})$, the function
    $$
    x \mapsto \mathop{\EE}_{z \sim \D_{x}}{\phi(y,z)}
    $$
    is $C^1$-smooth on $\cX$ for each $y\in\mathcal{K}$ and the map
    $$
    (x, y) \mapsto \nabla_{x}\Big(\mathop{\EE}_{z \sim \D_{x}}{\phi(y,z)}\Big) $$
    is $\vartheta({L}_{\phi} + M_{\phi})$-Lipschitz on $\mathcal{X}\times\mathcal{K}$.\footnote{The same conclusion then holds for all measurable maps $\phi\colon\cK\times\cZ\rightarrow\R^n$ with $n\in\NN$, $L_{\phi}:=\sup_{z}\text{Lip}\big(\phi(\cdot, z)\big)<\infty$, and $M_\phi := \sup \|\phi\|<\infty$.}
\item \textbf{(Lipschitz Jacobian)} There exist a measurable function $\Lambda\colon\cZ\rightarrow[0,\infty)$ and constants $\bar{\Lambda}, \beta' \geq 0$ such that for every $z \in \cZ$ and $x\in\X$, the section $G(\cdot, z)$ is $\Lambda(z)$-smooth on $\X$ with ${\EE}_{z \sim \D_{x}}[\Lambda(z)] \leq \bar{\Lambda}$, and the section $\nabla_x G(x,\cdot)$ is $\beta'$-Lipschitz on $\Z$.
  \end{enumerate}
\end{assumption}

The first condition is imposed mainly for simplicity.
The last two smoothness conditions are required in our arguments to apply  dominated convergence and implicit function theorems.
To illustrate with a concrete example, suppose that there exists a Borel probability measure $\mu$ on $\Z$ such that $\D_x \ll \mu$ for all $x\in\cX$, and consider the density $p(x,z) = \frac{d\cD(x)}{d\mu}(z)$. If there exist constants $\Lambda_p, L_p\geq0$ such that each section $p(\cdot, z)$ is $\Lambda_p$-smooth and $\sup_{x,z}\|\nabla_x p(x,z)\|\leq L_p$, then item \ref{ass:extra-smooth} of Assumption~\ref{ass:extra} holds with $\vartheta(s) = \max\{\Lambda_p, L_p\} \cdot s$.

Next, we specify the collection of functions $g\colon\cX\times\cZ\rightarrow\R^d$ satisfying the regularity conditions we require. 
\begin{definition}[Score functions]\label{def:G_n}
  Let  $\cG$ consist of all measurable functions $g\colon \cX \times \cZ\rightarrow\RR^{d}$ satisfying the following three conditions.
  \begin{enumerate}[label=(\roman*)]
    \item \textbf{(Lipschitz continuity)}
    There exists a constant $\beta_g\geq0$ such that for every $x\in\X$, the section $g(x,\cdot)$ is $\beta_g$-Lipschitz on $\cZ$.
    \item \textbf{(Unbiasedness)} ${\Expect}_{z\sim\D_{x}}[g(x,z)] = 0$ for all $x\in\X$.\item \textbf{(Smoothness)} There exist a measurable function $\Lambda_g\colon\cZ\rightarrow[0,\infty)$ and constants ${\bar{\Lambda}_g, \beta'_g \geq 0}$ such that for every $z \in \cZ$ and $x\in\X$, the section $g(\cdot, z)$ is $\Lambda_g(z)$-smooth on $\X$ with ${\Expect}_{z\sim\D_{x}}[\Lambda_g(z)] \leq \bar{\Lambda}_g$, and the section $\nabla_x g(x,\cdot)$ is $\beta'_g$-Lipschitz on $\Z$.
  \end{enumerate}
\end{definition}

\noindent For our purposes, the most important map in $\mathcal{G}$ will be the noise
\begin{equation}\label{eq:noise}
	\xi(x,z):= G(x, z) - G_{x}(x), 
\end{equation}
which belongs to $\mathcal{G}$ as a consequence of Assumptions~\ref{assmp:lipgrad} and \ref{ass:extra} and Lemma~\ref{lem:oh-the-smoothness}. 

Henceforth, we fix $g\in\cG$ and take $g_x(z)=g(x,z)$ in \eqref{eq:perturbation-of-the-force}, thereby defining the tilted distribution map $\cD^u \colon \X \rightarrow P_{1}(\cZ)$ given by $x\mapsto\cD_x^u$.
The following lemma guarantees that if Assumptions~\ref{assmp:lipdist}, \ref{assmp:lipgrad}, and \ref{ass:extra} hold, then $\cD^{u}$ is admissible for all $u$ in a neighborhood $\cU$ of zero; the proof, which we defer to Section~\ref{sec:proof_of_admis}, provides constants $\gamma^u$, $\bar L^u$, $\alpha^u$ that fulfill Assumptions~\ref{assmp:lipdist} and \ref{assmp:lipgrad} for $\mathcal{D}^u$ and deviate from $\gamma$, $\bar L$, $\alpha$ by $O(\|u\|)$ as $u\rightarrow 0$. 
\begin{lemma}[Tilted distributions are admissible]\label{lem:spain}
  Suppose that Assumptions~\ref{assmp:lipdist}, \ref{assmp:lipgrad}, and \ref{ass:extra} hold. Then there exists a neighborhood $\cU$ of zero such that for all $u\in \cU$, the map $\cD^{u}$ is admissible.
\end{lemma}

Thus, the solutions $x^{\star}_{\smash{\cD^{u}}}$ are well defined for small enough $u$. For ease of notation, we set
\begin{equation*}
		x^{\star}_u := x^{\star}_{\smash{\cD^{u}}}
\end{equation*}
for each $u$ in the neighborhood $\cU$. With the preceding definitions in place, we are now ready to state the main result of this section.

\begin{theorem}[Asymptotic optimality]\label{thm:optimality}
 Suppose that Assumptions~\ref{assmp:lipdist}, \ref{assmp:lipgrad}, and \ref{ass:extra} hold with the equilibrium point $\xs$ lying in the interior of $\cX$, and suppose $g=\xi$.\footnote{Recall that $g$ is the score function used to parameterize the perturbed distributions \eqref{eq:perturbation-of-the-force} and $\xi$ is the noise \eqref{eq:noise}.} Let $\mathcal{A}_k\colon \mathcal{Z}^{k}\times\cX^{k}\to \mathcal{X}$ be a dynamic estimation procedure, fix an initial point $\tilde{x}_0\in\cX$, and for each $u\in\R^d$, let $P_{k,u} = \bigotimes_{i=0}^{k-1}\cD_{\smash{\tilde{x}_{i}}}^{u}$ denote the distribution on $\Z^k$ induced by $\cD^u$ along the sequence \eqref{eq:Palgo}.
 Let $\widehat{x}_k\colon \cZ^{k}\rightarrow \RR^{d}$ be any sequence of estimators, and let $\mathcal{L}\colon \RR^{d} \rightarrow [0,\infty)$ be symmetric, quasiconvex, and lower semicontinuous. 
  \begin{enumerate}[label=(\roman*)]
    \item (\textbf{Lower bound}) The following lower bound on the local asymptotic minimax risk holds:
    	 \begin{equation}\label{eq:minimaxlb}
 		 	\sup_{\cI \subset \R^d,\, |\cI| < \infty}\liminf_{k \rightarrow \infty} \max_{u\in \cI} \,{\EE}_{P_{k,{\smash{u/\sqrt{k}}}}}\big[\cL \big(\sqrt{k} \big(\widehat x_{k} - \xs_{\smash{u/\sqrt{k}}}\big)\big)\big] \geq {\EE}[\cL(Z)],
  		\end{equation}
  		where $Z\sim \mathsf{N}\big(0, W^{-1} \Sigma W^{-\top}\big)$ with
  		\begin{equation*}\label{eq:cov-matrices}
   			 \Sigma = \underset{z\sim\D_{\xs}}{\Expect}\!\big[G(\xs, z)G(\xs, z)^\top\big] \quad \text{and} \quad W={\underset{z\sim\D_{x^{\star}}}{\Expect}\![\nabla_x G(x^{\star},z)]}+{\frac{d}{dy} \underset{z\sim\D_y}{\Expect}[G(\xs, z)]\Big|_{y=x^{\star}}}\!.
 	   \end{equation*}

    \item (\textbf{Tightness of SFB}) \label{item:tightness}
   	 If $\cA_k$ is the dynamic estimation procedure \eqref{eq:sfb_dyn} corresponding to Algorithm~\ref{alg:1} with initial point $x_0 = \tilde{x}_0$ and step sizes $\eta_{k} \propto k^{-\nu}$ for some $\nu\in\big(\tfrac{1}{2}, 1\big)$, and if the sequence of estimators $\widehat{x}_k$ is given by the average iterates $\bar{x}_k=\frac{1}{k}\sum_{i=1}^{k}x_{i}$, then equality holds in \eqref{eq:minimaxlb} whenever $\cL$ is bounded and continuous.
  \end{enumerate}
 
\end{theorem}

Most importantly, observe that the distribution of $Z$ in Theorem~\ref{thm:optimality} coincides with the asymptotic distribution of $\sqrt{t}(\bar{x}_t - x^{\star})$ in Theorem~\ref{thm:anperf}, thereby justifying asymptotic optimality of the stochastic forward-backward method (Algorithm~\ref{alg:1}). The lower bound in Theorem~\ref{thm:optimality} provides a decision-dependent analogue of the asymptotic optimality result of \citet[Theorem~1]{duchi2021asymptotic} when the minimizer lies in the interior of $\mathcal{X}$. We believe that all the techniques developed here can be adapted to the more general setting where $x^\star$ may lie on the boundary of $\cX$; since this generalization would require a significant technical overhead, we do not pursue it here. Taking the average of the SFB iterates to obtain an asymptotically optimal estimator as in item \ref{item:tightness} of Theorem~\ref{thm:optimality} is important: even in the static case, the last iterate is known to be asymptotically suboptimal \citep{fabian1968asymptotic}.

\begin{remark}[Convergence of equilibria and tilted distributions]
	In the setting of Theorem~\ref{thm:optimality}, one can show that the following approximations hold:
	\begin{equation}\label{eq:eq_conv}
		\|x^\star_{u}-x^\star\| = O(\|u\|) \qquad\text{as~}u\rightarrow0
	\end{equation} 
	and
	\begin{equation}\label{eq:W1_conv}
		\sup_{x\in\cX}W_1(\cD_x^u,\cD_x) = O(\|u\|)\qquad\text{as~}u\to0.
	\end{equation}
	Indeed, we will prove in the forthcoming Lemma~\ref{lem:smooth-kappa} that the map $u\mapsto x^\star_{u}$ is $C^1$-smooth on a neighborhood of zero, which implies \eqref{eq:eq_conv} by the mean value theorem.
	
	To verify the approximation \eqref{eq:W1_conv}, note first that for any $1$-Lipschitz function $\phi \in \text{Lip}_1(\Z)$, the translate $\bar{\phi}= \phi - \inf\phi$ is bounded by $\diam(\cZ)$, and
  \begin{equation}\label{eq:W1_conv_1}
   \mathop{\EE}_{z \sim \cD_{x}^{u}} [\phi(z)] - \mathop{\EE}_{z \sim \cD_{x}} [\phi(z)] = \frac{1}{C_{x}^{u}}\,  {\mathop{\EE}_{z \sim \cD_{x}}}\!\big[\bar{\phi}(z)\big(1+ h\big(u^{\top} g_{x}(z)\big)\big)\big] - \mathop{\EE}_{z \sim \cD_{x}} [\bar\phi(z)]	
  \end{equation}
  for any $x\in\cX$ and $u\in\R^d$. Further, Lemma~\ref{lem: uniform C(x,u)} shows $\sup_{x\in\cX} {\mathop{\EE}_{z \sim \cD_{x}}}|h(u^{\top}g_x(z))|=O(\|u\|)$ for all $u\in\R^d$ and $\sup_{x\in\cX}\frac{1}{C_x^u} = 1 + O(\|u\|^3)$ as $u\to 0$, so \eqref{eq:W1_conv_1} implies
\begin{equation*}
	\sup_{x\in\cX}W_1(\cD_x^u,\cD_x) \leq \diam(\cZ)\cdot\sup_{x\in\cX} {\mathop{\EE}_{z \sim \cD_{x}}}|h(u^{\top}g_x(z))| + O(\|u\|^3) = O(\|u\|) \qquad\text{as~}u\rightarrow0,
\end{equation*}
which follows from definition \eqref{eq:W1}.
This establishes \eqref{eq:W1_conv}, which in particular asserts that the collection of tilted distribution maps $\{\cD^{u}\}_{u\in\R^d}$ converges uniformly to $\cD$ as $u \to 0$.
\end{remark}

\begin{remark}[$f$-divergence of tilted distributions] We can also quantify the variation of the tilted distribution map $\D^u$ from the base distribution map $\cD$ via an $f$-divergence. Let $f\colon (0,\infty) \rightarrow \RR$ be any convex function that is $C^{2,1}$-smooth around $t=1$ and satisfies $f(1) = 0.$  Then for any distribution map $\cD' \colon \X \rightarrow P_{1}(\cZ)$, we may define the similarity measure
$$
\Delta_{f}(\D' \!\parallel\! \D) := \sup_{x\in \cX} \Delta_{f}(\D'_x \!\parallel\! \D_x),
$$
where $\Delta_{f}(\D'_x \!\parallel\! \D_x)$ denotes the usual $f$-divergence of $\D'_x$ from $\D_x$ given by \eqref{eq:fdivergence}.\footnote{Examples of $f$-divergences include the $\chi^2$-divergence, KL-divergence, and squared Hellinger distance.}  The following approximation holds:
\begin{equation}\label{eq:f_div_ineq}
	\Delta_{f}(\D^{u}\!\parallel\!\D) = O(\|u\|^2)\qquad\text{as~}u\to0.
\end{equation}
To verify \eqref{eq:f_div_ineq}, observe that for all sufficiently small $u\in\R^d$ and all $x\in\cX$, we have
\begin{align}
  \Delta_{f}(\cD_{x}^{u} \!\parallel\! \D_{x}) &= \int f\bigg(\frac{1+h(u^{\top}g_{x}(z))}{C_{x}^{u}}\bigg) \, d\mathcal{D}_{x}(z) \notag \\
  &= \int f\bigg(\frac{1+u^{\top}g_{x}(z)}{C_{x}^{u}}\bigg)\, d\mathcal{D}_{x}(z)\label{eqn:bounded}\\
  &= \frac{f''(1)}{2} u^{\top} \bigg(\mathop{\EE}_{z\sim \D_{x}}g_{x}(z)g_{x}(z)^{\top}\bigg)u + r_{x}(u),\label{eqn:expand}
\end{align}
where $\sup_{x \in \cX} |r_{x}(u)| = o(\|u\|^{2})$ as $u\to 0$. The equality \eqref{eqn:bounded} holds for all sufficiently small $u\in\R^d$ and all $x\in\cX$ because $g$ is uniformly bounded over $\mathcal{X}\times \mathcal{Z}$ (see Lemma~\ref{lem:ultra-bounded}) and $h(t) = t$ for all $t$ in a neighborhood of zero. The equality \eqref{eqn:expand} follows from a second-order approximation and the dominated convergence theorem; we defer the details to Lemma~\ref{lem:yet-another-bound}. Another appeal to the uniform boundedness of $g$ yields a constant $a \geq0$ for which $\sup_{x \in \cX} \|{\EE}_{z\sim \D_{x}} [g_{x}(z)g_{x}(z)^{\top}]\|_{\rm op} \leq a$. Further, given any $b>0$, there is a neighborhood $U$ of zero such that $\sup_{x \in \cX,\, u \in U} \|u\|^{-2} |r_{x}(u)| \leq b.$ Therefore  $ \Delta_{f}(\cD^{u}\!\parallel \!\D) \leq (\frac{a}{2} f''(1) + b)\|u\|^2$ for all sufficiently small $u\in\R^d$.

In light of \eqref{eq:f_div_ineq}, one may obtain from \eqref{eq:minimaxlb} a less refined local asymptotic minimax bound in terms of the ``admissible neighborhoods'' $B_{f}(\varepsilon)$ of $\cD$ defined for each $\varepsilon>0$ by  
\begin{equation*}
	B_{f}(\varepsilon):= \big\{\cD' \colon \X \rightarrow P_{1}(\cZ) \mid \cD'\text{~is admissible and~}\Delta_{f}(\D'\!\parallel\!\D)\leq \varepsilon \big\},
\end{equation*}
namely,
\begin{equation}\label{eq:minimaxlb_loose}
	\lim_{c\rightarrow \infty}\,\liminf_{k \rightarrow \infty} \sup_{\cD' \in B_{f}(c/k)} {\EE}_{P_{k}'}\big[\mathcal{L}\big(\sqrt{k} ( {\widehat x}_{k} - \xs_{\mathcal{D}'})\big)\big] \geq {\EE}[\mathcal{L}(Z)],
\end{equation}
where $P_{k}' = \bigotimes_{i=0}^{k-1}\cD_{\smash{\tilde{x}_{i}}}'$ denotes the distribution on $\Z^k$ induced by $\cD'$ along the sequence \eqref{eq:Palgo}.
Indeed, \eqref{eq:f_div_ineq} facilitates the elementary estimation
 \begin{align*}
&\lim_{c\rightarrow \infty}\,\liminf_{k \rightarrow \infty} \sup_{\cD' \in B_{f}(c/k)} {\EE}_{P_{k}'}\big[\mathcal{L}\big(\sqrt{k} ( {\widehat x}_{k} - \xs_{\mathcal{D}'})\big)\big] \\
&\hspace{1.5in}\geq \lim_{c\rightarrow \infty}\,\liminf_{k \rightarrow \infty} \sup_{\|u\|\leq c/\sqrt{k}} {\EE}_{P_{k,u}}\!\big[\cL\big(\sqrt{k} (\widehat{x}_{k} - \xs_{u})\big)\big] \\
&\hspace{1.5in}\geq \sup_{\cI \subset \R^d,\, |\cI| < \infty}\liminf_{k \rightarrow \infty} \max_{u\in \cI}\, {\EE}_{P_{k,{\smash{u/\sqrt{k}}}}}\big[\cL \big(\sqrt{k} \big(\widehat x_{k} - \xs_{\smash{u/\sqrt{k}}}\big)\big)\big]
  \end{align*}
  and hence \eqref{eq:minimaxlb} implies \eqref{eq:minimaxlb_loose}.
\end{remark}

\subsection{Proof of Theorem \ref{thm:optimality}}

The proof of Theorem~\ref{thm:optimality} is based on the classical H\'ajek-Le Cam minimax theorem. To state this result, we require several standard definitions from statistics. In the sequel, we let $\{Q_{k,u} \,|\, u\in \R^d\}$ denote a sequence of parametric statistical models, where $Q_{k,u}$ is a probability measure on $(\Omega_k, \cS_k)$ such that $Q_{k,u} \ll Q_{k,0}$ for each $k\in\N$ and $u\in \R^d$; following \cite{vanderVaart1996}, we write either $X_k \overset{\smash{u}}{\leadsto} X $ or $X_k \overset{\smash{u}}{\leadsto} \mathsf{D}$ to indicate that a sequence of random vectors $X_k\colon\Omega_k\rightarrow\R^m$ converges in distribution to a random vector $X\sim\mathsf{D}$  with respect to $Q_{k,u}$, i.e., $\lim_{k\rightarrow\infty}{\EE}_{\smash{Q_{k,u}}} [\varphi(X_k)] = {\EE}_{X\sim\mathsf{D}}[\varphi(X)]$ for every bounded continuous function $\varphi\colon\R^m\rightarrow\R$.

\begin{definition}[Locally asymptotically normal] The sequence $\{Q_{k,u} \!\mid\! u\in \R^d\}$ is \emph{locally asymptotically normal (LAN) with precision} $V\!$ \emph{at zero} if there exist a sequence of random vectors $Z_{k}\colon\Omega_k\rightarrow\R^d$ and a positive semidefinite matrix $V\in\R^{d\times d}$ such that $Z_{k} \overset{\smash{0}}{\leadsto}\mathsf{N}(0, V)$ and, for each $u\in\R^d$, 
  \begin{equation}\label{eq:LAN}
  \log \frac{dQ_{k, u}}{dQ_{k, 0}} =  u^{\top}Z_{k} - \frac{1}{2} u^{\top}V u + o_{Q_{k,0}}(1).
  \end{equation}
\end{definition}

\begin{definition}[Regular mapping sequence]
A sequence of mappings $\Gamma_{k}\colon \R^d \rightarrow \RR^{n}$ is \emph{regular with derivative} $\dot{\Gamma}$ \emph{at zero} if there exists a matrix $\dot{\Gamma} \in \RR^{n \times d}$ satisfying
  $$\lim_{k\rightarrow \infty}\sqrt{k} \big( \Gamma_{k}(u) - \Gamma_{k}(0)\big) = \dot{\Gamma} u \quad \text{for all~} u \in \R^d.$$
\end{definition}

 \begin{example}\label{example:regmapseq}
 	Given any $\psi\colon\R^d\to\R^n$ such that $\psi$ is differentiable at zero, the induced mapping sequence $\Gamma_{k}\colon \R^d \rightarrow \RR^{n}$ given by $\Gamma_k(u) = \psi(u/\sqrt{k})$ is clearly regular with derivative $\dot{\Gamma} = \nabla\psi(0)$ at zero. We will see that this construction provides the relevant regular mapping sequence for establishing Theorem~\ref{thm:optimality} by taking $\psi(u) = x^{\star}_{u}$ on a neighborhood of zero.
 \end{example}
 
 Equipped with the preceding definitions, we are ready to state the following version of the H\'ajek-Le Cam minimax theorem, which appears for example Lemma 8.2 of \citet{duchi2021asymptotic} and Theorem 3.11.5 of \citet{vanderVaart1996}.

\begin{theorem}[Local asymptotic minimax bound]\label{thm:localminimax}
  Let $\{Q_{k,u} \!\mid\! u\in \R^d\}$ be locally asymptotically normal with precision $V$ at zero, $\Gamma_{k}\colon \R^d \rightarrow \RR^{n}$ be a regular mapping sequence with derivative $\dot{\Gamma}$ at zero, and $\cL\colon \RR^{n} \rightarrow [0,\infty)$ be symmetric, quasiconvex, and lower semicontinuous. Then, for any sequence of estimators $T_{k}\colon \Omega_{k} \rightarrow \RR^{n}$, we have
  \begin{equation}\label{eq:lecam}
  \sup_{\cI \subset \R^d,\, |\cI| < \infty}\liminf_{k \rightarrow \infty}\, \max_{u \in \cI}\, {\EE}_{Q_{k, u}}\!\big[\cL \big(\sqrt{k}\big(T_{k} - \Gamma_{k}(u)\big)\big)\big] \geq {\EE}[\cL(Z)],
  \end{equation}
  where $Z \sim \mathsf{N}(0, \dot{\Gamma} (V + \lambda I)^{{-1}}\dot{\Gamma}^{\top})$ for any $\lambda>0$; if $V$ is invertible, then \eqref{eq:lecam} also holds with $Z \sim \mathsf{N}(0, \dot{\Gamma} V^{{-1}}\dot{\Gamma}^{\top})$.
  \end{theorem}

To establish the lower bound \eqref{eq:minimaxlb} in Theorem~\ref{thm:optimality}, we will apply Theorem~\ref{thm:localminimax} as follows. Suppose henceforth that Assumptions~\ref{assmp:lipdist}, \ref{assmp:lipgrad}, and \ref{ass:extra} hold with the equilibrium point $\xs$ lying in the interior of $\cX$. Let $\mathcal{A}_k\colon \mathcal{Z}^{k}\times\cX^{k}\to \mathcal{X}$ be a dynamic estimation procedure and fix an initial point $\tilde{x}_0\in\cX$ and a score function $g\in\cG$.
 For each $k\in\N$ and $u\in \R^d$, we let 
\begin{equation}\label{eq:P}
	P_{k,u} := \bigotimes_{i=0}^{k-1}\cD_{\tilde{x}_{i}}^{u}
\end{equation} 
denote the distribution on $\Z^k$ induced by $\cD^u$ along the sequence \eqref{eq:Palgo}, and we set
\begin{equation}\label{eq:Q}
	Q_{k,u} := P_{k,{\smash{u/\sqrt{{k}}}}}.
\end{equation} 
Further, we define $\psi\colon\R^d\rightarrow\R^d$ by 
\begin{equation*}
	\psi(u) = \begin{cases}
      \xs_{\smash{u}}& \text{if}\ u \in \cU \\
      0 & \text{otherwise}
    \end{cases}
\end{equation*}
and take $\Gamma_k\colon\R^d\rightarrow\R^d$ to be the induced mapping sequence given by 
\begin{equation*}
	\Gamma_k(u) = \psi(u/\sqrt{k});
\end{equation*} 
since $\cU$ is a neighborhood of zero, it follows that for each $u\in\R^d$, we have $\Gamma_k(u)=\xs_{\smash{u/\sqrt{k}}}$ for all but finitely many $k\in\NN$. 

We now state two key lemmas that will allow us to apply Theorem~\ref{thm:localminimax}; their proofs are deferred to Sections~\ref{sec:proof_LAN-new} and \ref{sec:proof_smooth-kappa}, respectively. The first lemma verifies that $\{Q_{k,u} \,|\, u\in \R^d\}$ is locally asymptotically normal at zero with precision 
 \begin{equation*}
 	\Sigma_{g} := \underset{z \sim \cD_{\xs}}{\mathop{\EE}}\!\big[g_{\xs}(z)g_{\xs}(z)^{\top}\big],
 \end{equation*}
 while the second lemma shows that $\psi$ is $C^1$-smooth on a neighborhood of zero and computes $\nabla\psi(0) = -W^{{-1}} \Sigma_{g, G}^{\top}$, where 
 \begin{equation*}
 	\Sigma_{g, G} := \underset{z\sim\cD_{\xs}}{\EE}\!\big[g_{\xs}(z) G(\xs, z)^{\top}\big].
 \end{equation*}

\begin{lemma}[LAN]\label{lem:LAN-new}
Let $Z_{k}\colon\cZ^k\rightarrow\R^d$ be the sequence of random vectors given by $$Z_{k} = \frac{1}{\sqrt{k}}\sum_{i=0}^{k-1} { g_{\tilde x_{i}}(z_{i})}.$$ Then $Z_k \overset{\smash{0}}{\leadsto} \mathsf{N}(0, \Sigma_g)$, where $\overset{\smash{0}}{\leadsto}$ denotes convergence in distribution with respect to $Q_{k,0}$. Moreover, for each $u\in\R^d$, 
\begin{equation*}
	\log \frac{dQ_{k, u}}{dQ_{k, 0}} = u^{\top}Z_{k} - \frac{1}{2} u^{\top}\Sigma_{g} u + o_{Q_{k,0}}(1).
\end{equation*}
Hence $\{Q_{k,u} \!\mid\! u\in \R^d\}$ is locally asymptotically normal with precision $\Sigma_{g}$ at zero.
\end{lemma}

\begin{lemma}[Smooth equilibrium perturbation] \label{lem:smooth-kappa}
  The map $\psi$ is $C^1$-smooth on a neighborhood of zero with $\nabla\psi(0) = -W^{{-1}} \Sigma_{g, G}^{\top}$. Hence $\Gamma_k$ is regular with derivative $\dot{\Gamma} = -W^{{-1}} \Sigma_{g, G}^{\top}$ at zero.
\end{lemma}

Importantly, taking $g$ to be the noise map $\xi$ given by \eqref{eq:noise} and noting $\xi(x^\star, z) = G(x^\star,z)$ yields
\begin{equation*}\label{eq:noisematrices}
	\Sigma_{\xi}=\Sigma_{\xi,G} = \underset{z\sim\cD_{\xs}}{\EE}\!\big[G(\xs, z) G(\xs, z)^{\top}\big] = \Sigma. 
\end{equation*}  
We are now in position to apply Theorem~\ref{thm:localminimax}. Let $\cL\colon \RR^{d} \rightarrow [0,\infty)$ be symmetric, quasiconvex, and lower semicontinuous, $\widehat x_{k}\colon \cZ^{k}\rightarrow \RR^{d}$ be any sequence of estimators, and suppose henceforth that $g=\xi$. 
Invoking Lemmas~\ref{lem:LAN-new} and \ref{lem:smooth-kappa} and applying Theorem~\ref{thm:localminimax} yields
\begin{align}\label{eq:4-almost}
	&\sup_{\cI \subset \R^d,\, |\cI| < \infty}\liminf_{k \rightarrow \infty} \max_{u\in \cI} \,{\EE}_{P_{k,{\smash{u/\sqrt{k}}}}}\big[\cL \big(\sqrt{k} \big(\widehat x_{k} - \xs_{\smash{u/\sqrt{k}}}\big)\big)\big]  \nonumber \\
	&\hspace{.9in}= \sup_{\cI \subset \R^d,\, |\cI| < \infty}\liminf_{k \rightarrow \infty} \max_{u\in \cI}\, {\EE}_{Q_{k,u}}\!\big[\cL\big(\sqrt{k} \big(\widehat{x}_{k} - \Gamma_k(u)\big)\big)\big] \geq {\EE}[\cL(Z_{\lambda})],
\end{align}
where $Z_{\lambda} \sim \mathsf{N}(0,W^{{-1}}\Sigma(\Sigma + \lambda I)^{{-1}} \Sigma W^{-\top})$ for any $\lambda > 0$. 

Letting $\lambda\downarrow0$ in \eqref{eq:4-almost} establishes \eqref{eq:minimaxlb}. Indeed, let $\Sigma = AA^{\top}$ be a Cholesky decomposition of $\Sigma$ and observe that the pseudoinverse identities 
$
A^\dagger = \lim_{\lambda\downarrow0} A^{\top}\big(AA^{\top} + \lambda I\big)^{-1}
$ 
and $AA^{\dagger}A = A$ imply
$$
\lim_{\lambda\downarrow0} \Sigma\big(\Sigma + \lambda I\big)^{{-1}} \Sigma = A\Big(\lim_{\lambda\downarrow0} A^{\top}\big(AA^{\top} + \lambda I\big)^{{-1}}\Big)AA^{\top} = \big(A A^{\dagger}A\big)A^{\top} = AA^{\top} = \Sigma. 
$$
Thus, upon setting $\widetilde{\Sigma}_\lambda := W^{{-1}}\Sigma(\Sigma + \lambda I)^{{-1}} \Sigma W^{-\top}$ and  $\widetilde{\Sigma} := W^{{-1}}\Sigma W^{-\top}$, we have $\widetilde{\Sigma}_\lambda \rightarrow \widetilde{\Sigma}$ as $\lambda\downarrow 0$. Further, for all $0<\lambda_2\leq\lambda_1$, we have ${\exp}\big({-}\tfrac{1}{2}v^{\top}\widetilde{\Sigma}_{\lambda_2}^{\dagger} v\big) \geq {\exp}\big({-}\tfrac{1}{2}v^{\top}\widetilde{\Sigma}_{\lambda_1}^{\dagger} v\big)$ for all $v\in\R^d$. Since the densities corresponding to $Z_{\lambda} \sim \mathsf{N}(0,\widetilde{\Sigma}_{\lambda})$ and $Z \sim \mathsf{N}(0,\widetilde{\Sigma})$ with respect to the Lebesgue measure restricted to $S:=\range \widetilde{\Sigma}$ are given by 
\begin{equation*}
	p_{\lambda}(v):= \frac{{\exp}\big({-}\tfrac{1}{2}v^{\top}\widetilde{\Sigma}_{\lambda}^{\dagger} v\big)}{\sqrt{(2\pi)^r\,{\det}^{*}\big(\widetilde{\Sigma}_\lambda\big)}}\qquad\text{and}\qquad p(v):= \frac{{\exp}\big({-}\tfrac{1}{2}v^{\top}\widetilde{\Sigma}^{\dagger} v\big)}{\sqrt{(2\pi)^r\,{\det}^{*}\big(\widetilde{\Sigma}\big)}},
\end{equation*}
where $r$ is the rank of $\Sigma$, we may therefore apply the monotone convergence theorem to obtain 
\begin{align*}
	\lim_{\lambda\downarrow0} {\EE}[\cL(Z_{\lambda})] &= \lim_{\lambda\downarrow0} \frac{1}{\sqrt{(2\pi)^r\,{\det}^{*}\big(\widetilde{\Sigma}_\lambda\big)}}\int_{S}\cL(v)\,{\exp}\big({-}\tfrac{1}{2}v^{\top}\widetilde{\Sigma}_{\lambda}^{\dagger} v\big)\,dv \\
	&= \frac{1}{\sqrt{(2\pi)^r\,{\det}^{*}\big(\widetilde{\Sigma}\big)}}\int_{S}\cL(v)\,{\exp}\big({-}\tfrac{1}{2}v^{\top}\widetilde{\Sigma}^{\dagger} v\big)\,dv \\
	&= {\EE}[\cL(Z)].
\end{align*}
Hence \eqref{eq:4-almost} entails \eqref{eq:minimaxlb}.

To prove the final claim of Theorem~\ref{thm:optimality}, we proceed by establishing a type of asymptotic equivariance of the average SFB iterates \citep[e.g., see][Lemma~8.14]{van2000asymptotic}.
\begin{lemma}[Asymptotic equivariance]\label{lem:equivariance}  
	Let $\cA_k$ be the dynamic estimation procedure \eqref{eq:sfb_dyn} corresponding to Algorithm~\ref{alg:1} with initial point $x_0 = \tilde{x}_0$ and step sizes $\eta_{k} \propto k^{-\nu}$ for some $\nu\in\big(\tfrac{1}{2}, 1\big)$. Then the average iterates $\bar{x}_k=\frac{1}{k}\sum_{i=1}^{k}x_{i}$ are asymptotically equivariant-in-law with respect to $\{Q_{k,u} \,|\, u\in \R^d\}$ for estimating $x^\star$, that is, for each $u\in\R^d$, 
	\begin{equation}\label{eq:equivariance}
		\sqrt{k}\big(\bar{x}_{k} - \Gamma_k(u)\big) \overset{\smash{u}}{\leadsto} \mathsf{N}\big(0, W^{-1}\Sigma W^{-\top}\big).
	\end{equation}
\end{lemma}
\begin{proof}
Lemma~\ref{lem:LAN-new} shows that the sequence of random vectors $Z_{k}\colon\cZ^k\rightarrow\R^d$ given by $$Z_{k} = \frac{1}{\sqrt{k}}\sum_{i=0}^{k-1} \xi_{x_i}(z_i)$$ satisfies
\begin{equation}\label{eq:lan_sfb_normal}
	Z_k \overset{\smash{0}}{\leadsto} \mathsf{N}(0, \Sigma),
\end{equation}
and, for each $u\in\R^d$, 
\begin{equation}\label{eq:lan_sfb}
	\log \frac{dQ_{k, u}}{dQ_{k, 0}} = u^{\top}Z_{k} - \frac{1}{2} u^{\top}\Sigma u + o_{Q_{k,0}}(1).
\end{equation} 
Moreover, Theorem~\ref{thm:anperf} reveals
\begin{equation}\label{eq:an_sfb_prob}
	\sqrt{k}(\bar x_{k} - \xs) = -W^{-1}Z_k + o_{Q_{k,0}}(1).
\end{equation} 	
Now let $\bar{Z}\sim\mathsf{N}(0, \Sigma)$, fix $u\in\R^d$, and consider the affine map $\varphi\colon\R^d\to\R^{d+1}$ given by
\begin{equation*}
	\varphi(z) =
	\begin{pmatrix} 
		-W^{-1} \\ 
		u^{\top} 	
	\end{pmatrix}
	\!z +
	\begin{pmatrix} 
		0  \\ 
		{-}\tfrac{1}{2}u^{\top}\Sigma u
	\end{pmatrix}\!.
\end{equation*}   
Then \eqref{eq:lan_sfb_normal} implies $\varphi(Z_k) \overset{\smash{0}}{\leadsto} \varphi(\bar{Z})$ and hence
\begin{equation}\label{eq:LeCam3rdLem_hyp}
	\!\begin{pmatrix} 
		\sqrt{k}(\bar{x}_k - x^{\star})  \\ 
		\log \frac{dQ_{k, u}}{dQ_{k, 0}} 
	\end{pmatrix}
	\overset{\smash{0}}{\leadsto} 
	\begin{pmatrix} 
		-W^{-1}\bar{Z}  \\ 
		u^{\top}\bar{Z} - \tfrac{1}{2} u^{\top}\Sigma u
	\end{pmatrix}
	\sim 
	\mathsf{N}\Bigg(\mkern-6mu \begin{pmatrix} 
						0  \\ 
						{-}\tfrac{1}{2}u^{\top}\Sigma u
					\end{pmatrix}\!, 
					\begin{pmatrix} 
						W^{-1}\Sigma W^{-\top} & {-}W^{-1}\Sigma u \\ 
						{-}u^{\top}\Sigma W^{-\top} & u^{\top}\Sigma u
					\end{pmatrix}\mkern-6mu\Bigg)\mkern-.2mu
\end{equation}   
by virtue of \eqref{eq:lan_sfb}, \eqref{eq:an_sfb_prob}, and the continuous mapping theorem \citep[see][Theorems~2.3 and 2.7]{van2000asymptotic}. 

In light of \eqref{eq:LeCam3rdLem_hyp}, Le Cam's Third Lemma asserts
\begin{equation}\label{eq:LeCam3rdLem_conc}
	\sqrt{k}(\bar{x}_k - x^{\star}) \overset{\smash{u}}{\leadsto} \mathsf{N}\big({-}W^{-1}\Sigma u, W^{-1}\Sigma W^{-\top}\big)
\end{equation}
\citep[see][Example~6.7]{van2000asymptotic}. On the other hand, Lemma~\ref{lem:smooth-kappa} shows that $\Gamma_k$ is a regular mapping sequence with derivative $\dot{\Gamma} = -W^{{-1}} \Sigma$ at zero, so 
\begin{equation}\label{eq:regular_eq_estimator}
	\sqrt{k}\big(x^\star - \Gamma_k(u)\big) = -\sqrt{k}\big(\Gamma_k(u) - \Gamma_k(0)\big) \rightarrow W^{-1}\Sigma u \qquad \text{as~} k\rightarrow\infty.
\end{equation} 
Combining \eqref{eq:LeCam3rdLem_conc} and \eqref{eq:regular_eq_estimator} yields \eqref{eq:equivariance}.
\end{proof}

 Finally, suppose that the assumptions of Lemma~\ref{lem:equivariance} hold. Let $\varphi\colon\R^d\rightarrow\R$ be any bounded continuous function and $Z\sim \mathsf{N}(0, W^{-1} \Sigma W^{-\top})$. Then \eqref{eq:equivariance} directly implies that for every finite subset $\cI\subset\R^d$, we have
 \begin{equation*}
 	\lim_{k \rightarrow \infty} \max_{u\in \cI} \,{\EE}_{P_{k,{\smash{u/\sqrt{k}}}}}\big[\varphi \big(\sqrt{k} \big(\bar{x}_k - \xs_{\smash{u/\sqrt{k}}}\big)\big)\big] = \max_{u\in \cI}\lim_{k \rightarrow \infty}{\EE}_{Q_{k,u}}\!\big[\varphi \big(\sqrt{k} \big(\bar{x}_{k} - \Gamma_k(u)\big)\big)\big] = {\EE}[\varphi(Z)].
 \end{equation*}
Hence
 \begin{align*}
 	\sup_{\cI \subset \R^d,\, |\cI| < \infty}\liminf_{k \rightarrow \infty} \max_{u\in \cI} \,{\EE}_{P_{k,{\smash{u/\sqrt{k}}}}}\big[\varphi \big(\sqrt{k} \big(\bar x_{k} - \xs_{\smash{u/\sqrt{k}}}\big)\big)\big] = {\EE}[\varphi(Z)],
 \end{align*}
thereby demonstrating equality in \eqref{eq:minimaxlb} whenever $\cL$ is bounded and continuous. The proof of Theorem~\ref{thm:optimality} is complete.

\acks{Research of Drusvyatskiy was supported by the NSF DMS 1651851 and CCF 1740551 awards. We thank the anonymous reviewers for their insightful comments.}

\appendix
\section{Proofs Deferred from Sections~\ref{sec:performative} and \ref{sec:normal}}\label{sec:missing_proofs_prelim}
\subsection{Proof of Lemma~\ref{lem:dev_game_jac}}\label{sec:proof_lem_dev}
For any $x,x',y\in\cX$, we successively estimate
\begin{align}
	\|G_{x}(y)-G_{x'}(y)\| & =\left\|\mathop{\EE}_{z\sim \mathcal{D}(x)} G(y,z)-\mathop{\EE}_{z\sim \mathcal{D}(x')}G(y,z)\right\|\notag                                                    \\
	                       & =\sup_{\|v\|\leq 1}\left\{\mathop{\EE}_{z\sim \mathcal{D}(x)}\langle G(y,z),v\rangle-\mathop{\EE}_{z\sim \mathcal{D}(x')}\langle G(y,z),v\rangle\right\}\notag \\
	                       & \leq  \beta \cdot W_1\big(\mathcal{D}(x),\mathcal{D}(x')\big)\label{eqn:simple_ineq}                                                                           \\
	                       & \leq  \beta\gamma \cdot \|x-x'\|,\notag
\end{align}
where inequality \eqref{eqn:simple_ineq} follows from the $\beta$-Lipschitz continuity of the function $z\mapsto \langle G(y,z),v\rangle$ and the characterization \eqref{eq:W1} of $W_1$.

\subsection{Proof of Theorem~\ref{thm:existence}}\label{sec:proof_thm_exist}
Fix any two points $x,x'\in\cX$ and set
$y:=\mathtt{Sol}(x)$ and $y':=\mathtt{Sol}(x')$. Note that the definition of the normal cone implies
$$\langle G_{x}(y),y-y'\rangle\leq 0\qquad \textrm{and}\qquad \langle G_{x'}(y'),y'-y\rangle\leq 0.$$
Strong monotonicity therefore ensures
\begin{align*}
	\alpha\|y-y'\|^2 & \leq  \langle G_{x}(y)-G_{x}(y'),y-y'\rangle   \\
	                 & \leq  \langle G_{x'}(y')-G_{x}(y'),y-y'\rangle \\
	                 & \leq \| G_{x'}(y')-G_{x}(y')\|\cdot\|y-y'\|\   \\
	                 & \leq \gamma\beta\|x-x'\|\cdot \|y-y'\|,
\end{align*}
where the last inequality follows from Lemma~\ref{lem:dev_game_jac}.
Dividing through by $\alpha\|y-y'\|$ guarantees that $\mathtt{Sol}(\cdot)$ is indeed a
contraction on $\cX$ with parameter $\frac{\gamma\beta}{\alpha}$. The result follows immediately from the Banach fixed point theorem.

\subsection{Proof of Proposition~\ref{prop:asconv}} \label{sec:proof_of_simple_prop}

We will use the following classical result known as the Robbins-Siegmund almost supermartingale convergence theorem \cite[for a proof, see][Theorem 1.3.12]{duflo97}.
\begin{lemma}[Robbins-Siegmund]\label{lem:rs}
	Let $(A_t),(B_t),(C_t),(D_t)$ be sequences of finite nonnegative random variables on a filtered probability space $(\Omega,\mathcal{F},\mathbb{F},\mathbb{P})$ adapted to the filtration $\mathbb{F}=(\mathcal{F}_t)$ and satisfying
	$${\Expect}[A_{t+1}\,|\,\mathcal{F}_t]\leq(1+B_t)A_t+C_t-D_t$$
	for all $t$. Then on the event $\left\{\sum_t B_t<\infty,\sum_t C_t<\infty\right\}$, there is a finite random variable $A_\infty$ such that $A_t\rightarrow A_\infty$ and $\sum_t D_t < \infty$ almost surely.
\end{lemma}

Toward applying Lemma~\ref{lem:rs} with $A_t = \|x_t - \xs\|^2$, let $(\mathcal{F}_t)$ be the filtration given by \eqref{eq:filtration} and observe that the SFB iterate sequence $(x_t)$ is given by
$$
	x_{t+1}=\proj_{\mathcal{X}}\big(x_t-\eta_t(R(x_t) + \xi_t) \big),
$$
where the map $R\colon\cX\rightarrow\R^d$ given by $R(x) = G_x(x)$ is Lipschitz continuous and strongly monotone on $\cX$ with constants $\bar{L} + \gamma\beta$ and $\bar\alpha = \alpha - \gamma\beta$, respectively (see Lemma~\ref{lem:lip}), and the noise vector $\xi_t = G(x_t, z_t) - R(x_t)$ satisfies $\EE[\xi_{t} \,|\, \cF_{t}] = 0$ (zero bias) with variance bound $\EE[\|\xi_{t}\|^2 \,|\, \cF_{t}]\leq K(1 + \|x_t - \xs\|^2)$ for all $t\geq0$ (Assumption~\ref{assmp:stoc}).
Thus, since $\eta_t\rightarrow 0$ (recall $\sum_{t}\eta_{t}^{2} <\infty$), we see that
for all sufficiently large $t$, we may apply the one-step improvement bound of \citet[Theorem 24]{narang2022multiplayer} with zero bias to obtain
\begin{align}
	{\EE}\big[\|x_{t+1} - \xs\|^2\mid\cF_t\big] & \leq \frac{1+2K\eta_t^2}{1+\bar{\alpha}\eta_t} \|x_t - \xs\|^2 + \frac{2K\eta_t^2}{1+\bar{\alpha}\eta_t} \label{eq:converge-please-1} \\ &\leq  (1 - \tfrac{1}{2}\bar{\alpha}\eta_t)\|x_t - \xs\|^2 + \frac{2K\eta_t^2}{1+\bar{\alpha}\eta_t}  \label{eq:converge-please-2}\\
	                                            & \leq  \|x_t - \xs\|^2+2K\eta^2_t-\tfrac{1}{2}\bar\alpha\eta_t \|x_t - \xs\|^2. \label{eq:converge-please-3}
\end{align}
(For \eqref{eq:converge-please-1}, it suffices to require $\eta_t \leq \frac{\bar\alpha}{2(\bar{L} + \gamma\beta)^2}$; for \eqref{eq:converge-please-2}, it suffices to require $\eta_t \leq \frac{\bar\alpha}{4K + \bar{\alpha}^2}$.)

Using \eqref{eq:converge-please-3}, we may now apply Lemma~\ref{lem:rs} with $A_{t} = \|x_t - \xs\|^2$, $B_{t}=0$, $C_{t} = 2K\eta_{t}^{2}$, and $D_{t}= \tfrac{1}{2}\bar\alpha \eta_{t}\|x_t - \xs\|^2$. By assumption, we have $\sum_{t}\eta_{t}^{2} <\infty$, so Lemma~\ref{lem:rs} yields a finite random variable $A_\infty$ such that $A_t\rightarrow A_{\infty}$ and $\sum_t D_{t} <  \infty$ almost surely. Hence $\|x_{t} - \xs\|^2 \rightarrow A_{\infty}$ and $\sum_t \eta_t\|x_{t} - \xs\|^2  <  \infty$ almost surely. Since $\sum_{t} \eta_{t} = \infty$, we conclude $A_{\infty} = \lim_{t}\|x_{t} - \xs\|^2   = 0$ almost surely, i.e., $x_t \to \xs$ almost surely.

Next, to establish the in-expectation rate, note that \eqref{eq:converge-please-3} and the tower rule imply
$$
	{\EE}\|x_{t+1} - \xs\|^2 \leq \big(1 - \tfrac{1}{2}\bar{\alpha}\eta_{t}\big){\EE}\|x_t - \xs\|^2 + 2K\eta_{t}^{2}
$$
for all sufficiently large $t$. Thus, upon supposing $\eta_{t} = \Theta(t^{-\nu})$ for some $\nu\in\big(\tfrac{1}{2}, 1\big)$, a standard inductive argument \cite[see, e.g.,][Lemma 3.11.8]{davis2021subgradient} yields a constant $C>0$ such that $ {\EE} \|x_t - \xs\|^2 \leq Ct^{-\nu}$ for all $t\geq 1$. Therefore
$$
	{\EE}\Bigg[\sum_{t=1}^{\infty}t^{-1/2}\|x_t - \xs\|^2 \Bigg] \leq C\sum_{t=1}^{\infty}t^{-(\nu+1/2)} < \infty
$$
and hence $\sum_{t=1}^{\infty}t^{-1/2}\|x_t - \xs\|^2 < \infty$ almost surely. This completes the proof.

\section{Review of Asymptotic Normality}\label{sec:rev_assymptot_norm}
In this appendix, we present a variation of the asymptotic normality result of \citet[Theorem~2]{polyak92}.  Consider a measurable set $\cX\subset\R^d$ and a measurable map $R\colon \cX \rightarrow \RR^{d}$. Suppose that there exists a solution $x^{\star} \in \cX$ to the equation $R(x) = 0$.
The goal is to approximate $x^\star$ while only having access to noisy evaluations of $R$. Given $x_0\in\X$, consider the iterative process
\begin{equation}
	\label{eq:pol}
	x_{t+1} = x_{t} - \eta_{t} \big(R(x_{t}) + \xi_{t}+\zeta_t\big),
\end{equation}
where $\eta_t$ is a deterministic positive step size, $\xi_{t}$ is a random vector in $\R^d$ representing noise with zero mean conditioned on prior information, and $\zeta_t$ is a random vector in $\R^d$ representing a residual element that both ensures $x_{t+1}\in\X$ and quantifies the difference between $x_{t+1}$ and the basic step $x_{t} - \eta_{t}(R(x_{t}) + \xi_{t})$ in the unbiased direction $-(R(x_t) + \xi_t)$; for example, taking $$\zeta_{t} = \frac{x_{t} - \eta_{t}\big(R(x_{t}) + \xi_{t}\big) - \proj_{\cX}\big(x_t-\eta_t\big(R(x_{t}) + \xi_{t}\big)\big)}{\eta_t}$$ in \eqref{eq:pol} yields the stochastic forward-backward method $x_{t+1}=\proj_{\mathcal{X}}(x_t-\eta_t(R(x_t) + \xi_t))$.

The following assumption formalizes the stochastic framework for our analysis.

\begin{assumption_ap}[Stochastic framework]\label{assmp:polyakstoc}
	The sequences $(x_{t})_{t\geq0}$, $(\xi_{t})_{t\geq0}$,  and $(\zeta_t)_{t\geq 0}$ in \eqref{eq:pol} are stochastic processes defined on a probability space $(\Omega, \cF, \mathbb{P})$ equipped with a filtration $(\cF_t)_{t\geq 0}$ such that $x_t$ is $\cF_t$-measurable, $\xi_t$ and $\zeta_t$ are $\mathcal{F}_{t+1}$-measurable, and $\xi_t$ constitutes a martingale difference sequence satisfying  $\EE[\xi_{t} \,|\, \cF_{t}] = 0$. Additionally, the following four conditions hold.
	\begin{enumerate}[label=(\roman*)]
		\item \label{cond1polyakstoc} \textbf{($L^2$-bounded noise)} $\sup_{t\geq0} {\EE}\|\xi_{t}\|^{2} < \infty$.
\item \label{cond2polyakstoc} \textbf{(Asymptotic covariance)} There is a deterministic positive semidefinite matrix $\Sigma$ satisfying $$\frac{1}{t}\sum_{i=0}^{t-1}{\EE}\big[\xi_{i}\xi_{i}^{\top}\,|\, \cF_{i}\big] \xlongrightarrow{p} \Sigma \qquad\text{as~}t\to\infty.$$
		\item \label{cond3polyakstoc} \textbf{(Lindeberg's condition)} For all $\varepsilon>0$,
		      $$
			      \frac{1}{t}\sum_{i=0}^{t-1}{\EE}\big[ \|\xi_{i}\|^2\mathbf{1}_{\smash{\{\|\xi_i\|\geq \varepsilon \sqrt{t}\}}} \,|\, \cF_{i}\big]\xlongrightarrow{p} 0 \qquad\text{as~}t\to\infty.
		      $$
		\item \label{cond4polyakstoc} \textbf{(Negligible residual)} $\frac{1}{\sqrt{t}}\sum_{i=0}^{t-1}\|\zeta_i\| \xlongrightarrow{p}0$ as $t\to\infty$.
	\end{enumerate}
\end{assumption_ap}

Next, we stipulate the stability conditions regulating the dynamics of \eqref{eq:pol} that we require to establish asymptotic normality of the average iterates. Recall that a matrix $A\in\R^{d\times d}$ is said to be \emph{positively stable} if every eigenvalue of $A$ has a positive real part.

\begin{assumption_ap}[Stable dynamics]\label{assmp:polyak_lin_conv} There is a positively stable matrix $A\in\R^{d\times d}$ for which the following two conditions hold.
	\begin{enumerate}[label=(\roman*)]
		\item \label{assmp:polyak_step_size} \textbf{(Step size)} The step size sequence $(\eta_{t})_{t\geq0}$ satisfies either
		      \begin{align}
			      \eta_t\equiv\eta \qquad & \text{and}\qquad 0<\eta< 2\Big(\min_{j} \realp \lambda_j (A)\Big)^{-1} \label{eq:stepconst}
			      \intertext{or}
			      \eta_t=o(1)\qquad       & \text{and}\qquad \frac{\eta_{t}-\eta_{t+1}}{\eta_{t}} = o(\eta_t) \qquad\text{as~}t\to\infty. 	\label{eq:stepdecay}
		      \end{align}
		\item \label{assmp:polyaksmooth} \textbf{(Linear approximation)} The iterate sequence $(x_{t})_{t\geq0}$ satisfies
		      \begin{equation}\label{eq:R_lin_approx_0}
			      \frac{1}{\sqrt{t}}\sum_{i=0}^{t-1} \|R(x_i) - A(x_i - \xs)\| \xlongrightarrow{p} 0 \qquad\text{as~}t\to\infty.
		      \end{equation}
	\end{enumerate}
\end{assumption_ap}
\begin{theorem}[{\citet[Theorem~2]{polyak92}}]\label{thm:polyak}
	Suppose that Assumptions \ref{assmp:polyakstoc} and \ref{assmp:polyak_lin_conv} hold.
	Then, as $t\to\infty$, the average iterates $\bar x_{t} = \frac{1}{t}\sum_{i=1}^{t} x_i$ satisfy
	\begin{equation*}
		\sqrt{t}(\bar x_{t} - x^{\star}) = -A^{-1}\bigg(\frac{1}{\sqrt{t}}\sum_{i=0}^{t-1}\xi_i\bigg) + o_{\PP}(1)
	\end{equation*}
	and hence
	\begin{equation*}
		\sqrt{t}(\bar x_{t} - x^{\star})\leadsto\mathsf{N}\big(0, A^{-1}\Sigma A^{-\top}\big).
	\end{equation*}
\end{theorem}

We remark that the assumptions of Theorem~\ref{thm:polyak} are somewhat more general than those of Theorem 2 of \cite{polyak92}, but the proof technique is the same. The primary differences are as follows:
\begin{enumerate}[label=(\alph*)]
	\item The residual term $\zeta_{t}$ in \eqref{eq:pol} need not satisfy ${\EE}[\zeta_t\,|\,\cF_t]=0$, but this causes no difficulty as we assume $\zeta_{t}$ is negligible in the sense of condition~\ref{cond4polyakstoc} of Assumption~\ref{assmp:polyakstoc}. The rest of our stochastic setting stipulates conditions on $\xi_{t}$ tailored to an application of the martingale central limit theorem (Theorem~\ref{thm:clt}); we note that Lindeberg's condition~\ref{cond3polyakstoc} of Assumption~\ref{assmp:polyakstoc} holds if the asymptotic uniform integrability condition $\limsup_{t\to\infty}{\EE}\big[ \|\xi_{t}\|^2\mathbf{1}_{\smash{\{\|\xi_t\|\geq N\}}} \,|\, \cF_{t}\big]\xlongrightarrow{p} 0 $ as $N\to\infty$ is fulfilled and $\sup_{t\geq0} {\EE}\big[ \|\xi_{t}\|^2 \,|\, \cF_{t}\big] < \infty$ almost surely.
	\item Theorem 2 of \cite{polyak92} requires $A = \nabla R (\xs)$ with
	      \begin{equation}\label{eq:R_lin_approx_1}
		      R(x) - \nabla R(\xs)(x-\xs) = O(\|x-\xs\|^{q})\qquad\text{as } x\rightarrow \xs
	      \end{equation}
	      for some $q\in(1,2]$, and assumes that the step size sequence $(\eta_{t})_{t\geq0}$ satisfies
	      \begin{equation*}
		      \sum_{t=1}^{\infty}\eta_{t}^{q/2} t^{-1/2}<\infty
	      \end{equation*}
	      in addition to \eqref{eq:stepdecay}; together with a further Lyapunov function assumption, this suffices to demonstrate that the iterate sequence $(x_{t})_{t\geq0}$ satisfies both $x_t \xlongrightarrow{\text{a.s.}} x^\star$ and $\frac{1}{\sqrt{t}}\sum_{i=0}^{t-1}\|x_i - x^\star\|^q \xlongrightarrow{\text{a.s.}} 0$ as $t\to\infty$, which by \eqref{eq:R_lin_approx_1} implies \eqref{eq:R_lin_approx_0}.
\end{enumerate}
\begin{proof}
	For each $t\geq0$, let $\Delta_t=x_t-x^\star$ denote the error of the process \eqref{eq:pol} at time $t$, with corresponding average errors given by
	$$\bar\Delta_t=\frac{1}{t}\sum_{j=1}^{t} \Delta_{j} = \bar{x}_t-x^\star\qquad\text{for all~}t\geq1.$$
	Let $A$ denote the matrix furnished by Assumption~\ref{assmp:polyak_lin_conv} and observe that \eqref{eq:pol} yields the following recursion for all $t\geq0$:
	\begin{align}
		\Delta_{t+1} & = \Delta_{t} - \eta_{t} \big(R(x_{t}) + \xi_{t}+\zeta_t\big) \notag                                     \\
		             & = (I- \eta_t A)\Delta_{t} - \eta_{t}\big(R(x_t) - A\Delta_t + \xi_{t}+\zeta_t\big).\label{eq:pj_recurs}
	\end{align}
	Unrolling the recursion \eqref{eq:pj_recurs} gives
	\begin{equation*}
		\Delta_{j} = \Bigg(\prod_{k=0}^{j-1}(I-\eta_k A)\!\Bigg)\Delta_0 - \sum_{i=0}^{j-1}\Bigg(\prod_{k=i+1}^{j-1}(I-\eta_{k})A\!\Bigg)\eta_{i}\big(R(x_i) - A\Delta_i + \xi_{i}+\zeta_i\big)
	\end{equation*}
	for all $j\geq0$ and hence
	\begin{align*}
		t\bar{\Delta}_t & = \sum_{j=1}^{t}\Bigg(\prod_{k=0}^{j-1}(I-\eta_k A)\!\Bigg)\Delta_0 - \sum_{j=1}^{t}\sum_{i=0}^{j-1}\Bigg(\prod_{k=i+1}^{j-1}(I-\eta_{k}A)\!\Bigg)\eta_{i}\big(R(x_i) - A\Delta_i + \xi_{i}+\zeta_i\big)        \\
		                & = \sum_{j=1}^{t}\Bigg(\prod_{k=0}^{j-1}(I-\eta_{k} A)\!\Bigg)\Delta_0 - \sum_{i=0}^{t-1}\sum_{j=i+1}^{t}\!\Bigg(\prod_{k=i+1}^{j-1}(I-\eta_{k} A)\!\Bigg)\eta_{i}\big(R(x_i) - A\Delta_i + \xi_{i}+\zeta_i\big)
	\end{align*}
	for all $t\geq1$ (interpreting empty products as the identity matrix and empty sums as zero). Thus, upon defining for each $t\geq1$ and $i\geq0$ the matrices
	\begin{equation*}
		B_t = \sum_{j=1}^{t}\Bigg(\prod_{k=0}^{j-1}(I-\eta_{k}A)\!\Bigg), \qquad B_i^t = \eta_i\!\sum_{j=i+1}^{t}\!\Bigg(\prod_{k=i+1}^{j-1}(I-\eta_{k}A)\!\Bigg),\qquad A_i^t = B_i^t - A^{-1},
	\end{equation*}
	we have
	\begin{align*}
		t\bar{\Delta}_t & = B_t\Delta_0 - \sum_{i=0}^{t-1}B_i^t \big(R(x_i) - A\Delta_i + \xi_{i}+\zeta_i\big)                                                                          \\
		                & = B_t\Delta_0 -\sum_{i=0}^{t-1}B_i^t \xi_i - \sum_{i=0}^{t-1}B_i^t \big(R(x_i) - A\Delta_i\big) - \sum_{i=0}^{t-1}B_i^t \zeta_i                               \\
		                & = B_t\Delta_0 - A^{-1}\sum_{i=0}^{t-1}\xi_i - \sum_{i=0}^{t-1}A_i^t\xi_i - \sum_{i=0}^{t-1}B_i^t \big(R(x_i) - A\Delta_i\big) - \sum_{i=0}^{t-1}B_i^t \zeta_i
	\end{align*}
	and hence
	\begin{equation}\label{eq:polyak_decomp}
		\begin{split}
			\sqrt{t}(\bar x_{t} - x^{\star}) + A^{-1}\bigg(\!\frac{1}{\sqrt{t}}\sum_{i=0}^{t-1}\xi_i\!\bigg) &= \frac{1}{\sqrt{t}}B_t(x_0 - \xs) - \frac{1}{\sqrt{t}}\sum_{i=0}^{t-1}A_i^t\xi_i \\
			&~~~~{-}\frac{1}{\sqrt{t}}\sum_{i=0}^{t-1}B_i^t \big(R(x_i) - A(x_i - \xs)\big)
			- \frac{1}{\sqrt{t}}\sum_{i=0}^{t-1}B_i^t \zeta_i.
		\end{split}
	\end{equation}

	We claim that the right-hand side of \eqref{eq:polyak_decomp} is $o_{\PP}(1)$ as $t\to\infty$. Indeed, since $A$ is positively stable and the step size condition~\ref{assmp:polyak_step_size} of Assumption~\ref{assmp:polyak_lin_conv} holds, it follows from Lemma~1 of \cite{polyak92} that the collection of matrices $\{A_i^t, B_i^t, B_t \mid t\geq 1, i\geq 0\}$ is bounded with respect to the operator norm and
	\begin{equation}\label{eq:polyak_lemma_1}
		\lim_{t\to\infty}\frac{1}{t}\sum_{i=0}^{t-1}\|A_i^t\|_\text{op} = 0.
	\end{equation}
	Let $C = \sup\{\|A_i^t\|_{\text{op}}, \|B_i^t\|_{\text{op}}, \|B_t\|_{\text{op}}, {\EE}\|\xi_i\|^2 \mid t\geq 1, i\geq 0\}$; by the $L^2$-boundedness condition~\ref{cond1polyakstoc} of Assumption~\ref{assmp:polyakstoc}, we have $C<\infty$. Therefore
	\begin{equation}\label{eq:op1_rhs1}
		\left\|\frac{1}{\sqrt{t}}B_t(x_0 - \xs)\right\| \leq \frac{C\|x_0 - \xs\|}{\sqrt{t}} \xlongrightarrow{\text{a.s.}} 0 \qquad\text{as~}t\to\infty,
	\end{equation}
	and since $(\xi_{i})_{i\geq0}$ is a martingale difference sequence, we deduce from \eqref{eq:polyak_lemma_1} the following convergence in mean square:
	\begin{equation*}
		{\EE}\left\|\frac{1}{\sqrt{t}}\sum_{i=0}^{t-1}A_i^t\xi_i\right\|^2  = \frac{1}{t}\sum_{i=0}^{t-1}{\EE}\|A_i^t\xi_i\|^2 \leq \frac{C}{t}\sum_{i=0}^{t-1}\|A_i^t\|^{2}_\text{op} \leq \frac{C^2}{t}\sum_{i=0}^{t-1}\|A_i^t\|_\text{op} \to 0 \qquad\text{as~}t\to\infty,
	\end{equation*}
	which by Markov's inequality implies
	\begin{equation}\label{eq:op1_rhs2}
		\frac{1}{\sqrt{t}}\sum_{i=0}^{t-1}A_i^t\xi_i \xlongrightarrow{p} 0 \qquad\text{as~}t\to\infty.
	\end{equation}
	Moreover, the linear approximation condition~\ref{assmp:polyaksmooth} of Assumption~\ref{assmp:polyak_lin_conv} implies
	\begin{equation}\label{eq:op1_rhs3}
		\left\|\frac{1}{\sqrt{t}}\sum_{i=0}^{t-1}B_i^t \big(R(x_i) - A(x_i - \xs)\big)\right\| \leq \frac{C}{\sqrt{t}}\sum_{i=0}^{t-1} \|R(x_i) - A(x_i - \xs)\| \xlongrightarrow{p} 0 \qquad\text{as~}t\to\infty,
	\end{equation}
	while the negligible residual condition~\ref{cond4polyakstoc} of Assumption~\ref{assmp:polyakstoc} implies
	\begin{equation}\label{eq:op1_rhs4}
		\left\|\frac{1}{\sqrt{t}}\sum_{i=0}^{t-1}B_i^t \zeta_i \right\| \leq \frac{C}{\sqrt{t}}\sum_{i=0}^{t-1} \|\zeta_i\| \xlongrightarrow{p} 0 \qquad\text{as~}t\to\infty.
	\end{equation}
	By \eqref{eq:op1_rhs1}--\eqref{eq:op1_rhs4}, we conclude that the right-hand side of \eqref{eq:polyak_decomp} is $o_{\PP}(1)$ as $t\to\infty$, so
	\begin{equation*}
		\sqrt{t}(\bar x_{t} - x^{\star}) = -A^{-1}\bigg(\frac{1}{\sqrt{t}}\sum_{i=0}^{t-1}\xi_i\bigg) + o_{\PP}(1) \qquad\text{as~}t\to\infty.
	\end{equation*}

	Finally, by virtue of Assumption~\ref{assmp:polyakstoc}, we may apply the martingale central limit theorem (Theorem~\ref{thm:clt}) to the square-integrable martingale $M_t = \sum_{i=0}^{t-1}\xi_i$ to obtain $t^{-1/2}M_t\leadsto\mathsf{N}(0,\Sigma)$ and hence, by the continuous mapping theorem \cite[see][Theorem~2.3]{van2000asymptotic},
	\begin{equation*}
		-A^{-1}\bigg(\!\frac{1}{\sqrt{t}}\sum_{i=0}^{t-1}\xi_i\!\bigg)\leadsto\mathsf{N}\big(0, A^{-1}\Sigma A^{-\top}\big) \qquad\text{as~}t\to\infty.
	\end{equation*}
	This completes the proof.
\end{proof}

\section{Proofs Deferred from Section~\ref{sec:optimality}}

This appendix presents contains all of the proofs deferred from Section~\ref{sec:optimality}. We assume throughout that the assumptions used in Section~\ref{sec:optimality} are valid; in particular, $\X$ is compact, $\Z$ is bounded, and $g\in\cG$ (see Definition~\ref{def:G_n}). To begin, we present three preliminary lemmas.

\begin{lemma}\label{lem:ultra-bounded}
We have
\begin{equation*}\label{eq:ultra-bounded}
	\sup_{x \in \cX,\, z \in \cZ} \|g_{x}(z)\|<\infty \qquad\text{and}\qquad  \sup_{x \in \cX,\, z \in \cZ} \|\nabla_x g_{x}(z)\|_\textup{op}<\infty.
\end{equation*}
\end{lemma}
  \begin{proof}
    Fix $x^{\circ} \in \cX$ and $z^\circ \in \cZ$. Since $\cX$ and $\cZ$ are bounded, we compute
    \begin{align*}
      M'_{g} := \sup_{x \in \cX,\, z \in \cZ} \|\nabla_x g_{x} (z)\|_\text{op} &\leq
      \begin{aligned}[t]
      \|\nabla_x g_{x^{\circ}} (z^{\circ})\|_\text{op} &+ \sup_{x \in \cX}\|\nabla_x g_{x } (z^{\circ}) -  \nabla_x g_{ x^{\circ}} (z^{\circ})\|_\text{op} \\ &+ \sup_{x \in \cX,\, z \in \cZ}\|\nabla_x g_{x} (z) - \nabla_x g_{x} (z^{\circ})\|_\text{op}
      \end{aligned}
 \\
                 &\leq \|\nabla_x g_{x^{\circ}} (z^{\circ})\|_\text{op} + \Lambda_g(z^{\circ})\diam(\cX)  + \beta'_{g}\diam(\cZ)<\infty.
    \end{align*}
 Hence every section $g(\cdot, z)$ is $M'_{g}$-Lipschitz on $\cX$, and the estimate
      \begin{align*}
      M_{g} :=\sup_{x \in \cX,\, z \in \cZ} \|g_{x} (z)\| &\leq \|g_{x^{\circ}} (z^{\circ})\| +  \sup_{x \in \cX}\|g_{x } (z^{\circ}) -  g_{ x^{\circ}} (z^{\circ})\| + \sup_{x \in \cX,\, z \in \cZ}\|g_{x} (z) - g_{x} (z^{\circ})\| \\
                 &\leq \|g_{x^{\circ}} (z^{\circ})\| +  M'_{g}\diam(\cX)  + \beta_g\diam(\cZ)
    \end{align*}
completes the proof. \end{proof}

\clearpage
  \begin{lemma}\label{lem: uniform C(x,u)}
   Let $L_{h}=\sup |h'|$, $L_{h''}=\sup |h'''|$, $A_{g} = \sup_{x \in \cX}\EE_{z\sim\D_x}\!\|g_{x} (z)\|$, and $B_{g} = \sup_{x \in \cX}\EE_{z\sim\D_x}\!\|g_{x} (z)\|^3$. 
   \begin{enumerate}[label=(\roman*)]  
   \item Let $u\in\R^d$. Then
   \begin{equation}\label{eq:h-bound-0}
	|h(u^{\top}g_x(z))|\leq L_{h}\|g_{x}(z)\|\|u\|
	\end{equation}  
	for all $x\in\cX$ and $z\in\cZ$ and hence 
   \begin{align}\label{eq:h-bound-1}
	\sup_{x\in\cX} {\mathop{\EE}_{z \sim \cD_{x}}}|h(u^{\top}g_x(z))|\leq L_hA_g\|u\| = O(\|u\|).
	\end{align}	
	\item For each $x\in \cX$, the function $u\mapsto C_x^u = 1 + {\EE}_{z\sim \D_x} h(u^\top g_{x}(z))$ is $C^{2}$-smooth on $\R^d$ with $L_{h''} B_{g}$-Lipschitz continuous Hessian, and we have $C_x^0=1$, $\nabla_uC_x^u|_{u=0}=0$, and $\nabla_{uu}^{2}C_x^u|_{u=0}=0$. Therefore 
	\begin{equation}\label{eq:C-bound-1}
		\sup_{x\in\mathcal{X}}\left|{\EE}_{z\sim \D_x} h(u^{\top}g_{x}(z))\right| \leq \frac{L_{h''} B_{g}}{6}\|u\|^3
	\end{equation}
	for all $u\in\R^d$ and hence
   \begin{equation}\label{eq:C-little-o}
	\sup_{x\in\cX}\frac{1}{C_x^u} = 1 + O(\|u\|^3)\qquad\text{as~}u\rightarrow0.
	\end{equation}
   \end{enumerate}
\end{lemma}
\begin{proof}
Note first that $h(t)=t$ for all $t$ in a neighborhood of zero and the first three derivatives of $h$ are bounded by assumption, while $A_g, B_g < \infty$ by Lemma~\ref{lem:ultra-bounded}. Since $h(0)=0$ and $h$ is $L_h$-Lipschitz continuous, the inequalites \eqref{eq:h-bound-0} and \eqref{eq:h-bound-1}  follow immediately. Next, let $x\in\cX$ and observe that the dominated convergence theorem yields
$$
\nabla_u\Big( \mathop{\EE}_{z\sim \D_x} h\big(u^{\top} g_{x}(z)\big) \Big) = \mathop{\EE}_{z\sim \D_x} h'\big(u^{\top}g_{x}(z)\big)g_{x}(z)
$$
and
$$
\nabla_{uu}^2\Big( \mathop{\EE}_{z\sim \D_x} h\big(u^{\top} g_{x}(z)\big) \Big) = \mathop{\EE}_{z\sim \D_x} h''\big(u^{\top} g_{x}(z)\big)g_{x}(z)g_{x}(z)^\top
$$
for all $u\in\R^d$. Thus, $u\mapsto C_{x}^{u}$ is $C^2$-smooth on $\RR^d$, and since $h''$ is $L_{h''}$-Lipschitz continuous, it follows at once that $u\mapsto \nabla_{uu}^2 C_{x}^{u}$  is $L_{h''} B_{g}$-Lipschitz continuous on $\R^d$.

Clearly $C_x^0=1$ since $h(0)=0$. Further,
$$
\nabla_uC_x^u|_{u=0} = \nabla_u\Big( \mathop{\EE}_{z\sim \D_x} h\big(u^{\top} g_{x}(z)\big) \Big)\Big|_{u=0} = \mathop{\EE}_{z\sim \D_x} g_{x}(z) = 0
$$
since $h'(0) = 1$ and $g\in\cG$, while 
$$
\nabla_{uu}^{2}C_x^u|_{u=0} = \nabla_{uu}^2\Big( \mathop{\EE}_{z\sim \D_x} h\big(u^{\top} g_{x}(z)\big) \Big)\Big|_{u=0} = 0
$$
since $h''(0) = 0$. The second-order Taylor polynomial of the function $u\mapsto{\EE}_{z\sim \D_x} h(u^{\top} g_{x}(z))$ about $u=0$ is therefore identically zero, so $L_{h''} B_{g}$-Lipschitzness of the Hessian implies \eqref{eq:C-bound-1}. Finally, the estimate
$$
\frac{1}{1+t} = 1 - \frac{t}{1+t} \leq 1 + 2|t|\qquad\text{~for all~}t\geq{-}\frac{1}{2}
$$
together with \eqref{eq:C-bound-1} yields $\eqref{eq:C-little-o}$.
\end{proof}

\begin{lemma}\label{lem:yet-another-bound}
  Let $f\colon (0,\infty)\rightarrow \RR$ be a function that is $C^{2,1}$-smooth around $t=1$ and satisfies $f(1) = 0.$ Then for all sufficiently small $u\in\R^d$ and all $x\in \cX$, we have
  \begin{equation}
    \label{eq:taylor-expanding-is-nice}
  \int f\bigg(\frac{1+u^{\top}g_{x}(z)}{C_{x}^{u}}\bigg)\,d\D_{x}(z)  = \frac{f''(1)}{2} u^{\top}\bigg(\mathop{\EE}_{z\sim \D_{x}}g_{x}(z)g_{x}(z)^{\top}\bigg)u + r_{x}(u),
  \end{equation}
 where $\sup_{x \in \cX}|r_{x}(u)| = O(\|u\|^{3})$ as $u\rightarrow 0$.
\end{lemma}
\begin{proof}
  Fix $x\in \cX$ and define $\varphi_{x}(u) := {\EE}_{z \sim \D_{x}} f\Big(\frac{1+u^{\top}g_{x}(z)}{C_{x}^{u}}\Big)$. By the dominated convergence theorem, $\varphi_{x}$ is $C^{2}$-smooth on a neighborhood of zero with
  $$\nabla_{u} \varphi_{x}(u) = {\mathop{\EE}_{z\sim \D_x}}\bigg[f'\bigg(\frac{1+ u^{\top}g_{x}(z)}{C_{x}^{u}}\bigg)\bigg(\frac{g_{x}(z) C_{x}^{u} - \big(1 + u^{\top}g_{x}(z)\big) \nabla_{u} C_{x}^{u}}{(C_{x}^{u})^{2}}\bigg)\bigg]$$
  and $(C_{x}^{u})^{4}\cdot\nabla_{uu}^{2} \varphi_{x}(u)$ equal to
\begin{align*}
&\!\mathop{\EE}_{z\sim \D_x}\!\left[f''\bigg(\frac{1+ u^{\top}g_{x}(z)}{C_{x}^{u}}\bigg)\Big(g_{x}(z) C_{x}^{u} - \big(1 + u^{\top}g_{x}(z)\big) \nabla_{u} C_{x}^{u} \Big)\Big(g_{x}(z) C_{x}^{u} - \big(1 + u^{\top}g_{x}(z)\big) \nabla_{u} C_{x}^{u}\Big)^{\top}\right. \\
    &\hspace{.95cm}+\left. f'\bigg(\frac{1+ u^{\top}g_{x}(z)}{C_{x}^{u}}\bigg)\bigg((C_{x}^{u})^{2}\Big( g_x(z)(\nabla_u C_x^u)^{\top} - (\nabla_u C_x^u)g_x(z)^{\top} -  \big(1+u^{\top}g_{x}(z)\big) \nabla^{2}_{uu} C_{x}^{u}\Big) \right. \\
    &\hspace{5cm}-\left.
     2C_{x}^{u}\Big( g_{x}(z) C_{x}^{u} - \big(1 + u^{\top} g_{x}(z)\big) \nabla_{u}C_{x}^{u}\Big) (\nabla_{u}C_{x}^{u})^{\top} \bigg)\right]\!.
  \end{align*}
  Thus, taking a second-order Taylor expansion of $\varphi_x$ at $u = 0$ with remainder $r_x$ and applying the equalities $C_x^0=1$, $\nabla_uC_x^u|_{u=0}=0$, $\nabla_{uu}^{2}C_x^u|_{u=0}=0$, and $f(1) = 0$ yields \eqref{eq:taylor-expanding-is-nice}. It remains to verify $\sup_{x \in \cX}|r_{x}(u)| = O(\|u\|^{3})$ as $u\rightarrow 0$.

  Lemmas~\ref{lem:ultra-bounded} and \ref{lem: uniform C(x,u)} ensure that $C_{x}^{u}, \nabla_u C_{x}^{u},$ and $\nabla_{uu}^{2} C_{x}^{u}$ are Lipschitz continuous and bounded on a compact neighborhood of $u=0$, with Lipschitz constants and bounds independent of $x$. Further, since $f$ is $C^{2,1}$-smooth around $t=1$, we have that $f'$ and $ f''$ are Lipschitz continuous and bounded on a compact neighborhood of $t=1$. It follows that $\nabla^{2}_{uu}\varphi_{x}$ is $\tilde L$-Lipschitz on a neighborhood $U$ of $u=0$, with constant $\tilde L$ independent of $x$. Thus we deduce
  $
  |r_{x}(u)| \leq \frac{\tilde L}{6}\|u\|^{3}
  $ for all $(x,u)\in \cX\times U$, and the result follows.
\end{proof}

\subsection{Proof of Lemma~\ref{lem:spain}}\label{sec:proof_of_admis}

The proof of this lemma is divided into four steps:
the first step verifies Assumption~\ref{assmp:lipdist}  and the next three steps establish Assumption~\ref{assmp:lipgrad}. The strategy in all steps is to prove that various quantities of interest change continuously with $u$ near zero. One of the main tools we will use to this end is the following elementary lemma (which we will also use crucially later in the proof of Lemma~\ref{lem:smooth-kappa}). Its proof consists of several applications of the dominated convergence theorem and is deferred to Section~\ref{proof:oh-the-smoothness}.

  \begin{lemma}[Inferring smoothness]\label{lem:oh-the-smoothness}
    Suppose that $T \colon \cX \times \cZ \rightarrow \RR^n$ is a map satisfying the following two conditions.
    \begin{enumerate}[label=(\roman*)]
      \item (\textbf{Lipschitz continuity}) There exists a constant $\beta_T\geq0$ such that for every $x\in\X$, the section $T(x,\cdot)$ is $\beta_T$-Lipschitz on $\cZ$.
      \item (\textbf{Smoothness}) There exist a measurable function $\Lambda_T\colon\cZ\rightarrow[0,\infty)$ and constants $\bar{\Lambda}_T, \beta'_T \geq 0$ such that for every $z \in \cZ$ and $x\in\X$, the section $T(\cdot, z)$ is $\Lambda_T(z)$-smooth on $\X$ with $\EE_{z \sim \D_{x}}[\Lambda_T(z)] \leq \bar{\Lambda}_T$, and the section $\nabla_x T(x,\cdot)$ is $\beta'_T$-Lipschitz on $\Z$.
\end{enumerate}
    Set $$M_T := \sup_{x\in\cX,\,z\in\cZ} \|T(x,z)\| \qquad\text{and}\qquad M'_T := \sup_{x\in\cX,\,z\in\cZ} \|\nabla_x T(x,z)\|_\textup{op}.$$ Then $M_T$ and $M'_{T}$ are finite. Moreover, given any fixed compact neighborhood $\cW\subset\R^d$ of zero, the maps $\bar H\colon \cX \times \cX \times \cW \rightarrow \RR^n$ and $H\colon \cX \times \cW \rightarrow \R^n$ given by
    $$
    \bar H(x,y, u) = {\mathop{\EE}_{z \sim \cD_{x}}} \big[ T(y, z) \big(1+ h\big(u^{\top} g_{y}(z)\big)\big)\big]\qquad\text{and}\qquad H(x, u) = \mathop{\EE}_{z \sim \cD_{x}^{u}} T(x, z)
    $$
    are smooth with Lipschitz continuous Jacobians with constants depending  on $T$ only through $\beta_T, \bar{\Lambda}_T, \beta'_T$, $M_T$, and $M'_T$; further, we have
    $$
    \nabla_x H(x, 0 ) = \nabla_x \Big( \mathop{\EE}_{z\sim \D_x}\!T(x,z)\Big)\qquad\text{and}\qquad \nabla_u H(x, 0) = {\mathop{\EE}_{z\sim \D_{x}}}\!\big[T(x,z)g_{x}(z)^{\top}\big]
    $$
    for all $x\in\cX$.
  \end{lemma}

  \paragraph{Step 1 (Assumption~\ref{assmp:lipdist})} First, we show that the perturbed distribution map $\D^{u}$ satisfies Assumption~\ref{assmp:lipdist} with Lipschitz constant $\gamma^u = \gamma + O(\|u\|)$ as $u\rightarrow0$, where $\gamma$ is the Lipschitz constant for $\cD$. To this end, we take $\cW$ to be the unit ball in $\R^d$ and apply Lemma~\ref{lem:oh-the-smoothness} to identify a constant $L_1\geq0$ such that for every $1$-Lipschitz function $\phi \in \text{Lip}_1(\Z)$ and every $u\in\cW$, the function 
  $$
  \rho_{\phi}(x, u) := \mathop{\EE}_{z \sim \cD_{x}^{u}} \phi(z) $$
  is Lipschitz in the $x$-component with constant $\gamma^u:=\gamma + L_1\|u\|$. Indeed, for every $\phi \in \text{Lip}_1(\Z)$, the translate $\bar{\phi}= \phi - \inf\phi$ is $1$-Lipschitz and bounded by $\diam(\cZ)$, and $\rho_{\bar{\phi}} = \rho_{\phi} - \inf\phi$. Thus, Lemma~\ref{lem:oh-the-smoothness} yields a constant $L_{1}$ such that for every $\phi \in \text{Lip}_1(\Z)$, the function $\rho_{\bar{\phi}}$ is $L_{1}$-smooth on $\cX\times\cW$ and hence so is $\rho_{\phi}$. Moreover, Lemma~\ref{lem:oh-the-smoothness} shows $$\nabla_{x} \rho_{\phi}(x, 0) = \nabla_{x}\Big( \mathop{\EE}_{z\sim \D_x}\!\phi(z)\Big)$$ for all $x\in\cX$, so $\sup_{x\in\cX}\|\nabla_{x} \rho_{\phi}(x, 0)\|\leq\gamma$ by Assumption~\ref{assmp:lipdist}. Thus, the triangle inequality yields 
  \begin{align*}
    \|\nabla_{x} \rho_{\phi}(x, u)\| &\leq \|\nabla_{x} \rho_{\phi}(x, 0)\| + \|\nabla_{x} \rho_{\phi}(x, u) - \nabla_{x} \rho_{\phi}(x, 0)\| 
    \leq \gamma + L_{1}\|u\| 
    = \gamma^u
  \end{align*}
for all $(x,u)\in\cX\times\cW$. Therefore $\rho_{\phi}(\cdot, u)$ is $\gamma^u$-Lipschitz on $\cX$ for all $\phi \in \text{Lip}_1(\Z)$ and $u\in\cW$, so $\cD^u$ satisfies Assumption~\ref{assmp:lipdist} with Lipschitz constant $\gamma^u = \gamma + O(\|u\|)$ as $u\rightarrow0$.

 \paragraph{Step 2 (Lipschitz continuity)} Next, we establish Assumption~\ref{assmp:lipgrad}\ref{assmp:lipgrad1} for the problem with the perturbed distribution map $\D^{u}$. Observe that the Lipschitz bounds in Assumption~\ref{assmp:lipgrad}\ref{assmp:lipgrad1} remain unchanged, and that we only need to identify for all sufficiently small $u$ a constant $\bar{L}^u\geq0$ such that $\sup_{x\in\cX}{\EE}_{z \sim \D_{x}^{u}}[L(z)^2]\leq (\bar{L}^{u})^2$. We will show more, namely, that we can select $(\bar{L}^u)^2 = \bar{L}^2 + O(\|u\|)$ as $u\rightarrow 0$, where $\bar{L}$ is the constant satisfying $\sup_{x\in\cX}{\EE}_{z \sim \D_{x}}[L(z)^2]\leq \bar{L}^2$ furnished by Assumption~\ref{assmp:lipgrad}\ref{assmp:lipgrad1}. Indeed, for all $x\in\cX$ and $u\in\R^d$, we have
 $$
 {\mathop{\EE}_{z \sim \cD_{x}^{u}}}\big[L(z)^2\big] = \frac{1}{C_{x}^{u}}  {\mathop{\EE}_{z \sim \cD_{x}}}\!\big[ L(z)^2 \big(1+ h\big(u^{\top} g_{x}(z)\big)\big)\big].
 $$
Thus, an application of Lemma~\ref{lem: uniform C(x,u)} yields
$$
  \sup_{x\in\cX} {\mathop{\EE}_{z \sim \cD_{x}^{u}}}\big[L(z)^2\big] \leq \big(1+O(\|u\|^3)\big)\big(1 + L_h M_g\|u\|\big)\bar{L}^2 = \bar{L}^2 + O(\|u\|)\qquad\text{as~}u\rightarrow0,
 $$
where $L_{h}=\sup |h'|$ and $M_{g} =\sup \|g\|$.

 \paragraph{Step 3 (Monotonicity)} We prove that for all $x\in\cX$, the map $G_{x}^{u}(\cdot)$ given by $$ G_{x}^{u}(y):=\underset{z \sim \D_{x}^{u}}{\EE}G(y,z)$$ is strongly monotone on $\cX$ with constant $\alpha^u = \alpha + O(\|u\|)$ as $u\rightarrow0$, where $\alpha$ is the strong monotonicity constant of $G_{x}(\cdot)$. Given $x\in\cX$ and $u\in\R^d$, we have
 \begin{align*}
 \big\langle G_x^u(y) - G_x^u(y'), y - y' \big\rangle &=  \big\langle G_x(y) - G_x(y'), y - y' \big\rangle \\
 &~~~~~~~~~~~~~+  \big\langle \big(G_x^u(y) - G_x(y)\big) - \big(G_x^u(y') - G_x(y')\big), y - y' \big\rangle \\
 &\geq \alpha\|y-y'\|^2 - \big\|\big(G_x^u(y) - G_x(y)\big) - \big(G_x^u(y') - G_x(y')\big) \big\|\cdot\big\|y - y'\big\|
 \end{align*}
 for all $y,y'\in\cX$ by the $\alpha$-strong monotonicity of $G_x(\cdot)$. We claim that for all sufficiently small $u$, there exists $\ell^u = O(\|u\|)$ independent of $x$ such that the map $y\mapsto G_x^u(y) - G_x(y)$ is $\ell^u$-Lipschitz on $\cX$ for all $x\in\cX$. Indeed, upon noting $\sup_{x\in\cX,\,z\in\cZ} \|\nabla_x G(x,z)\|_\textup{op}<\infty$ (see Lemma~\ref{lem:oh-the-smoothness}) and applying the dominated convergence theorem together with Lemma~\ref{lem: uniform C(x,u)}, we obtain
\begin{align*}
 \ell^u &:= \sup_{x,y\in\cX} \big\|\nabla_y\big(G_x^u(y)-G_x(y)\big)\big\|_\text{op} \\ &= \sup_{x,y\in\cX}\left\|\frac{1}{C_x^u}\underset{z \sim \D_{x}}{\EE} \big[\nabla_y G(y,z) \big(1 + h\big(u^{\top} g_x(z)\big)\big)\big]- \underset{z \sim \D_{x}}{\EE} \big[\nabla_y G(y,z)\big]\right\|_\text{op}\\
 &\leq \underbrace{\sup_{x,y\in\cX}\left\|\bigg(\frac{1}{C_x^u}-1\bigg)\underset{z \sim \D_{x}}{\EE} \big[\nabla_y G(y,z)\big]\right\|_\text{op}}_{O(\|u\|^3)} + \underbrace{\sup_{x,y\in\cX}\left\|\frac{1}{C_x^u}\underset{z \sim \D_{x}}{\EE} \big[\nabla_y G(y,z)h\big(u^{\top} g_x(z)\big)\big]\right\|_\text{op}}_{\left(1+O(\|u\|^3)\right)\,\cdot\, O(\|u\|)} \\
 &=O(\|u\|) \qquad\text{as~}u\rightarrow0.
\end{align*}
Setting $\alpha^u := \alpha - \ell^u$ for all $u$ in a neighborhood of zero, we conclude that for all $x,y,y'\in\cX$, 
$$
\big\langle G_x^u(y) - G_x^u(y'), y - y' \big\rangle \geq \alpha^u\|y-y'\|^2
$$
and hence $G^{u}_{x}(\cdot)$ is strongly monotone on $\cX$ with constant $\alpha^u = \alpha + O(\|u\|)$ as $u\rightarrow0$.

\paragraph{Step 4 (Compatibility)} Finally, we verify that Assumption~\ref{assmp:lipgrad}\ref{assmp:lipgrad3} holds for the perturbed problem corresponding to $\cD^u$. Indeed, as a consequence of the previous steps, we have $\gamma^u \rightarrow \gamma$ and $\alpha^u \rightarrow \alpha$ as $u\rightarrow0$, so the compatibility inequality $\gamma\beta<\alpha$ corresponding to $\mathcal\cD$ implies $\gamma^u \beta < \alpha^u$ for all sufficiently small $u$.

\subsection{Proof of Lemma~\ref{lem:LAN-new}}\label{sec:proof_LAN-new}
Fix $u\in\R^d$. For each $k\in\NN$, it follows immediately from the definitions \eqref{eq:perturbation-of-the-force}, \eqref{eq:P}, and \eqref{eq:Q} that for all $E_{0},\ldots,E_{k-1}\in\cB(\Z)$, the $Q_{k,u}$-measure of the rectangle $E =E_{0}\times\cdots\times E_{k-1}$ is given by
\begin{align*}
	Q_{k,u}(E) &= \int_{E_{0}}\cdots \int_{E_{k-1}}d\cD_{\tilde{x}_{k-1}}^{\smash{u/\sqrt{k}}}(z_{k-1})\cdots d\cD_{\tilde{x}_0}^{\smash{u/\sqrt{k}}}(z_0) \\
	&= \int_{E_{0}}\cdots \int_{E_{k-1}} \,\prod_{i=0}^{k-1}\tfrac{ 1 +  h(u^{\top} g_{\tilde x_{i}}(z_{i}) /\sqrt{k})}{C_{\tilde x_{i}}^{u/\sqrt{k}}} \, d\cD_{\tilde{x}_{k-1}}(z_{k-1})\cdots d\cD_{\tilde{x}_0}(z_0) \\
	&= \int_{E} \,\prod_{i=0}^{k-1}\tfrac{ 1 +  h(u^{\top} g_{\tilde x_{i}}(z_{i}) /\sqrt{k})}{C_{\tilde x_{i}}^{u/\sqrt{k}}} \, dQ_{k,0}.
\end{align*}
Therefore
\begin{equation*}
  \frac{d Q_{k,u}}{dQ_{k,0}} = \prod_{i=0}^{k-1}\tfrac{ 1 +  h(u^{\top} g_{\tilde x_{i}}\!(z_{i}) /\sqrt{k})}{C_{\tilde x_{i}}^{u/\sqrt{k}}}
\end{equation*}
and hence
  \begin{equation}\label{eq:loglike}
    \log \frac{d Q_{k,u}}{dQ_{k,0}} = \sum_{i=0}^{k-1}{\log}\bigg(\!1 + h\bigg(\frac{u^{\top} g_{\tilde x_{i}}(z_{i})}{\sqrt{k}}\bigg)\!\bigg) - \sum_{i = 0}^{k-1}\log C_{\tilde x_{i}}^{\smash{u/\sqrt{k}}}.
  \end{equation}

By Lemma~\ref{lem: uniform C(x,u)}, we have 
  $
  C_{x}^{u} = 1 + r_{x}(u) 
  $
  with $\sup_{x\in\mathcal{X}}|r_x(u)| = o(\|u\|^2)$ as $u\rightarrow0$,
so the first-order approximation $\log(1 + t) = t + o(t)$ as $t\rightarrow 0$ reveals that the last sum in \eqref{eq:loglike} satisfies
  \begin{equation*}
  	  \sum_{i = 0}^{k-1}\log C_{\tilde x_{i}}^{\smash{u/\sqrt{k}}} = \sum_{i = 0}^{k-1} \Big(r_{\tilde x_{i}}\!\big(u/\sqrt{k}\big) + o\big(r_{\tilde x_{i}}\!\big(u/\sqrt{k}\big)\big)\Big) = k\cdot o(k^{-1}) = o(1) \qquad\text{as~}k\to\infty.
  \end{equation*}
Further, since $h(t) = t$ for all $t$ in a neighborhood of zero and $c:=\sup_{x\in\cX,\,z\in\cZ}|u^{\top}g_x(z)|$ is finite by Lemma~\ref{lem:ultra-bounded}, it follows that for all sufficiently large $k\in\NN$, we have
$$h\bigg(\frac{u^{\top} g_{\tilde x_{i}}(z_{i})}{\sqrt{k}}\bigg) = \frac{u^{\top} g_{\tilde x_{i}}(z_{i})}{\sqrt{k}} \in \bigg[{-}\frac{c}{\sqrt{k}},\frac{c}{\sqrt{k}} \bigg] \qquad \text{for all~}i\geq0.$$
 Thus, the second-order approximation $\log(1+t) = t - \tfrac{1}{2}t^2 + o(t^2)$ as $t\to 0$ reveals that the first sum in \eqref{eq:loglike} satisfies 
	\begin{align*}
  	&\sum_{i=0}^{k-1}{\log}\bigg(\!1 + h\bigg(\frac{u^{\top} g_{\tilde x_{i}}(z_{i})}{\sqrt{k}}\bigg)\!\bigg) \\ 
     &\hspace{1in}= u^{\top}\bigg(\frac{1}{\sqrt{k}}\sum_{i=0}^{k-1} g_{\tilde x_{i}}(z_{i})\bigg) - \frac{1}{2}u^{\top}\bigg(\frac{1}{k}\sum_{i=0}^{k-1}g_{\tilde x_{i}}(z_{i})g_{\tilde x_{i}}(z_{i})^{\top}\bigg)u + k\cdot o(k^{-1}) \\
     &\hspace{1in}= u^{\top}Z_{k} - \frac{1}{2} u^{\top} V_k u + o(1) \qquad\text{as~}k\to\infty,
  \end{align*} 
  where $Z_{k}\colon\cZ^k\rightarrow\R^d$ and $V_{k}\colon\cZ^k\rightarrow\R^{d\times d}$ are given by
  \begin{equation*}
  	Z_{k} = \frac{1}{\sqrt{k}}\sum_{i=0}^{k-1} { g_{\tilde x_{i}}(z_{i})}\qquad\text{and}\qquad V_k = \frac{1}{k}\sum_{i=0}^{k-1}g_{\tilde x_{i}}(z_{i})g_{\tilde x_{i}}(z_{i})^{\top}.
  \end{equation*}
  Therefore
  \begin{equation*}
  	\log \frac{d Q_{k,u}}{dQ_{k,0}} = u^{\top}Z_{k} - \frac{1}{2} u^{\top} V_k u + o(1) \qquad\text{as~}k\to\infty.
  \end{equation*}  
Hence, to complete the verification that $\{Q_{k,u} \!\mid\! u\in \R^d\}$ is locally asymptotically normal at zero with precision $\Sigma_g$, it only remains to demonstrate $Z_{k}\overset{\smash{0}}{\leadsto}\mathsf{N}(0, \Sigma_{g})$ and $V_k = \Sigma_g + o_{Q_{k,0}}(1)$. 
 
The assertion $V_k = \Sigma_g + o_{Q_{k,0}}(1)$ is equivalent to $V_k \xlongrightarrow{p}\Sigma_g$ as $k\to\infty$ on the filtered probability space $(\cZ^{\NN}, \cB(\cZ^{\NN}), \FF, \PP)$, where $\mathbb{F} = (\cF_{k})_{k\geq 0}$ is the filtration given by
\begin{equation*}
	\mathcal{F}_0 := \{\emptyset,\cZ^{\NN}\}\quad\text{and}\quad\mathcal{F}_k := \{E\times \cZ^{\NN} \mid E\in\cB(\Z^k)\} \qquad\text{for all~}k\geq 1
\end{equation*} 
and $\PP :=  \bigotimes_{i=0}^{\infty}\cD_{\tilde{x}_{i}}$. We will show more, namely, that almost sure convergence holds:
  \begin{equation}
    \label{eq:cov-limit}
     V_k \xlongrightarrow{\text{a.s.}}\Sigma_g\qquad\text{as~}k\to\infty.
  \end{equation}
  This is a consequence of the martingale strong law of large numbers (Theorem~\ref{thm:slln}). Indeed, for each $i\geq 0$, set
  \begin{align*}
  X_{i+1} &= g_{\tilde{x}_{i}}(z_{i})g_{\tilde{x}_{i}}(z_{i})^{\top} - {\EE}\big[ g_{\tilde{x}_{i}}(z_{i})g_{\tilde{x}_{i}}(z_{i})^{\top} \,|\, \cF_i \big] \\
  &= g_{\tilde{x}_{i}}(z_{i})g_{\tilde{x}_{i}}(z_{i})^{\top} - {\mathop{\EE}_{z_i\sim\cD_{\tilde x_i}}}\!\big[ g_{\tilde{x}_i}(z_{i}) g_{\tilde{x}_i}(z_{i})^\top \big],
  \end{align*}
  thereby defining a martingale difference sequence $X$ in $\R^{d\times d}$ adapted to $\mathbb{F}$; note that we have $\sup_i{\EE}\|X_i\|_\text{F}^2 < \infty$ by Lemma~\ref{lem:ultra-bounded}, so $\sum_{i = 1}^\infty i^{-2}{\EE}\|X_i\|_\text{F}^2 < \infty$ and hence Theorem~\ref{thm:slln} implies
  \begin{equation}\label{eq:app-of-slln}
  	 V_k - \frac{1}{k} \sum_{i=0}^{k-1} {\mathop{\EE}_{z_i\sim\cD_{\tilde x_i}}}\!\big[ g_{\tilde{x}_i}(z_{i}) g_{\tilde{x}_i}(z_{i})^\top \big]  = \frac{1}{k}\sum_{i=1}^{k}X_i \xlongrightarrow{\textrm{a.s.}} 0 \qquad\text{as~}k\to\infty.
  \end{equation}
  On the other hand, we have $\tilde{x}_i \xlongrightarrow{\textrm{a.s.}} x^\star$ as $i\to\infty$ by Definition~\ref{ass:beautiful-algo}, so Lemma~\ref{lem:covariance-convergence} implies
  $$
  {\mathop{\EE}_{z_i\sim\cD_{\tilde x_i}}}\!\big[ g_{\tilde{x}_i}(z_{i}) g_{\tilde{x}_i}(z_{i})^\top\big] \xlongrightarrow{\textrm{a.s.}} \underset{z \sim \cD_{\xs}}{\mathop{\EE}}\!\big[g_{\xs}(z)g_{\xs}(z)^{\top}\big]=\Sigma_g \qquad\text{as~}i\to\infty
  $$
  and hence the arithmetic mean satisfies 
  \begin{equation}\label{eq:cesaro}
  	  \frac{1}{k} \sum_{i=0}^{k-1} {\mathop{\EE}_{z_i\sim\cD_{\tilde x_i}}}\!\big[ g_{\tilde{x}_i}(z_{i}) g_{\tilde{x}_i}(z_{i})^\top \big] \xlongrightarrow{\text{a.s.}} \Sigma_g \qquad\text{as~}k\to\infty.
  \end{equation}
Combining \eqref{eq:app-of-slln} and \eqref{eq:cesaro} gives \eqref{eq:cov-limit}.

 Finally, we establish $Z_{k}\overset{\smash{0}}{\leadsto}\mathsf{N}(0, \Sigma_{g})$ by applying the martingale central limit theorem (Theorem~\ref{thm:clt}). Set $M_0 = 0$ and $M_k = \sum_{i=0}^{k-1}g_{\tilde{x}_i}(z_{i})$ for each $k\geq1$; then $M$ is a square-integrable martingale in $\R^d$ adapted to the filtration $\mathbb{F}$. Indeed, the increments of $M$ are clearly uniformly bounded (Lemma~\ref{lem:ultra-bounded}), $M_k$ is $\cF_k$-measurable, and
  	$$
  	{\EE}\big[M_{k+1} \,|\, \cF_{k} \big] = M_{k} + {\mathop{\EE}_{z_{k} \sim \cD_{\tilde{x}_{k}}}}\!\big[ g_{\tilde{x}_k}(z_{k}) \big] = M_{k}
  	$$
  by the unbiasedness condition of Definition~\ref{def:G_n}. The predictable quadratic variation of $M$ is given by
  $$
  \langle M \rangle_k = \sum_{i=1}^{k}{\EE}\big[(M_{i}-M_{i-1})(M_{i}-M_{i-1})^{\top} \,|\, \cF_{i-1} \big] = \sum_{i=0}^{k-1} {\mathop{\EE}_{z_i\sim\cD_{\tilde x_i}}}\!\big[ g_{\tilde{x}_i}(z_{i}) g_{\tilde{x}_i}(z_{i})^\top \big].
  $$
 Thus, by \eqref{eq:cesaro}, we have
  $$
  k^{-1} \langle M \rangle_k = \frac{1}{k} \sum_{i=0}^{k-1} {\mathop{\EE}_{z_i\sim\cD_{\tilde x_i}}}\!\big[ g_{\tilde{x}_i}(z_{i}) g_{\tilde{x}_i}(z_{i})^\top \big] \xlongrightarrow{\textrm{a.s.}} \Sigma_g \qquad\text{as~}k\to\infty.
  $$
  The assumptions of Theorem~\ref{thm:clt} are therefore fulfilled with $a_k = k$ (note that Lindeberg's condition holds trivially by the uniform boundedness of the increments of $M$). Hence
  $$
Z_k = k^{-1/2}M_k \overset{\smash{0}}{\leadsto} \mathsf{N}(0, \Sigma_{g}).
  $$
This completes the proof.

\subsection{Proof of Lemma~\ref{lem:smooth-kappa}}\label{sec:proof_smooth-kappa}

Let $F\colon\cX\times\R^d\rightarrow\R^d$ be the map given by
$$F(x,u) =  {\mathop{\EE}_{z \sim \cD_{x}^{u}}}\big[G(x,z)\big] = \frac{1}{C_{x}^{u}}\,  {\mathop{\EE}_{z \sim \cD_{x}}}\! \big[\big(1+ h\big(u^{\top} g_{x}(z)\big)\big) G(x, z) \big],$$
where we recall
$C_{x}^{u} = 1+ {\mathop{\EE}_{z \sim \cD_{x}}} [h(u^\top g_{x}(z))]$.
Lemma~\ref{lem:oh-the-smoothness} directly implies that $F$ is $C^1$-smooth. Consider now the family of smooth nonlinear equations
\begin{equation}\label{eq:equi-fam}
	F(x,u)=0
\end{equation}
parameterized by $u\in\mathcal\R^d$.
Note $F(x^\star, 0)=G_{x^\star}(x^\star)=0$ since $\xs\in\interior\cX$. More generally, the equality \eqref{eq:equi-fam} with $(x,u) \in (\rm{int}\,\cX)\times\cU$ holds precisely when $x$ is equal to $x_u^\star$.
We will apply the implicit function theorem to show that \eqref{eq:equi-fam} determines $x_u^{\star}$ as a smooth function of $u$ on a neighborhood of zero. To this end, observe that Lemma~\ref{lem:oh-the-smoothness} reveals 
$$
\nabla_x F(x^\star, 0 ) = \nabla_x \Big( \mathop{\EE}_{z\sim \D_x}\!G(x,z)\Big)\Big|_{x=x^\star} =  W,
$$
which is invertible by Lemma~\ref{lem:posjac}. Consequently, the implicit function theorem yields open neighborhoods $U\subset \cU$ of $0$ and $V\subset\interior\cX$ of $x^\star$ and a $C^1$-smooth map $U \to V$ given by $u\mapsto x_u^\star$ with Jacobian $-W^{-1}\nabla_u F(x^\star, 0)$ at $u=0$. This yields the first-order approximation
\begin{equation}\label{eqn:taylor_opt}
x^\star_u = x^\star - W^{-1}\nabla_u F(x^\star, 0)u + o(\|u\|) \qquad \textrm{as } u\to 0.
\end{equation}
By Lemma~\ref{lem:oh-the-smoothness}, we have
$$
\nabla_u F(x, 0) = {\mathop{\EE}_{z\sim \D_{x}}}\!\big[G(x,z)g_{x}(z)^{\top}\big]
$$
for all $x\in\cX$. In particular, $\nabla_u F(x^{\star}, 0) = \Sigma_{g,G}^{\top}$.
Thus, \eqref{eqn:taylor_opt} asserts
\begin{align*}
x^\star_u &= x^\star - W^{-1}\Sigma_{g,G}^{\top}u + o(\|u\|)\qquad \textrm{as }u\to 0.
\end{align*}
Consequently, for any fixed $u\in\R^d$, we have
$$\sqrt{k}\big(x^\star_{\smash{u/\sqrt{k}}} - x^\star\big) = - W^{-1}\Sigma_{g,G}^{\top}u + \sqrt{k}\cdot o\bigg(\frac{1}{\sqrt{k}}\bigg)  \rightarrow - W^{-1}\Sigma_{g,G}^{\top}u \qquad \text{as~} k\rightarrow\infty.$$
The proof is complete.

\subsection{Proof of Lemma~\ref{lem:oh-the-smoothness}} \label{proof:oh-the-smoothness}
 Recall first that the quantities 
 $$M'_g := \sup_{x\in\cX,\,z\in\cZ} \|\nabla_x g_x(z)\|_\textup{op} \qquad\text{and}\qquad M_g := \sup_{x\in\cX,\,z\in\cZ} \|g_x(z)\|$$
 are finite by Lemma~\ref{lem:ultra-bounded}. The same argument shows that $M'_{T}$ and $M_T$ are finite.

Next, we turn to establishing that the map $H\colon \cX \times \cW \rightarrow \R^n$ given by
$$
H(x, u) = \frac{1}{C_{x}^{u}}\, {\mathop{\EE}_{z \sim \cD_{x}}}\! \big[ T(x, z) \big(1+ h\big(u^{\top} g_{x}(z)\big)\big)\big]
$$
is smooth with Lipschitz Jacobian on the compact set $\mathcal{K}:= \cX \times \cW$. By Lemma~\ref{lem:lip-quo}, it is enough to show that  $(x, u) \mapsto C_{x}^{u}$ and
$$
\hat H (x, u) := {\mathop{\EE}_{z \sim \cD_{x}}}\!\big[ T(x, z) \big(1+ h\big(u^{\top} g_{x}(z)\big)\big)\big]
$$
are smooth with Lipschitz Jacobians on $\mathcal{K}$; in turn, it suffices to establish this fact for $\hat H$ since we can then take $T\equiv 1$ to derive the result for $C_{x}^{u}$.

We reason this via the chain rule. Namely, consider the map $\bar H\colon \cX \times \cX \times \cW \rightarrow \RR^n$ given by
$$
\bar H(x,y, u) = {\mathop{\EE}_{z \sim \cD_{x}}}\!\big[ T(y, z) \big(1+ h\big(u^{\top} g_{y}(z)\big)\big)\big].
$$
Clearly $\hat H = \bar H \circ J$ with $J(x, u) := (x, x, u)$ and therefore the chain rule implies $\nabla \hat H(x, u) = \nabla \bar H(x, x, u) \nabla J(x, u)$ provided $\bar{H}$ is smooth. Thus, it suffices to show that $\bar H$ is smooth with Lipschitz Jacobian. To this end, we demonstrate that the three partial derivatives of $\bar H$ are all Lipschitz with constants depending on $T$ only through $ \beta_T, \bar{\Lambda}_T, \beta'_T$, $M_T$, and $M'_T$.

We begin with the partial derivative of $\bar H$ with respect to $x$. Consider the function $\phi\colon \cK \times \cZ \rightarrow \R^n$ given by
$$
\phi(y,u,z) = T(y, z) \big(1 + h\big(u^{\top} g_{y}(z)\big)\big).
$$
Let us verify that $\phi$ is a test function to which item \ref{ass:extra-smooth} of Assumption~\ref{ass:extra} applies. Clearly $\phi$ is measurable and bounded with $\sup \|\phi\| \leq 2M_T$. Further, for each $z\in\cZ$, it follows readily that the section $\phi(\cdot, z)$ is Lipschitz on $\cK$ with constant
$$
L_{\phi} := 2M'_{T} + M_TL_h\big(\!\diam(\cW)M'_{g} + M_g\big),
$$
where $L_{h}:=\sup|h'|$.
Thus, item \ref{ass:extra-smooth} of Assumption~\ref{ass:extra} implies that the map
    $$
    x \mapsto \mathop{\EE}_{z \sim \D_{x}}{\phi(y,u,z)} = \bar{H}(x,y,u)
    $$
    is smooth on $\cX$ for each $(y,u)\in\mathcal{K}$, and that the map
    $$
    (x, y, u) \mapsto \nabla_{x}\bar{H}(x,y,u)
    $$
    is Lipschitz on $\mathcal{X}\times\mathcal{X}\times\mathcal{W}$ with constant $\vartheta( L_{\phi} + 2M_{T} )$, which depends on $T$ only through $M_T$ and $M'_T$.

Next, we consider the partial derivative of $\bar H$ with respect to $y$. Given $(x, u) \in \cX \times \cW$, the dominated convergence theorem ensures that $\bar H(x, y, u)$ is smooth in $y$ with
\begin{equation}\label{eq:y-deriv-exp}
\nabla_{y}\bar{H}(x,y,u) = {\mathop{\EE}_{z \sim \D_{x}}} \nabla_{y}\phi(y, u, z)
\end{equation}
provided $\|\nabla_{y}\phi(y, u, z)\|_\text{op}$ is dominated by a $\cD_{x}$-integrable random variable independent of $y$. Using the product rule, we have
\begin{equation}\label{eq:y-deriv}
	\nabla_{y}\phi(y, u, z) = \big(\nabla_{y} T(y, z)\big)\big(1 + h\big(u^{\top} g_{y}(z)\big)\big) + h'\big(u^{\top} g_{y}(z)\big)\big(T(y,z)u^{\top}\big)\nabla_{y}g_{y}(z)
	\end{equation}
and hence
\begin{align*}
  \|\nabla_{y}\phi(y, u, z)\|_{\text{op}} &\leq 2\|\nabla_y T(y, z) \|_{\text{op}} + (\sup |h'|) \|T(y, z)\| \|u\| \|\nabla_{y} g_{y}(z)\|_{\text{op}} \\
                              & \leq 2M'_{T} + \diam(\cW)L_{h}M_{T}M'_{g},
\end{align*}
so $\nabla_{y}\phi$ is in fact uniformly bounded. Therefore $\bar H(x, y, u)$ is smooth in $y$ and \eqref{eq:y-deriv-exp} holds. Moreover, it follows from \eqref{eq:y-deriv} that the map
$$
(x, y, u) \mapsto \nabla_{y}\bar{H}(x,y,u)
$$
is Lipschitz on $\cX\times\cX\times\cW$; we will verify this by computing Lipschitz constants separately in $x$, $y$, and $u$. To begin, note that it follows from \eqref{eq:y-deriv} that $z\mapsto \nabla_{y}\phi(y, u, z)$ is Lipschitz on $\cZ$ with constant
$$
a:= 2\beta'_{T} + \diam(\cW)M'_{T}L_{h}\beta_{g} + \diam(\cW)L_h(\beta_{T}M'_{g} + M_{T}\beta'_{g}) + \diam(\cW)^{2}M_{T}M'_{g}L_{h'}\beta_{g},
$$
where $L_{h'}:=\sup|h''|$. Hence \eqref{eq:y-deriv-exp} and Assumption~\ref{assmp:lipdist} imply that $x\mapsto\nabla_{y}\bar{H}(x,y,u)$ is Lipschitz on $\cX$ with constant $\gamma a$, which depends on $T$ only through $\beta_{T}, \beta'_{T}, M_{T}$, and $M'_{T}$. Likewise, it follows from \eqref{eq:y-deriv} that $y\mapsto \nabla_{y}\phi(y, u, z)$ is Lipschitz on $\cX$ with constant
$$
2 \Lambda_{T}(z) + \diam(\cW)M'_{T}L_{h}M'_{g} + \diam(\cW)L_{h}\big(M_{T}\Lambda_{g}(z) + M'_{T}M'_{g}\big) + \diam(\cW)^{2}M_{T}L_{h'}(M'_{T})^2.
$$
Hence \eqref{eq:y-deriv-exp} implies that $y \mapsto \nabla_{y}\bar{H}(x,y,u)$  is Lipschitz on $\cX$ with constant
$$
2\bar{\Lambda}_{T} + \diam(\cW)L_{h}\big(M_{T}\bar{\Lambda}_{g} + 2M'_{T}M'_{g}\big) + \diam(\cW)^2M_{T}L_{h'}(M'_{T})^2.
$$
Similarly, it follows from \eqref{eq:y-deriv} that $u\mapsto \nabla_{y}\phi(y, u, z)$ is Lipschitz on $\mathcal{W}$ with constant
$$
M'_{T}M_{g} L_{h} + M_{T}M'_{g}\big( L_h + \diam(\cW)M_{g}L_{h'} \big),
$$
so \eqref{eq:y-deriv-exp} implies that  $u\mapsto \nabla_{y}\bar{H}(x, y, u)$ is Lipschitz on $\mathcal{W}$ with  the same constant.
We conclude therefore that the map $(x, y, u) \mapsto \nabla_{y}\bar{H}(x,y,u)$ is Lipschitz on $\cX\times\cX\times\cW$ with constant depending on $T$ only through $\beta_T, \bar{\Lambda}_T, \beta'_T$, $M_T$, and $M'_{T}$.

Finally, we consider the partial derivative of $\bar H$ with respect to $u$. Given $(x, y) \in \cX \times \cX$, the dominated convergence theorem ensures that $\bar H(x,y,u)$ is smooth in $u$ with
\begin{equation}\label{eq:u-deriv-exp}
\nabla_{u} \bar H(x, y, u) = {\mathop{\EE}_{z \sim \D_{x}}} \nabla_{u}\phi(y, u, z)
\end{equation}
provided $\|\nabla_{u}\phi(y, u, z)\|_\text{op}$ is dominated by a $\cD_{x}$-integrable random variable independent of $u$. In this case, we have
\begin{equation}\label{eq:u-deriv}
	\nabla_{u}\phi(y, u, z) = h'\big(u^{\top} g_{y}(z)\big)T(y,z)g_{y}(z)^{\top}
	\end{equation}
and hence
\begin{align*}
  \|\nabla_{u}\phi(y, u, z)\|_{\text{op}} &\leq (\sup |h'|) \|T(y, z)\| \|g_{y}(z)\| \leq L_{h}M_{T}M_{g}.
\end{align*}
Therefore $\bar H(x,y,u)$ is smooth in $u$ and \eqref{eq:u-deriv-exp} holds. Moreover, it follows from \eqref{eq:u-deriv} that the map
$$
(x, y, u) \mapsto \nabla_{u}\bar{H}(x,y,u)
$$
is Lipschitz on $\cX\times\cX\times\cW$; as before, we will verify this by computing Lipschitz constants separately in $x$, $y$, and $u$. First, note that it follows from \eqref{eq:u-deriv} that $z\mapsto \nabla_{u}\phi(y, u, z)$ is Lipschitz on $\cZ$ with constant
$$
b:= L_{h}(\beta_{T}M_{g} + M_{T}\beta_{g}) + \diam(\cW)M_{T}M_{g}L_{h'}\beta_{g}.
$$
Hence \eqref{eq:u-deriv-exp} and Assumption~\ref{assmp:lipdist} imply that $x\mapsto\nabla_{u}\bar{H}(x,y,u)$ is Lipschitz on $\cX$ with constant $\gamma b$, which depends on $T$ only through $\beta_{T}$ and $M_{T}$. Likewise, it follows from \eqref{eq:u-deriv} that $y\mapsto \nabla_{u}\phi(y, u, z)$ is Lipschitz on $\cX$ with constant
$$
L_{h}(M'_{T}M_{g} + M_{T}M'_{g}) + \diam(\cW)M_{T}M_{g}L_{h'}M'_{g},
$$
hence so is $y \mapsto \nabla_{u}\bar{H}(x,y,u)$ by \eqref{eq:u-deriv-exp}.
Similarly, it follows from \eqref{eq:u-deriv} that $u\mapsto \nabla_{u}\phi(y, u, z)$ is $L_{h'}M_{T}M_{g}^2$-Lipschitz on $\mathcal{W}$, hence so is $u\mapsto \nabla_{u}\bar{H}(x, y, u)$ by \eqref{eq:u-deriv-exp}. We conclude therefore that the map $(x, y, u) \mapsto \nabla_{u}\bar{H}(x,y,u)$ is Lipschitz on $\cX\times\cX\times\cW$ with constant depending on $T$ only through $\beta_T, M_{T}$, and $M'_{T}$.

The preceding reveals that $\bar H$ and hence $\hat H = \bar H \circ J$ are smooth, with Lipschitz Jacobians with constants depending on $T$ only through $\beta_T, \bar{\Lambda}_T, \beta'_T$, $M_T$, and $M'_{T}$. Taking $T\equiv 1$, we conclude that $(x, u) \mapsto C_{x}^{u}$ is smooth, with Lipschitz Jacobian with constant independent of $T$. Upon observing in the same way as above that $\bar H$ and hence $\hat H$ are Lipschitz with constants depending on $T$ only through $\beta_T, M_{T}$, and $M'_{T}$, it follows from Lemma~\ref{lem:lip-quo} and its proof that $H$ is smooth, with Lipschitz Jacobian with constant depending on $T$ only through $\beta_T, \bar{\Lambda}_T, \beta'_T$, $M_T$, and $M'_{T}$. 

Finally, given any $x\in\cX$, the equalities  
    $$
    \nabla_x H(x, 0 ) = \nabla_x \Big( \mathop{\EE}_{z\sim \D_x}\!T(x,z)\Big)\qquad\text{and}\qquad \nabla_u H(x, 0) = {\mathop{\EE}_{z\sim \D_{x}}}\!\big[T(x,z)g_{x}(z)^{\top}\big]
    $$
follow from straightforward computations (using the quotient rule, dominated convergence theorem, and chain and product rules). This completes the proof.

\section{Underlying Probability Space}
In this appendix, we formally construct the probability space where decision-dependent dynamics take place. 
The following lemma shows that Assumption~\ref{assmp:lipdist} implies $\{\mathcal{D}_{x}\}_{x\in\cX}$ is a Markov kernel from $\cX$ to $\cZ$, i.e., for each $E\in\cB(\cZ)$, the function $\cX\rightarrow[0,1]$ given by $x\mapsto\cD_{x}(E)$ is measurable.

\begin{lemma}[Markov kernel]\label{lem:markov}
Let $\cZ$ be a nonempty Polish metric space. Then, for any bounded measurable function $\varphi\colon\cZ\rightarrow \R$, the function $P_1(\cZ)\rightarrow\R$ given by $\mu\mapsto\int\varphi\,d\mu$ is measurable. In particular, for any measurable space $\cX$ and any measurable map $x\mapsto\cD_x$ from $\cX$ to $P_1(\cZ)$, it follows that $\{\mathcal{D}_{x}\}_{x\in\cX}$ is a Markov kernel from $\cX$ to $\cZ$.
\end{lemma}
\begin{proof}
	Let $\mathcal{M}_b$ denote the set of all bounded measurable functions $\cZ\rightarrow\R$. For each $\varphi\in\mathcal{M}_b$, let $I_{\varphi}\colon P_1(\cZ)\rightarrow\R$ be the function given by $I_{\varphi}(\mu) = \int \varphi \, d\mu$. Now consider the set
	\begin{equation*}
		\mathcal{C} = \big\{\varphi\in\mathcal{M}_b \mid I_{\varphi}\text{~is measurable}\big\}.
	\end{equation*}
	To demonstrate $\mathcal{C} = \mathcal{M}_b$, it suffices by the functional monotone class theorem \citep[e.g., see][Exercise 11.7]{Kechris1995} to show that $\mathcal{C}$ possesses the following two properties:
	\begin{enumerate}[label=(\roman*)]
    \item\label{cond:continuous_func} Every bounded continuous function $\cZ\rightarrow\R$ is contained in $\mathcal{C}$.
    \item\label{cond:pointwise_closure} If $(\varphi_n)$ is a uniformly bounded sequence in $\mathcal{C}$ with pointwise limit $\varphi\colon\cZ\rightarrow\R$ (i.e., $\sup_{n,z} |\varphi_n(z)| < \infty$ and $\lim_{n\to\infty}\varphi_n(z) = \varphi(z)$ for all $z\in\cZ$), then $\varphi\in\mathcal{C}$.
  \end{enumerate}
  To this end, note first that \ref{cond:continuous_func} holds because $W_1$-convergence in $P_1(\cZ)$ implies weak convergence \citep[e.g., see][Proposition 7.1.5]{Ambrosio2008}; indeed, if $\varphi\colon\cZ\rightarrow\R$ is bounded and continuous, then for any sequence $(\mu_n)$ in $P_1(\cZ)$ such that $W_1(\mu_n,\mu)\rightarrow 0$ for some $\mu\in P_1(\cZ)$, we have $I_{\varphi}(\mu_n)\rightarrow I_{\varphi}(\mu)$, so $I_\varphi$ is continuous and hence measurable. On the other hand, \ref{cond:pointwise_closure} follows from the dominated convergence theorem: if $(\varphi_n)$ is a uniformly bounded sequence in $\mathcal{C}$ with pointwise limit $\varphi\colon\cZ\rightarrow\R$, then $\varphi\in\mathcal{M}_b$ and
  \begin{equation*}
  	I_{\varphi}(\mu) = \int \lim_{n\to\infty}\varphi_n(z)\,d\mu(z) = \lim_{n\to\infty} \int \varphi_n(z)\,d\mu(z) = \lim_{n\to\infty} I_{\varphi_n}(\mu)
  \end{equation*}
	for all $\mu\in P_1(\cZ)$, so $I_\varphi$ is measurable as the pointwise limit of the sequence of measurable functions $(I_{\varphi_n})$. Hence $\mathcal{C} = \mathcal{M}_b$, i.e., $I_\varphi$ is measurable for every bounded measurable function $\varphi\colon\cZ\rightarrow \R$; in particular, the last claim of the lemma follows by taking $\varphi$ to be the indicator function $\mathbf{1}_{E}$ of any measurable set $E\in\cB(\cX)$.
\end{proof}

We will require the existence of the probability measure $\bigotimes_{i=0}^{\infty}\cD_{x_{i}}$ on the countable product space $\Z^{\NN}$ with marginals given by recursive application of the Markov kernel $\{\mathcal{D}_{x}\}_{x\in\cX}$ from $\cX$ to $\cZ$ along a sequence of measurable maps $x_t\colon\cZ^t\rightarrow\cX$ (corresponding to iterates of a decision-dependent algorithm). This is provided by the following theorem, which may be viewed as a special case of either the Kolmogorov extension theorem \citep[see][Appendix~D]{Bass2011} or the Ionescu–Tulcea extension theorem \citep[see][Theroem~14.35]{Klenke2020}.

\begin{theorem}[Ionescu-Tulcea]\label{thm:kolmogorov}
	Let $\cX$ be a measurable space, $\cZ$ be a nonempty Polish metric space, $\{\mathcal{D}_{x}\}_{x\in\cX}$ be a Markov kernel from $\cX$ to $\cZ$, and $x_t\colon\cZ^t\rightarrow\cX$ be a sequence of measurable maps (with $x_0\in\cX$). For each $t\geq1$, let $\PP_t = \bigotimes_{i=0}^{t-1}\cD_{x_{i}}$ be the probability measure on $\cZ^t$ defined recursively by setting $\PP_1 = \cD_{x_0}$ and
\begin{equation*}
	\PP_{t+1}(A \times E) = \int_{A} \cD_{x_t}(E) \,d\PP_{t} \qquad\text{for all~}A\in\cB(\Z^t)\text{~and~}E\in\cB(\cZ),
\end{equation*}
and let $\pi_{t}\colon\cZ^{\NN}\rightarrow\Z^t$ denote the projection from the countable product space $\Z^{\NN}$ onto the first $t$ coordinates. Then there exists a unique probability measure $\PP =  \bigotimes_{i=0}^{\infty}\cD_{x_{i}}$ on $\Z^{\NN}$  satisfying $(\pi_t)_{\#}\PP = \PP_t$ for all $t\geq1$, that is,
\begin{equation*}
	\PP(A\times \cZ^{\NN}) = \PP_t(A) \qquad\text{for all~}A\in\cB(\Z^t)\text{~and~}t\geq 1.
\end{equation*}
Thus, for every $t\geq 0$ and every measurable function $\varphi\colon\Z^{t+1}\rightarrow\overline{\R}$ that is nonnegative or $\PP_{t+1}$-integrable, we have
\begin{equation*}
	 \EE[\varphi\circ\pi_{t+1}] = \int_{\Z^{t+1}} \varphi\,d\PP_{t+1} =\int_{\Z}\cdots \int_{\Z} \varphi(z_0,\ldots,z_{t}) \,d\cD_{x_{t}}(z_{t})\cdots d\cD_{x_0}(z_0)
\end{equation*}
and
\begin{equation*}
\EE[\varphi\circ\pi_{t+1}\,|\,\cF_{t}] = \int_{\Z} \varphi(z_0,\ldots,z_{t}) \,d\cD_{x_{t}}(z_{t}) = \underset{z_t\sim\D_{x_t}}{\Expect}[\varphi(z_0,\ldots,z_{t})],
\end{equation*}
where $\cF_t = \{A\times \cZ^{\NN} \mid A\in\cB(\Z^t)\}$ denotes the $\sigma$-algebra generated by $\pi_t$ (with $\cF_0 = \{\emptyset, \cZ^{\NN}\}$).
\end{theorem}

\section{Supplementary Results}

In this appendix, we record some supplementary results fundamental to our analysis. First, we record suitably general versions of the Strong Law of Large Numbers \citep[see][Exercise 5.3.35]{dembo_probthry} and the Central Limit Theorem \citep[see][Corollary 2.1.10]{duflo97} for square-integrable martingales.

\begin{theorem}[Martingale Strong Law of Large Numbers]\label{thm:slln}
Let $X$ be a square-integrable martingale difference sequence in $\R^n$ adapted to a filtration $(\cF_k)$ and $(a_k)$ be a sequence of positive constants such that $a_{k} \uparrow \infty$ as $k \to \infty$. Then on the event $\big\{\sum_{i = 1}^\infty a_{i}^{\smash{-2}}\EE[\|X_i\|^2\,|\,\cF_{i-1}] < \infty\big\}$, we have $a_{k}^{\smash{-1}}\sum_{i=1}^{k}X_i\to 0$ almost surely as $k\to\infty$. In particular, if $\sum_{i = 1}^\infty a_{i}^{\smash{-2}}\EE\!\|X_i\|^2 < \infty$, then $a_{k}^{\smash{-1}}\sum_{i=1}^{k}X_i\to 0$ almost surely as $k\to\infty$.
\end{theorem}

\begin{theorem}[Martingale Central Limit Theorem]\label{thm:clt}
  Let $M$ be a square-integrable martingale in $\R^n$ adapted to a filtration $(\cF_k)$, and let $\langle M \rangle$ denote the predictable quadratic variation of $M\!$:
  $$
  \langle M \rangle_k = \sum_{i=1}^{k}{\EE}\big[(M_{i}-M_{i-1})(M_{i}-M_{i-1})^{\top} \,|\, \cF_{i-1} \big] \qquad\text{for all~} k\geq 1.
  $$
  Let $(a_k)$ be a sequence of positive constants such that $a_{k} \uparrow \infty$ as $k \to \infty$. Suppose that the following two assumptions hold.
  \begin{enumerate}[label=(\roman*)]
  	\item (\textbf{Asymptotic covariance}) There is a deterministic positive semidefinite matrix $\Sigma$ satisfying $$a_k^{-1}\langle M \rangle_k \xlongrightarrow{p} \Sigma \qquad\text{as~}k\to\infty.$$
  	\item (\textbf{Lindeberg's condition}) For all $\varepsilon>0$,
  		$$
  		a_{k}^{\smash{-1}}\sum_{i=1}^{k}{\EE}\big[ \|M_i - M_{i-1}\|^2\mathbf{1}_{\smash{\{\|M_i - M_{i-1}\|\geq \varepsilon a_{k}^{\smash{1/2}}\}}} \,|\, \cF_{i-1}\big]\xlongrightarrow{p} 0 \qquad\text{as~}k\to\infty.
  		$$
  \end{enumerate}
  Then
  $$
  a_{k}^{\smash{-1}}M_{k} \xlongrightarrow{\textup{a.s.}} 0 \quad \text{and} \quad a_{k}^{\smash{-1/2}}M_{k} \leadsto \mathsf{N}(0, \Sigma) \qquad\text{as~}k\to\infty.
  $$

\end{theorem}

The following lemma is used multiple times in our arguments to compute limits of covariance matrices.

\begin{lemma}[Asymptotic covariance]\label{lem:covariance-convergence}
  Let $x_{t}\in\X$ be a sequence in some set $\mathcal{X}\subset\R^d$ converging to some point $\xs\in \mathcal{X}$, and let $\mu_t\in P_1(\mathcal{Z})$ be a sequence of probability measures on a nonempty Polish space $\mathcal{Z}$ converging to some measure $\mu^{\star}\in P_1(\mathcal{Z})$ in the Wasserstein-1 metric. Suppose that $g \colon \cX \times \cZ \rightarrow \RR^{n}$ is a measurable map satisfying the following two conditions.
  \begin{enumerate}[label=(\roman*)]
    \item\label{covlem:cond1} (\textbf{Asymptotic uniform integrability}) For every $\delta>0$, there exists a constant $N_{\delta}\geq0$ such that
    \begin{align*}
    	\limsup_{t\to \infty}\mathop{\EE}_{z \sim \mu_t} \!\big[\|g(\xs, z)\|^2 \mathbf{1}_{\smash{\{\|g(\xs, z)\| \geq N_\delta\}}}\big] &\leq \delta, \\
    	\mathop{\EE}_{z \sim \mu^{\star}} \!\big[\|g(\xs, z)\|^2 \mathbf{1}_{\smash{\{\|g(\xs, z)\| \geq N_\delta\}}}\big] &\leq \delta.
    \end{align*}

    \item\label{covlem:cond2}  (\textbf{Lipschitz continuity}) There exist a neighborhood $\mathcal{V}$ of $x^{\star}$, a measurable function $L\colon\cZ\rightarrow[0,\infty)$, and constants $ \beta, \bar{L}\geq0$ such that for every $z\in\Z$, the section $g(\cdot,z)$ is $L(z)$-Lipschitz on $\mathcal{V}$ with $\limsup_{t\to \infty}{\EE}_{z\sim \mu_t}[L(z)^2] \leq \bar{L}^2$, and the section $g(x^\star,\cdot)$ is $\beta$-Lipschitz on $\cZ$.

\end{enumerate}
  Then $$\lim_{t\to\infty}\,\mathop{\EE}_{z\sim \mu_t} \!\big[ g(x_{t},z)  g(x_{t},z)^{\top} \big] = \mathop{\EE}_{z\sim \mu^{\star}}  \!\big[g(\xs,z) g(\xs,z)^{\top}\big].$$
\end{lemma}
\begin{proof}
  For notational convenience, set $g_x(z) = g(x, z)$ and
  $$
  \Sigma = {\mathop{\EE}_{z\sim \mu^{\star}}} \big[ g_{\xs}(z) g_{\xs}(z)^{\top} \big].
  $$
  For any $\delta>0$, the decomposition
  \begin{align*}\begin{split}
      \underset{z\sim\mu_t}{\EE} \big[\|g_{\xs}(z)\|^{2}\big] & =
      \underset{z\sim\mu_t}{\Expect}\big[\|g_{\xs}(z)\|^{2}\mathbf{1}_{\smash\{\|g_{\xs}(z)\| < N_{\delta}\}}\big] + \underset{z\sim\mu_t}{\Expect}\big[\|g_{\xs}(z)\|^{2}\mathbf{1}_{\smash{\{\|g_{\xs}(z)\|\geq N_{\delta}\}}}\big]
    \end{split}
  \end{align*}
  holds for all $t$, so condition~\ref{covlem:cond1} implies
  \begin{equation}\label{eq:unifM1}
  \limsup_{t\to \infty}\underset{z\sim\mu_t}{\EE} \big[\|g_{\xs}(z)\|^{2}\big]  \leq N_{\delta}^{2} + \delta.
  \end{equation}
  On the other hand, for all $t$, we also have the decomposition
	\begin{align}\label{eq:cov_decomp}
		\begin{split}
			\underset{z\sim \mu_t}{\Expect}\big[g_{x_{t}}(z)g_{x_{t}}(z)^\top\big]
			&= \underset{z\sim\mu_t}{\Expect}\big[g_{\xs}(z)g_{\xs}(z)^{\top}\big] + \underset{z\sim\mu_t}{\Expect}\big[g_{\xs}(z)\big(g_{x_t}(z)-g_{\xs}(z)\big)^\top\big] \\
			&~~~~~+ \underset{z\sim\mu_t}{\Expect}\big[\big(g_{x_t}(z) - g_{\xs}(z)\big)g_{x_t}(z)^\top\big].
		\end{split}
	\end{align}
  The last two summands in \eqref{eq:cov_decomp} tend to zero as $t\to\infty$. Indeed, since $x_t\to\xs$ as $t\to\infty$, we have $x_t\in\cV$ for all but finitely many $t$ and so we may apply condition \ref{covlem:cond2} together with H\"older's inequality and \eqref{eq:unifM1} to conclude
	\begin{align*}
    \left\|\underset{z\sim\mu_t}{\Expect}\big[g_{\xs}(z)\big(g_{x_t}(z) - g_{\xs}(z)\big)^\top\big]\right\|_\text{op}
    &\leq \underset{z\sim\mu_t}{\Expect}\big[\|g_{\xs}(z)\|\cdot\|g_{x_t}(z) - g_{\xs}(z)\|\big] \\
    &\leq \|x_t - \xs\|\sqrt{\underset{z\sim\mu_t}{\Expect}\big[\|g_{\xs}(z)\|^2\big]\cdot\underset{z\sim\mu_t}{\Expect}\big[L(z)^2\big]}\\	
    &\to 0 \qquad\text{as~}t\to\infty
    \end{align*}
	and
	\begin{align*}
    &\left\|\underset{z\sim\mu_t}{\Expect}\big[\big(g_{x_t}(z) - g_{\xs}(z)\big)g_{x_t}(z)^\top\big]\right\|_\text{op} \\
    &\hspace{1in}\leq \underset{z\sim\mu_t}{\Expect}\big[\|g_{x_t}(z) - g_{\xs}(z)\|\cdot\|g_{x_t}(z)\|\big] \\
    &\hspace{1in}\leq \|x_t - \xs\|\sqrt{\underset{z\sim\mu_t}{\Expect}\big[L(z)^2\big]\cdot\underset{z\sim\mu_t}{\Expect}\big[\|g_{x_t}(z)\|^2\big]} \\
    &\hspace{1in}\leq \|x_t - \xs\|\sqrt{2\underset{z\sim\mu_t}{\Expect}\big[L(z)^2\big]\Big(\underset{z\sim\mu_t}{\Expect}\big[\|g_{\xs}(z)\|^2\big] + \underset{z\sim\mu_t}{\Expect}\big[\|g_{x_t}(z) - g_{\xs}(z)\|^2\big]\Big)} \\
    &\hspace{1in}\leq \|x_t - \xs\|\sqrt{2\underset{z\sim\mu_t}{\Expect}\big[L(z)^2\big]\Big(\underset{z\sim\mu_t}{\Expect}\big[\|g_{\xs}(z)\|^2\big] + \|x_t - \xs\|^2\cdot\underset{z\sim\mu_t}{\Expect}\big[L(z)^2\big]\Big)}\\
    &\hspace{1in}\to 0 \qquad\text{as~}t\to\infty.
	\end{align*}

  To complete the proof, it now suffices by \eqref{eq:cov_decomp} to show ${\Expect}_{z\sim\mu_t}[g_{\xs}(z)g_{\xs}(z)^{\top}] \rightarrow \Sigma$ as $t \to \infty$. To this end, define for each $q\in\RR$ the step-like function $\varphi_{q}\colon \RR \rightarrow \RR$ by setting
  $$
  \varphi_{q}(x) = \begin{cases} 1 & \text{if }x \leq q, \\ - x + q + 1 & \text{if } q \leq x \leq q+1, \\ 0 &\text{if }q + 1\leq x.\end{cases}
  $$
  Let $\delta > 0$ be arbitrary. Then for any given $t$, we have the decomposition
  \begin{align*}
    \underset{z\sim\mu_t}{\Expect}\big[&g_{\xs}(z)g_{\xs}(z)^{\top}\big] - \Sigma \\ & = \underset{z\sim\mu_t}{\Expect}\big[g_{\xs}(z)g_{\xs}(z)^\top\big] -	\underset{z\sim\mu^{\star}}{\Expect}\big[g_{\xs}(z)g_{\xs}(z)^\top\big] \\ & = \underbrace{\underset{z\sim\mu_t}{\Expect}\big[\big(1-\varphi_{N_{\delta}}(\|g_{\xs}(z)\|)\big)g_{\xs}(z)g_{\xs}(z)^\top\big] -\underset{z\sim\mu^{\star}}{\Expect}\big[\big(1-\varphi_{N_{\delta}}(\|g_{\xs}(z)\|)\big)g_{\xs}(z)g_{\xs}(z)^\top\big]}_{A_t} \\
        &  \hspace{1.1cm} + \underbrace{\underset{z\sim\mu_t}{\Expect}\big[\varphi_{N_{\delta}}(\|g_{\xs}(z)\|)g_{\xs}(z)g_{\xs}(z)^\top\big] -\underset{z\sim\mu^{\star}}{\Expect}\big[\varphi_{N_{\delta}}(\|g_{\xs}(z)\|)g_{\xs}(z)g_{\xs}(z)^\top\big]}_{B_t}\!.
  \end{align*}
  By the triangle inequality, $\| A_t \|_\text{op}$ is bounded above by
  \begin{align*}
    &\Big\| \underset{ z\sim\mu_t}{\Expect}\big[\big(1-\varphi_{N_{\delta}}(\|g_{\xs}(z)\|)\big)g_{\xs}(z)g_{\xs}(z)^\top\big]\Big\|_\text{op} \!+ \Big\|\underset{z\sim\mu^{\star}}{\Expect}\big[\big(1-\varphi_{N_{\delta}}(\|g_{\xs}(z)\|)\big)g_{\xs}(z)g_{\xs}(z)^\top\big]\Big\|_\text{op} \\
        & \hspace{1.1cm} \leq  \underset{z\sim\mu_t}{\Expect}\big[\|g_{\xs}(z)\|^{2}\mathbf{1}_{\smash{\{\|g_{\xs}(z)\| \geq N_{\delta}\}}}\big] + \underset{z\sim\mu^{\star}}{\Expect}\big[\|g_{\xs}(z)\|^{2}\mathbf{1}_{\smash{\{\|g_{\xs}(z)\| \geq N_{\delta}\}}}\big],
  \end{align*}
  so
  $$\limsup_{t\to\infty} \| A_t \|_\text{op}\leq 2\delta.$$
  In order to bound $B_t$, consider the map $\Phi\colon \R^{n} \rightarrow \R^{n \times n}$ given by $\Phi(w) = \varphi_{N_{\delta}}(\|w\|) ww^{\top}$, set $\phi = \Phi \circ g_{\xs}$, and note
  $$ B_t = \underset{z\sim\mu_t}{\EE}{\big[\phi(z)\big]} - \underset{z\sim\mu^{\star}}{\EE}\big[\phi(z)\big].$$
  Clearly $\Phi$ is Lipschitz continuous on any compact set and zero outside of the ball of radius $N_{\delta} + 1$ centered at the origin. Therefore $\Phi$ is globally Lipschitz. Since $g_{\xs}$ is $\beta$-Lipschitz on $\cZ$ by condition \ref{covlem:cond2}, 
  we conclude that $\phi$ is Lipschitz on $\cZ$ with a constant $C$ that depends only on $N_{\delta}$ and $\beta$.
Consequently, 
  \begin{align*}
    \|B_t\|_{\text{op}} &= \Big\| \underset{z\sim\mu_t}{\Expect}\big[\phi(z)\big] - \underset{z\sim\mu^{\star}}{\Expect}\big[\phi(z)\big] \Big\|_{\text{op}} \\
    &=\sup_{\|u\|,\|v\|\leq1}\Big\{\underset{z\sim\mu_t}{\Expect}\big[\langle \phi(z)u, v \rangle\big] - \underset{z\sim\mu^{\star}}{\Expect}\big[ \langle \phi(z)u, v \rangle \big]\Big\}\\
    &\leq C\cdot W_1(\mu_t,\mu^{\star})\to 0 \qquad\text{as~}t\to\infty,
  \end{align*}
	where the inequality follows from the $C$-Lipschitz continuity of the function  $z\mapsto \langle \phi(z)u, v \rangle$.
  Hence $$\limsup_{t\rightarrow \infty}\Big\| \underset{z\sim\mu_t}{\Expect} \big[g_{\xs}(z)g_{\xs}(z)^{\top}\big] - \Sigma \Big\|_\text{op} \leq \limsup_{t\rightarrow \infty} \big(\|A_t\|_\text{op} + \|B_t\|_\text{op} \big) \leq 2\delta.$$
  Since $\delta>0$ is arbitrary, we deduce ${\Expect}_{z\sim\mu_t}[g_{\xs}(z)g_{\xs}(z)^{\top}] \rightarrow \Sigma$ as $t\to\infty$.
\end{proof}

Finally, we record two basic lemmas about products and quotients of Lipschitz functions.
\begin{lemma}\label{lem:lip-prod}
  Let $\mathcal{K}$ be a metric space and suppose that $f\colon \mathcal{K} \rightarrow \R^{n\times q}$ and $g\colon \mathcal{K} \rightarrow \R^{q\times m}$ are bounded and Lipschitz. Then the product $f g \colon \mathcal{K} \rightarrow \R^{n\times m}$ is Lipschitz.
\end{lemma}
\begin{proof}
  Let $L_{f}$ and $ L_{g}$ be the Lipschitz constants of $f$ and $g$ with respect to the operator norm $\|\cdot\|$. Then for all $x, y \in \mathcal{K}$, we have
  \begin{align*}
    \|f(x) g(x) - f(y)g(y)\| &\leq \| f(x)(g(x) - g(y))\| + \|(f(x) - f(y))g(y)\| \\
                             &\leq \sup_{z \in \mathcal{K}} \| f(z)\| \cdot \|g(x) - g(y)\| +  \| f(x) - f(y)\|\cdot \sup_{z\in \mathcal{K}} \|g(z)\| \\
                             &\leq \left(L_{g} \cdot\sup_{z \in \mathcal{K}} \|f(z)\| + L_{f}\cdot \sup_{z\in \mathcal{K}}\|g(z)\| \right) \cdot d_{\mathcal{K}}(x,y).
  \end{align*}
  Since $f$ and $g$ are bounded, this demonstrates that $fg$ is Lipschitz.
\end{proof}

\begin{lemma}\label{lem:lip-quo}
  Let $\mathcal{K} \subset \RR^{m}$ be a compact set and suppose that $f \colon \mathcal{K} \rightarrow \RR^{n}$ and $g \colon \mathcal{K} \rightarrow\RR\setminus\{0\}$ are $C^{1}$-smooth with Lipschitz Jacobians. Then $f/g$ is $C^{1}$-smooth with Lipschitz Jacobian.
\end{lemma}
\begin{proof}
  Since $f$ and $g$ are $C^1$-smooth, it follows immediately from the quotient rule that $f/g$ is $C^1$-smooth with Jacobian given by
  \begin{equation}\label{eq:quotient-grad}
  	 \nabla(f/g)= (1/g)(\nabla f) - (f/g^2)(\nabla g)^{\top}.
  \end{equation}
  By assumption, $\nabla f$ and $\nabla g$ are Lipschitz, and they are bounded by the compactness of $\cK$. Further, the functions $1/g$ and $f/g^2$ are $C^1$-smooth, so they are locally Lipschitz by the mean value theorem; hence $1/g$ and $f/g^2$ are Lipschitz and bounded by the compactness of $\mathcal{K}$. Thus, \eqref{eq:quotient-grad} and Lemma~\ref{lem:lip-prod} show that $\nabla(f/g)$ is the difference of two Lipschitz maps. Therefore $\nabla(f/g)$ is Lipschitz.
\end{proof}

\bibliography{biblio.bib}

\end{document}